\documentclass[a4paper,12pt]{amsart} 

\title{Height moduli on cyclotomic stacks and \\ counting elliptic curves over function fields}
\date{}
\author{Dori Bejleri, Jun--Yong Park and Matthew Satriano}

\usepackage{verbatim}
\usepackage[charter]{mathdesign}
\SetMathAlphabet{\mathcal}{normal}{OMS}{cmsy}{m}{n}

\usepackage{amsmath,graphicx,amsthm,amsfonts,verbatim,mathtools,thmtools,fullpage,tikz-cd,url,tensor}
\usepackage[a4paper,margin=3.3cm, headsep=15pt, headheight=2cm]{geometry}
\usepackage[normalem]{ulem}
\usepackage[shortlabels]{enumitem}
\usepackage[all,cmtip]{xy}
\usepackage{hyperref}
\usepackage{subfiles} 

\newtheorem{thm}{Theorem}[section]

\newtheorem{lem}[thm]{Lemma}
\newtheorem{cor}[thm]{Corollary}
\newtheorem{prop}[thm]{Proposition}
\theoremstyle{definition}
\newtheorem{defn}[thm]{Definition}

\newtheorem{conj}[thm]{Conjecture}

\newtheorem{rmk}[thm]{Remark}

\newtheorem{exmp}[thm]{Example}

\newcommand{\bL}{\mathbb{L}}
\newcommand{\bG}{\mathbb{G}}
\newcommand{\bP}{\mathbb{P}}
\newcommand{\bA}{\mathbb{A}}
\newcommand{\bV}{\mathbb{V}}
\newcommand{\bZ}{\mathbb{Z}}

\newcommand{\cX}{\mathcal{X}}
\newcommand{\cU}{\mathcal{U}}
\newcommand{\cY}{\mathcal{Y}}
\newcommand{\cF}{\mathcal{F}}
\newcommand{\cO}{\mathcal{O}}
\newcommand{\cC}{\mathcal{C}}
\newcommand{\cD}{\mathcal{D}}
\newcommand{\cQ}{\mathcal{Q}}
\newcommand{\cL}{\mathcal{L}}
\newcommand{\cW}{\mathcal{W}}
\newcommand{\cH}{\mathcal{H}}
\newcommand{\cR}{\mathcal{R}}
\newcommand{\cM}{\mathcal{M}}
\newcommand{\cE}{\mathcal{E}}
\newcommand{\cS}{\mathcal{S}}
\newcommand{\cP}{\mathcal{P}}
\newcommand{\cI}{\mathcal{I}}

\newcommand{\nuj}{\nu^{(j)}}
\newcommand{\muj}{\mu^{(j)}}

\DeclareMathOperator{\height}{ht}
\DeclareMathOperator{\Bs}{Bs}
\DeclareMathOperator{\Emb}{Emb}
\DeclareMathOperator{\Conf}{Conf}

\newcommand{\Bmu}{\mbox{$\raisebox{-0.59ex}
		{$l$}\hspace{-0.18em}\mu\hspace{-0.88em}\raisebox{-0.98ex}{\scalebox{2}
			{$\color{white}.$}}\hspace{-0.416em}\raisebox{+0.88ex}
		{$\color{white}.$}\hspace{0.46em}$}{}}

\newcommand{\iso}{\cong}

\newcommand{\Me}{\overline{\mathcal{M}}_{1,1}}

\newcommand{\blambda}{\bar{\lambda}}

\newcommand{\I}{{\mathop{\rm I}}}
\newcommand{\II}{{\mathop{\rm II}}}
\newcommand{\III}{{\mathop{\rm III}}}
\newcommand{\IV}{{\mathop{\rm IV}}}

\newcommand{\Pcv}{\mathcal{P}(\vec{\lambda})}

\newcommand{\Cc}{\mathcal{C}}

\newcommand{\Ec}{\mathcal{E}}

\newcommand{\Oc}{\mathcal{O}}

\newcommand{\Pc}{\mathcal{P}}

\newcommand{\Nc}{\mathcal{N}}

\newcommand{\Zc}{\mathcal{Z}}

\newcommand{\Ic}{\mathcal{I}}

\newcommand{\lambdavec}{{\vec{\lambda}}}

\newcommand{\Yc}{\mathcal{Y}}

\newcommand{\Z}{\mathbb{Z}}
\newcommand{\Q}{\mathbb{Q}}
\newcommand{\N}{\mathbb{N}}

\newcommand{\Pb}{\mathbb{P}}
\newcommand{\Ab}{\mathbb{A}}
\newcommand{\A}{\mathbb{A}}

\newcommand{\Fb}{\mathbb{F}}
\newcommand{\Gb}{\mathbb{G}}
\newcommand{\Lb}{\mathbb{L}}

\DeclareMathOperator{\Spec}{Spec}

\DeclareMathOperator{\Pic}{Pic}

\DeclareMathOperator{\PGL}{PGL}

\DeclareMathOperator{\GL}{GL}
\DeclareMathOperator{\SL}{SL}

\DeclareMathOperator{\Sym}{Sym}

\DeclareMathOperator{\Hom}{Hom}

\DeclareMathOperator{\Aut}{Aut}

\newcommand{\Ppic}{\mathcal{P}ic}

\renewcommand{\setminus}{\smallsetminus}

\begin{document}
	
	\maketitle
	
	\vspace{-3ex}
	\centerline{\footnotesize{\textit{In memory of Yuri Manin, 1937--2023}}}
	
	\begin{abstract}
		For proper stacks, unlike schemes, there is a distinction between rational and integral points. Moreover, rational points have extra automorphism groups. We show that these distinctions exactly account for the lower order main terms appearing in precise counts of elliptic curves over function fields, answering a question of Venkatesh in this case. More generally, using the theory of twisted stable maps and the stacky height functions recently introduced by Ellenberg, Zureick-Brown, and the third author, we construct finite type moduli spaces which parametrize rational points of fixed height on a large class of stacks, so-called cyclotomic stacks. The main tool is a correspondence between rational points, twisted maps and weighted linear series. Along the way, we obtain the Northcott property as well as a generalization of Tate's algorithm for cyclotomic stacks, and compute the exact motives of these moduli spaces for weighted projective stacks.

	\end{abstract}

	\medskip
	\medskip

	\setcounter{tocdepth}{1}
	
	\tableofcontents
	\vspace{-3ex}
	
	\section{Introduction}
	\label{sec:intro}
	
	For many asymptotic counting problems in analytic number theory, it is of interest to determine the main leading term as well as the lower order terms. While the lower order terms are interesting theoretically, having good control of these lower order terms is often also important for numerical computations. A well-known example is the enumeration of cubic number fields ordered by height of discriminant. The number of cubic fields of height $\le B$ for certain constants $a,b > 0$ is 
	
	\vspace{-.13in}
	\begin{equation*}\label{eq:cubic-fields}
		aB - bB^{5/6} + o(B^{5/6})
	\end{equation*}
	The existence of the precise second main term of order $B^{5/6}$ was conjectured by Roberts in \cite{Roberts} and also implicitly in an earlier paper of Datskovsky--Wright in \cite{DW}. This was proven by the remarkable works of Bhargava--Shankar--Tsimerman in \cite{BST} and also by Taniguchi--Thorne in \cite{TT} along the lines of establishing the Davenport–Heilbronn theorems \cite{DH, DH2} on cubic fields and 3-torsion in class groups of quadratic fields. In general, however, one has very little understanding of the origin of these lower order main terms (c.f.  \cite[\S 2.6]{VE}). In this regard, Venkatesh in \cite[Prob. 5]{GGW} asks the following question:
	
	\begin{center}
		\textit{What is the topological meaning of secondary terms \\ appearing in asymptotic counts in number theory?}
	\end{center}
	
	In the case of elliptic curves over a function field $K=k(C)$ of a curve $C$ over $k=\Fb_q$, it is well known that the counting of semi-stable elliptic curves, that is, the integral points of the moduli stack $\Me$, are governed by the geometry of the moduli space of morphisms $C \to \Me$ (see e.g. \cite{HP, BPS}). We generalize this story to give a moduli-theoretic description of the rational points that \textit{do not extend} to integral points. By understanding the arithmetic geometry of these moduli spaces, we show that integral points of $\Me$ account for the main leading term in the enumeration of all elliptic curves over a function field, while rational points of $\Me$ that do not extend to integral points and have \textit{extra automorphisms} account for the lower order main terms. These rational points are necessarily concentrated at the special $j$-invariant and require us to consider the relationship between \textit{weighted and unweighted} point counts over finite fields.
	
	\medskip
	
	Specifically, we establish the following sharp enumeration of elliptic curves over a global function field $K = \Fb_q(t)$ with precise lower order main terms. This is achieved by considering the totality of rational points on $\Me$ over $K$. Recall that the height of the discriminant of an elliptic curve $E$ over $K$ is given by $ht(\Delta):=q^{\deg \Delta} = q^{12n}$ for some integer $n$ (also called the Faltings height of $E$).
	
	\begin{thm}\label{thm:ell_curve_min_count} 
		Let $n \in \mathbb{Z}_{\ge 0}$ and $\text{char} (\Fb_q) > 3$. The counting function $\Nc^w(\Fb_q(t),~B)$ (resp. $\Nc(\Fb_q(t), ~B)$), which gives the weighted count (resp. unweighted count) of the number of minimal elliptic curves over $\Pb^1_{\Fb_q}$ ordered by the multiplicative height of the discriminant $ht(\Delta) = q^{12n} \le B$, is given by the following expressions where $a_{q},b_{q},c_{q},d_{q} \ge 0$ are explicit rational functions of $q$ (see Theorem \ref{thm:ell_curve_min_count_body}):
		\begin{align*}
			\Nc^w(\Fb_q(t),~B) &= ~ a_{q}B^{5/6} - B^{1/6} \\
			\Nc(\Fb_q(t),~B) &= ~ 2a_{q} B^{5/6} + 4b_{q} B^{1/2} + 2c_{q} B^{1/3} - 2 B^{1/6} + d_{q}
		\end{align*}
		\noindent The $B^{1/2}$ and $B^{1/3}$ terms come from counting $\mu_6$- and $\mu_4$-twist families while the $B^{1/6}$ term comes from counting $\mu_2$-twist families concentrated at $j = \infty$. 
	\end{thm}
	
	Notably, the weighted and unweighted counts of semi-stable elliptic curves over $\Pb^1_{\Fb_q}$, i.e. of the space $\Hom_n(\mathbb{P}^1, \Me)$, only contribute to the \emph{leading term} $B^{5/6}$ with \emph{constant lower order term} (c.f. \cite{dJ2,HP,BPS}).
	
	\medskip
	
	To prove Theorem \ref{thm:ell_curve_min_count}, we construct \textit{height moduli spaces} $\cM_{n,K}(\Me)$ (resp. $\cH^\Gamma_{n,K}(\Me)$) whose $\Fb_q$-points correspond to rational points of $\Me$ of height $n$ over the function field $K$ (resp. with prescribed local conditions $\Gamma$ corresponding to Tate's algorithm, see Theorem~\ref{thm:corr_46})). We count the number of $\Fb_q$-points, and more generally compute the exact motives of height moduli and their inertia stacks in the Grothendieck ring of stacks $K_0(\mathrm{Stck}_{\Fb_q})$ of these moduli spaces to answer the question of Venkatesh in this case. The height moduli framework we introduce in this paper applies to a large class of moduli stacks as well as to higher genus function fields and other perfect residue fields $k$ (where point counting over finite fields is replaced with various topological or motivic invariants).

	\medskip
	
	Fix a perfect field $k$ and a smooth proper geometrically connected curve $C/k$ with function field $K = k(C)$. For $X$ a projective scheme with ample line bundle $L$, the height of a point $P \in X(K)$ is simply the degree, measured with respect to $L$, of the unique extension $C \to X$. Thus for any integer $n$, we have a moduli space
	$$
	\mathrm{Hom}_n(C, X)
	$$
	of degree $n$ maps to $X$ whose points are identified with $K$-points of height $n$. Our goal is to generalize this construction when $X$ is replaced by an algebraic stack $\cX$. In this paper we focus on the class of cyclotomic stacks, whose properties best resemble those of projective varieties. Introduced by Abramovich and Hassett in \cite{AH}, these are algebraic stacks whose stabilizers are $\mu_r$ and equipped with a uniformizing line bundle $\cL$ (see Definition \ref{def:cyc_uni}) which plays the role of an ample line bundle. This includes familiar examples such as weighted projective stacks $\Pcv \coloneqq \Pc(\lambda_1,\dots,\lambda_N)$ with $\cL=\cO(1)$ and fine modular curves. Particularly, $\Me \iso \Pc(4,6)$ over $\bZ\left[\frac{1}{6}\right]$ with $\cL$ the Hodge line bundle by the short Weierstrass equation $y^2 = x^3 + a_4x + a_6$, where $\zeta \cdot a_i=\zeta^i a_i$ for $\zeta \in \Gb_m$ and $i=4,6$.
	
	\medskip
	
	One of the subtleties for proper algebraic stacks is that not every rational point $C \dashrightarrow \cX$ extends to an integral point $C \to \cX$. Instead, a rational point always extends to a \emph{twisted map} $\cC \to \cX$ from a stacky curve $\cC$ with coarse moduli space $C$ (see Section \ref{sec:mod-sp-twismaps}). This observation was one of the key insights of \cite{ESZB} for defining the \emph{stacky height} of a rational point of $P \in \cX(K)$ (see Section \ref{sec:heights:stacks}).  As in the case of schemes, twisted maps form a moduli space $\cH^\Gamma_d(\cX, \cL)$ where 
	$$
	\Gamma = \{(r_1, a_1), \ldots, (r_s, a_s)\}
	$$
	encodes the stacky structure of $\cC$ as well as the stabilizer action on the fiber of $\cL|_{\cC}$ (Definition \ref{def:local-conditions}) and $d$ is the degree. On the other hand, given $(\cX, \cL)$, the aforementioned stacky height $\height_{\cL}(P)$ decomposes as $\height_{\cL}^{st}(P)+\sum_v \delta_{\cL,v}(P)$, where $\delta_{\cL,v}(P)$ are local contributions coming from finitely many places $v$ of $K$ and $\height_{\cL}^{st}(P)$ is the so-called \emph{stable height}, which is stable under base change. We show (Theorem \ref{thm:tuning-stacks-and-minimal-linear-series}) that we can identify the terms in the stacky height with the data $d, \Gamma$ for a twisted map:
	$$
	d = \height_{\cL}^{st}(P), \ \ \ \delta_{\cL,v_i}(P) = \frac{a_i}{r_i}. 
	$$
	Our Main Theorem is that there is a height moduli space $\cM_n(\cX, \cL)$ parametrizing all rational points on general proper polarized cyclotomic stacks of stacky height $n$ and that the spaces of twisted maps yield a stratification of $\cM_n(\cX,\cL)$ corresponding to fixing the local contributions to the stacky height. The fact that $\cM_n(\cX, \cL)$ is of finite type is a geometric incarnation of the Northcott property. 
	
	\begin{thm}[Theorems \ref{thm:moduli-of-twisted-maps-aov} \& \ref{thm:height:moduli_body}]\label{thm:height:moduli} Let $(\cX, \cL)$ be a proper polarized cyclotomic stack over a perfect field $k$. Fix a smooth projective curve $C/k$ with function field $K = k(C)$ and $n,d \in \mathbb{Q}_{\ge 0}$.
		
		\begin{enumerate} 
			\item There exists a separated Deligne--Mumford stack $\cM_{n,C}(\cX,\cL)$ of finite type over $k$ with a quasi-projective coarse space and a canonical bijection of $k$-points
			
			$$
			\cM_{n,C}(\cX, \cL)(k) = \left\{ P \in \cX(K) \mid \height_{\cL}(P) = n \right \}.
			$$
			\item There is a finite locally closed stratification 
			$$
			\bigsqcup_{\Gamma,d} 
			\cH^\Gamma_{d,C}(\cX, \cL)/S_\Gamma\to \cM_{n,C}(\cX, \cL)
			$$
			where $\cH^{\Gamma}_{d,C}$ are moduli spaces of twisted maps and the union runs over all possible admissible local conditions
			$$\Gamma = \left(\{r_1, a_1\}, \ldots, \{r_s, a_s\}\right)$$ and degrees $d$ for a twisted map to $(\cX, \cL)$ satisfying
			$$
			n = d + \sum_{i = 1}^s \frac{a_i}{r_i} 
			$$
			and $S_\Gamma$ is a subgroup of the symmetric group on $s$ letters that permutes the stacky points of the twisted map.
			
			\item Under the bijection in part (1), each $k$-point of $\cH_{d,C}^\Gamma(\cX, \cL)/S_\Gamma$ corresponds to a $K$-point $P$ with the stable height and local contributions given by
			$$
			\height^{st}_{\cL}(P) = d \quad \quad \quad \left\{\delta_i = \frac{a_i}{r_i}\right\}_{i = 1}^s. 
			$$
		\end{enumerate}
	\end{thm}
	
	We prove Theorem \ref{thm:height:moduli} by reducing to the case where $\cX$ is a weighted projective stack $\Pcv \coloneqq \Pc(\lambda_0,\dots,\lambda_N)$ and $\cL=\cO(1)$. In this case, we obtain even more precise information in Theorem \ref{thm:main-thm-of-sec4}: we construct $\cM_{n,C}(\Pcv, \cO(1))$ as a moduli space of $\lambdavec$-weighted linear series $(L, s_0, \ldots, s_N)$ on the curve $C$.
	
	\begin{defn} (Definitions \ref{defn:wls} \& \ref{def:mimal-wted-lin-ser})
		A \emph{$\lambdavec$-weighted linear series} on $C$ is a tuple $(L, s_0, \ldots, s_N)$ where $s_i \in H^0(C, L^{\otimes \lambda_i})$. The tuple is \emph{minimal} if for all $x \in C(k^{sep})$, there exists $j$ such that $\nu_x(s_j) < \lambda_j$ where $\nu_x$ is the order of vanishing at $x$. 
	\end{defn}
	
	We show in Theorem \ref{thm:tuning-stacks-and-minimal-linear-series} that there is a correspondence between the twisting conditions $(r,a)$ at a point $x \in C$ and the order of vanishing of the sections $s_i$ in a weighted linear series. This correspondence induces a bijection between twisted maps and minimal weighted linear series where the twisting conditions on the map at $x \in C$ can be computed explicitly from the vanishing orders $\nu_x(s_j)$. Notably, this generalizes Tate's algorithm for computing the Kodaira fiber type of a minimal elliptic surface from the order of vanishing of the Weierstrass data (see below and Section \ref{sec:Mod_Wss_Surf}). Given any rational point $P : C \dashrightarrow \Pcv$, there exists a unique minimal weighted linear series $(L^{min}, s_0, \ldots, s_N)$ on $C$ which realizes $P$ such that
	$$
	\height_{\cO(1)}(P) = \deg L^{min}. 
	$$
	Significantly, this generalizes the existence of global minimal Weierstrass models for elliptic curves over $k(C)$ and the fact that the stacky height (i.e. the Faltings height) is the na\"ive height in this case. 
	
	To prove Theorem \ref{thm:height:moduli} for weighted projective stacks, we construct a moduli space $\cW_{n,C}^{min}(\lambdavec)$ of minimal $\lambdavec$-weighted linear series on $C$ of degree $n$ and stratify it into locally closed substacks $\cW_{n,C}^\gamma(\lambdavec)$ by imposing vanishing conditions $\gamma$ on the sections $s_i$ (Proposition \ref{prop:stack:min} \& Theorem \ref{thm:main-thm-of-sec4}). Finally, we prove that the correspondence between twisted maps and minimal weighted linear series holds in families to identify $\cW_{n,C}^{min}(\lambdavec)$ and $\cW_{n,C}^\gamma(\lambdavec)$ with $\cM_{n,C}$ and the twisted maps strata $\cH^\Gamma_{d,C}/S_\Gamma$ respectively after matching up the corresponding vanishing and twisting conditions $\gamma$ and $\Gamma$ (Section \ref{subsec:modwps}).
	
	\medskip

	With $\cW_{n,C}^{min}(\lambdavec)$ in hand, we move to the question of computing its number of points, or more generally the class in the Grothendieck ring of stacks. Fix weights $\lambdavec = (\lambda_0, \ldots, \lambda_N)$ and let $|\lambdavec| \colon = \sum_{i = 0}^N \lambda_i$. Suppose for simplicity that $k$ contains all $\mathrm{lcm} = \mathrm{lcm}(\lambda_0, \ldots, \lambda_N)$ roots of unity (see Section \ref{sec:zeta} for the general case). 
	
	\begin{thm}[Theorem \ref{thm:ratl}]\label{thm:motive_generating_series} For $k, \lambdavec$ as above and $C = \bP^1_k$, consider $\cW_n^{min}$ and its inertia stack $\cI \cW_n^{min}$. We have the following formulas over the Grothendieck ring of stacks $K_0(\mathrm{Stck}_k)$. 
		
		\begin{enumerate}[(a)]
			\item 
			$$
			\sum_{n \geq 0} \{\cW_n^{min}\}t^n = \frac{ 1-\Lb t }{1-\Lb^{|\vec{\lambda}|}t} \left( \{\Pb^N\} + \Lb^{N+1}  \{\Pb^{|\vec{\lambda}|-N-2} \} t \right) 
			$$
			\item 
			$$
			\sum_{n \geq 0} \{\cI \cW_n^{min}\}t^n = \sum_{g \in \mu_{\mathrm{lcm}}(k)} \frac{ 1-\Lb t }{1-\Lb^{|\vec{\lambda_{g}}|}t} \left( \{\Pb^{{N_g}}\} + \Lb^{{N_g}+1}  \{\Pb^{|\vec{\lambda_{g}}| - {N_g} -2} \} t \right)
			$$
		\end{enumerate}
		where $g$ runs over the $\mathrm{lcm}$ roots of unity and $\lambdavec_g$ is a subset of $\lambdavec$ of size $N_g + 1$ depending explicitly on the order of $g$ (see Section \ref{sec:zeta}). 
	\end{thm}
	
	Here we note that the unweighted count of rational points (i.e. isomorphism classes) is given by the unweighted point count of $\cW_n^{min}$ which is equal (c.f. Theorem \ref{thm:ptcounts}) to the weighted point count (i.e. motive) of the inertia stack $\cI \cW_n^{min}$. The above explicit formulas lead immediately to Theorem \ref{thm:ell_curve_min_count}.
	
	\medskip
	
	The above theorem is an example of \emph{motivic stabilization} for the moduli spaces $\cW_n^{min}$ in the sense of of \cite{VW} and this should hold in any genus. We will explore this in the future. Similarly, the spaces $\cH_{d,C}^\Gamma(\cX, \cL)$ are twisted analogs of mapping stacks. They lie over certain configuration spaces $\mathrm{Conf}_\Gamma(C)$ equipped with an action of a product of symmetric groups $S_\Gamma$ and one can consider representation stability for these spaces as in \cite{CEF}. 
	
	\begin{conj}
		The cohomology of the spaces $\cH^\Gamma_{d,C}(\cX, \cL)$ equipped with the action of $S_\Gamma$ exhibits representation stability as $d \to \infty$ (see e.g. \cite{CEF}). 
	\end{conj}

	Specializing to the case of $\Me \cong \cP(4,6)$ with $\cL$ the Hodge line bundle and $P \in \Me(K)$ a rational point corresponding to an elliptic curve $E/K$. A weighted linear series on $\bP^1$ of height $n$ consists of Weierstrass coefficients $a_4 \in H^0(\bP^1, \cO(4n))$ and $a_6 \in H^0(\bP^1, \cO(6n))$ and the orders of vanishing at a point can be encoded in a vector $\gamma = (\nu_x(a_4), \nu_x(a_6))$ which corresponds to a certain twisting data $\Gamma = (r,a)$ by Theorem \ref{thm:tuning-stacks-and-minimal-linear-series}. This naturally extends the classical Tate's algorithm between the vanishing conditions $\gamma$ and the Kodaira fiber types $\Theta$ of the minimal model of $E$.  The spaces $\cW_{n,\bP^1}^{\gamma}$ and $\cH_{d,\bP^1}^{\Gamma}$ can be identified with moduli of certain canonical models of elliptic surfaces with a specified fiber of additive bad reduction and the isomorphism between the two via Tate's algorithm can be understood in the context of the minimal model program (Section \ref{sec:Mod_Wss_Surf}). This is summarized in the table below. 
	
	\begin{thm}\label{thm:corr_46} If $\mathrm{char}(K) \neq 2,3$. Then the local twisting condition $(r,a)$ and the order of vanishing of $j$ at $j = \infty$ determine the Kodaira fiber type of the relative minimal model, and $(r,a)$ is in turn determined by $m = \min\{3\nu(a_4), 2\nu(a_6)\}$. 
		\begin{table}[ht!]
			\label{TableOfTateCorr}
			\centering
			\begin{tabular}{|c|c|c|}
				\hline
				$\gamma : \left( \nu(a_4), ~ \nu(a_6) \right)$ & Reduction type with  $j \in \overline{M}_{1,1}$ & $\Gamma : \left( r,a \right)$ \\ \hline
				$(\ge 1,1)$ & $\II$ with $j=0$ &  $(6,1)$ \\ \hline
				$(1, \ge 2)$ & $\III$ with $j=1728$ &  $(4,1)$ \\ \hline
				$(\ge 2,2)$ & $\IV$ with $j=0$ &  $(3,1)$ \\ \hline
				$(2, 3)$ & $\I_{k > 0}^*$ with $j=\infty$ & $(2,1)$ \\
				$~$ & $\I_0^*$ with $j\neq 0, 1728$ & $~$ \\ \hline
				$(\ge 3,3)$ & $\I_0^*$ with $j=0$ &  $(2,1)$ \\ \hline
				$(2, \ge 4)$ & $\I_0^*$ with $j=1728$ &  $(2,1)$ \\ \hline
				$(\ge 3,4)$ & $\IV^*$ with $j=0$ &  $(3,2)$ \\ \hline
				$(3, \ge 5)$ & $\III^*$ with $j=1728$ &  $(4,3)$ \\ \hline
				$(\ge 4,5)$ & $\II^*$ with $j=0$ &  $(6,5)$ \\ \hline
			\end{tabular}
		\end{table}        
	\end{thm} 
	
	Finally we compute the motives of $\cW^{\gamma}_{n, \bP^1}$ for any single fixed vanishing condition $\gamma = (\nu_x(a_4), \nu_x(a_6))$. Following \cite{FW, HP}, we stratify $\cW^{\gamma}_{n, \bP^1}$ into spaces of monic polynomials of given degrees and with fixed vanishing conditions and compute the motives of these spaces. The resulting formulas for $\cW^{\gamma}_{n, \bP^1}$ are recorded in Theorem \ref{thm:motive_Rat} below. Via Tate's algorithm, we get counts of elliptic surfaces over $\bP^1$ with a single fiber of specified additive reduction. 
	
	\begin{thm}[Theorem \ref{thm:ell_curve_count_2}]\label{thm:count_kodaira_intro} Let $n \in \mathbb{Z}_{+}$ and $\mathrm{char}(\Fb_q) \neq 2,3$. Fix an additive Kodaira fiber type $\Theta$. The number of minimal elliptic surfaces over $\bP^1$ with exactly one singular fiber of type $\Theta$ and at worst multiplicative reduction elsewhere, ordered by the multiplicitve height of the discriminamt $ht(\Delta) = q^{12n} \le B$ is given by 
		$$
		\Nc(\Fb_q(t), \Theta, B) = 2 e_{\Theta,q} (B^{5/6} - 1)
		$$
		where $e_{\Theta,q}$ is an explicit rational function of $q$ depending only on the fiber type $\Theta$. 
	\end{thm}

	\newpage
	
	\begin{thm}\label{thm:motive_Rat}
		If $\mathrm{char}(K) \neq 2,3$, then motives $\{\cW_{n, \bP^1}^{\gamma}\} \in K_0(\mathrm{Stck}_{K})$ of moduli stacks of elliptic surfaces with a specified Kodaira fiber are
		\begin{table}[ht!]
			\label{TableOfMotive}
			\centering
			\begin{tabular}{|c|c|}
				\hline 
				Reduction type with  $j \in \overline{M}_{1,1}$ & Motivic class $\{\cW_{n, \bP^1}^{\gamma}\} \in K_0(\mathrm{Stck}_{K})$ \\ \hline
				$\II$ with $j=0$ & $ \Lb^{10n} - \Lb^{10n-2}$ \\ \hline
				$\III$ with $j=1728$ & $ \Lb^{10n-1} - \Lb^{10n-3}$ \\ \hline
				$\IV$ with $j=0$ & $ \Lb^{10n-2} - \Lb^{10n-4}$ \\ \hline
				$\I_{k > 0}^*$ with $j=\infty$ &  $ \Lb^{10n-3} - \Lb^{10n-4} - \Lb^{10n-5} + \Lb^{10n-6}$ \\
				$\I_0^*$ with $j\neq 0, 1728$ & $~$ \\ \hline
				$\I_0^*$ with $j=0, 1728$ & $ \Lb^{10n-4} - \Lb^{10n-6}$ \\ \hline
				$\IV^*$ with $j=0$ & $ \Lb^{10n-5} - \Lb^{10n-7}$ \\ \hline
				$\III^*$ with $j=1728$ & $ \Lb^{10n-6} - \Lb^{10n-8}$ \\ \hline
				$\II^*$ with $j=0$ & $ \Lb^{10n-7} - \Lb^{10n-9}$ \\ \hline
			\end{tabular}
		\end{table}
	\end{thm}

	\subsection{Relation to other work} There has been a lot of recent activity in counting points on weighted projective stacks \cite{BGS, SST} with height zeta functions \cite{Darda, Tristan3}, computing asymptotics counts of elliptic curves with additive reduction \cite{CS, CJ, Tristan2, Tristan}, and stacky approaches to heights \cite{DY}. We intend to explore the relationship between our method and theirs in these papers in the future. 
	
	\subsection{Outline of the paper} In Section \ref{sec:heights:twisted}, we discuss heights on cyclotomic stacks from the point of view of twisted maps and construct the moduli space $\cH^\Gamma_d(\cX, \cL)$. In Section \ref{sec:min-lin-ser-univ-tuning}, we show the bijection between twisted maps and minimal weighted linear series. In Section \ref{sec:mod-sp-ratpts}, we construct the moduli spaces of minimal weighted linear series. In Section \ref{sec:moduli}, we prove Theorem \ref{thm:height:moduli}. In Section \ref{sec:ex}, we study the moduli space in the case of weighted projective stacks. In Section \ref{sec:Mod_Wss_Surf}, we state Tate's algorithm via twisting data and interpret $\cH^\Gamma_{d, C}(\Me)$ as moduli of elliptic surfaces with specified Kodaira fibers. In Section \ref{sec:motive_ptcts}, we compute the classes of some height moduli in the Grothendieck ring of stacks and prove Theorems \ref{thm:motive_generating_series} and \ref{thm:motive_Rat}. In Section \ref{sec:enumeration}, we enumerate elliptic curves over $k(t)$ and prove Theorems \ref{thm:ell_curve_min_count} and \ref{thm:count_kodaira_intro}.

	
	\section{Heights on cyclotomic stacks and twisted maps}\label{sec:heights:twisted}
	
	We review the definition of heights on stacks from \cite{ESZB} and specialize it to the class of cyclotomic stacks. We assume that $\cX$ is normal throughout. 
	
	\subsection{Heights on stacks}\label{sec:heights:stacks}
	
	\par In \cite{ESZB}, the authors introduced height functions on stacks and used them to give point-counting conjectures generalizing the Batyrev--Manin \cite{BM, FMT} and Malle Conjectures, see \cite[Conj. 4.14]{ESZB}. When the stack is a scheme, these stacky heights recover the usual Weil heights. On the other hand, when the stack is $\mathcal{B} G$ and the rational point $\mathcal{B} G(K)$ corresponds to a Galois $G$-extension $L/K$, the heights recover the discriminant of the extension.
	
	Unlike the case of schemes, the Weil height machine fails for stacks. One can see this by considering the universal $n$-torsion line bundle $\mathcal{L}$ on $\mathcal{B} \mu_n$; if the Weil height machine were to hold, then $n$ times the height of $\mathcal{L}$ would be trivial. Moreover, one cannot define heights via embeddings into projective space since all stacks that admit such embeddings are necessarily schemes. Instead the definition of heights on stacks is given by extending the rational point to a stacky curve, called a tuning stack.
	
	\begin{defn}[{\cite[Def. 2.1]{ESZB}}]
		\label{def:tuning stack}
		Let $C$ be a smooth proper curve over $K$ and let $p\colon\cX\to C$ be a proper map from a normal Artin stack $\cX$ with finite diagonal. Let $x\in\cX(K)$ be a rational point. A \emph{tuning stack} for $x$ is a diagram
		\[
		\xymatrix{
			\Spec(K)\ar[r]\ar[dr]\ar@/^1.5pc/[rr]^-{x} & \cC\ar[d]_-{\pi}\ar[r]^-{\overline{x}} & \cX\ar[dl]^-{p}\\
			& C &
		}
		\]
		where $\cC$ is a normal Artin stack with finite diagonal and $\pi$ is a birational coarse space map.
		
		A morphism of tuning stacks $(\cC',\pi',\overline{x}')\to(\cC,\pi,\overline{x})$ is a map $f\colon\cC'\to\cC$ such that $\pi\circ f=\pi'$ and $\overline{x}\circ f=\overline{x}'$. A tuning stack is said to be \emph{universal} if it is terminal among all tuning stacks.
	\end{defn}
	
	\begin{rmk}\label{rmk:universal-tuning-stack-rep}
		In \cite[Cor. 2.6]{ESZB}, it is shown that a tuning stack $(\cC,\pi,\overline{x})$ is universal if and only if $\overline{x}$ is representable.
	\end{rmk}
	
	Heights of a rational point on a stack are then given as follows.
	
	\begin{defn}[{\cite[Def. 2.11]{ESZB}}]
		\label{def:stacky-ht}
		With hypotheses as in Definition \ref{def:tuning stack}, if $\mathcal{V}$ is a vector bundle on $\cX$ and $x\in\cX(K)$, the \emph{height of $x$ with respect to $\mathcal{V}$} is defined as
		\[
		\height_{\mathcal{V}}(x) \coloneqq -\deg(\pi_*\overline{x}^*\mathcal{V}^\vee)
		\]
		for any choice of tuning stack $(\cC,\pi,\overline{x})$.
	\end{defn}
	
	Notice that the height $\height_{\mathcal{V}}(x)$ is defined via pullback then pushforward before taking degree. One may alternatively consider the height function obtained by pullback and then taking degree directly. This is known as the stable height:

	\begin{defn}[{\cite[Def. 2.12]{ESZB}}]
		\label{def:stacky-st-ht}
		With hypotheses as in Definition \ref{def:tuning stack}, if $\mathcal{V}$ is a vector bundle on $\cX$ and $x\in\cX(K)$, the \emph{stable height of $x$ with respect to $\mathcal{V}$} is defined as
		\[
		\height^{\mathrm{st}}_{\mathcal{V}}(x) \coloneqq -\deg_{\cC}\overline{x}^*\mathcal{V}^\vee
		\]
		for any choice of tuning stack $(\cC,\pi,\overline{x})$.
	\end{defn}
	
	\begin{rmk}
		The height and stable height functions are shown to be independent of the choice of tuning stack in \cite[Prop. 2.13]{ESZB}.
	\end{rmk}
	
	When $X$ is a scheme, we can take $\cC = C$ and so stable height and height are the same. More generally, height agrees with stable height whenever the vector bundle $\mathcal{V}$ is pulled back from a vector bundle on a scheme. 
	
	\begin{rmk}
		The stable height $\height^{\mathrm{st}}_{\mathcal{V}}(x)$ is stable under base change as opposed to the height $\height_{\mathcal{V}}(x)$ (c.f. \cite[Prop. 2.14]{ESZB}).
	\end{rmk}
	
	Later, we compute the stable height as the coarse map degree plus the local contributions from the ramified base changes in terms of $r,a$ (see Definition \ref{def:rmin-for-any-lin-ser})

	\subsection{Cyclotomic stacks and weighted projective stacks}\label{sec:cyclotomic} 
	
	Here we review the basic definitions of cyclotomic stacks following \cite[Sec. 2]{AH}. This is a class of stacks whose properties best resemble those of projective varieties. The central example is that of a \emph{weighted projective stack}. 
	
	\begin{defn}\label{def:wtproj} 
		Let $\lambdavec = (\lambda_0, \ldots, \lambda_N) \in \mathbb{Z}_{\geq 1}^{N+1}$ be a vector of $N+1$ positive integers. Consider the affine space $U_\lambdavec = \Ab_{x_0, \dotsc, x_N}^{N+1}$ endowed with the action of $\Gb_m$ with weights $\lambdavec$, i.e.~an element $\zeta \in \Gb_m$ acts by
		\begin{equation}
			\zeta \cdot (x_0, \ldots, x_N) = (\zeta^{\lambda_0} x_0, \ldots, \zeta^{\lambda_N} x_N)\,.
		\end{equation}
		The $N$-dimensional weighted projective stack $\Pcv$ is then defined as the quotient stack
		\[
		\Pcv = \left[(U_{\lambdavec} \setminus \{0\})/\Gb_m \right]\,.
		\]
	\end{defn}
	
	\begin{rmk}\label{rmk:wtprojtame} 
		When we wish to emphasize a base field $k$ of definition for $\Pcv$, we use the notation $\Pc_{k}(\vec\lambda)$. The stack $\Pcv$ is a smooth and proper tame Artin stack. It is Deligne--Mumford if and only if all weights $\lambda_i$ are prime to the characteristic. For example $\Pc(1,p)$ is not Deligne--Mumford in characteristic $p$ since it has a point with automorphism group $\mu_p$ which is not formally unramified. When $\Pcv$ is Deligne--Mumford, it is an orbifold if and only if $\gcd(\lambda_0, \dotsc, \lambda_N)=1$. More generally, the natural map
		$$
		\Pc(d\lambda_0, \dotsc, d\lambda_N) \to \Pc(\lambda_0, \ldots, \lambda_N)
		$$
		is a $\mu_d$-gerbe.
	\end{rmk}

	The natural morphism $U_\lambdavec \setminus 0 \to \Pcv$ is the total space of the \emph{tautological line bundle} $\mathcal{O}_{\Pcv}(-1)$ on $\Pcv$. As in the classical case, we denote by $\mathcal{O}_{\Pcv}(1)$ the dual of this line bundle. 
	
	\begin{defn}\label{defn:wls}
		A \emph{$\lambdavec-$weighted linear series} on a scheme $B$ is the data of a line bundle $L$ and sections
		$$
		s_i : \cO_B \to L^{\lambda_i}. 
		$$
		where $\cO_B$ is the structure sheaf of $B$. The \emph{set theoretic base locus} of $(L,s_0, \ldots, s_N)$ is the reduced closed subscheme $Z \subset B$ of points $b \in B$ where $s_j(b) = 0$ for all $j = 0, \ldots, N$. A point $b \in Z$ will be called an \emph{indeterminacy} of the weighted linear series. 
	\end{defn}
	
	\begin{prop}\cite[Lem. 2.1.3]{AH} The stack $\Pcv$ with universal line bundle $\cO_{\Pcv}(1)$ is equivalent to the stack of $\lambdavec-$weighted linear series with empty base locus. 
	\end{prop}
	
	\begin{rmk}
		The open embedding $\Pcv \subset [U_{\lambdavec}/\bG_m]$ corresponds to the inclusion of basepoint-free linear series into the stack of all $\lambdavec$-weighted linear series. 
	\end{rmk}
	
	A weighted projective stack is an example of the following. 
	
	\begin{defn}\label{def:cyc_uni}
		A separated algebraic stack $\cX$ of finite type over a field $k$ is \emph{cyclotomic} if the stabilizer $\Aut(\bar{x})$ of each geometric point $\bar{x} : \Spec \bar{k} \to \cX$ is isomorphic to $\mu_r$ for some $r$. A \emph{uniformizing line bundle} on $\cX$ is a line bundle $\cL$ such that for each geometric point $\bar{x} : \Spec \bar{k} \to \cX$, the natural map
		$$
		\Aut(\bar{x}) \to \Aut(\cL|_{\bar{x}})
		$$
		is injective. 
	\end{defn}
	
	We denote the coarse moduli space of $\cX$ by $\pi : \cX \to X$. By \cite[Prop. 2.3.10]{AH}, the condition that $\cL$ is uniformizing if and only if the map $\cX \to \mathcal{B} \bG_m$ classyfing $\cL$ is representable. Moreover, by \cite[Lem. 2.3.7]{AH}, there exists an $M$ such that $\cL^{\otimes M} \cong \pi^*L$ for some line bundle $L$ on $X$. 
	
	\begin{defn}
		A uniformizing line bundle $\cL$ is a \emph{polarizing line bundle} if $\cL^{\otimes M} \cong \pi^*L$ for some $M$ where $L$ is an ample line bundle on $X$. We say that the pair $(\cX, \cL)$ is a \emph{polarized cyclotomic stack}. 
	\end{defn}
	
	\begin{exmp}
		If $\cX \subset \Pcv$ is a locally closed substack, then the pullback $\mathcal{O}_{\Pcv}(1)|_{\cX}$ is polarizing. More generally, the same is true if $\cX \to \Pcv$ is representable quasi-finite and $\cX$ is Noetherian. 
	\end{exmp}
	
	\begin{prop}\cite[Prop. 2.4.2 \& 2.4.3, Corollary 2.4.4]{AH}\label{prop:polarizing:embedding} Let $(\cX,\cL)$ be a polarized cyclotomic stack and suppose that $\cX$ is proper. Then there exists a weighted projective stack $\Pcv$ and a closed embedding $\cX \subset \Pcv$ such that $$\cL \cong \cO_{\Pcv}(1)|_{\cX}.$$ 
	\end{prop}
	
	A key observation we will exploit in the sequel is that Proposition \ref{prop:polarizing:embedding} reduces questions about the height $\height_{\cL}$ on a polarized cyclotomic stack $(\cX,\cL)$ to the case of $\height_{\cO(1)}$ on $\Pcv$. Compare this to the fact that the Weil height with respect to a very ample line bundle on a projective variety is determined by the na\"ive height on $\mathbb{P}^N$. 
	
	\subsection{Twisted maps to cyclotomic stacks}
	\label{sec:mod-sp-twismaps}
	
	In Section \ref{sec:min-lin-ser-univ-tuning}, we will see that the universal tuning stack of a cyclotomic stack is always a \emph{twisted curve}. In this section we review some background on twisted curves and construct the moduli of twisted curves on cyclotomic stacks. 
	
	\begin{defn}\label{def:spcurve} A stacky genus $g$ curve $\Cc$ is a 1-dimensional smooth proper tame Artin stack whose coarse moduli space is isomorphic to an irreducible projective genus-$g$ curve $\pi : \Cc \to C$. A twisted genus $g$ curve is a stacky genus $g$ curve $\Cc$ which is generically a scheme and such that for each geometric point $\bar{x} \to C$, there exists a non-negative integer $r$ such that
		$$
		\Spec (\cO_{C,\bar{x}}^{sh}) \times_C \Cc \cong \left[ \Spec k(\bar{x})[t]^{sh} / \mu_r\right]
		$$
		where $\cO_{C,\bar{X}}^{sh}$ is the strict Henselization at $\bar{x}$ and $r$ acts by $t \mapsto \zeta t$ for $\zeta \in \mu_r$. 
	\end{defn}
	
	\par Note that a twisted curve is a root stack (c.f. \cite[\S 10.3]{Olsson2}). In particular, $\Cc$ is uniquely determined up to an isomorphism by the points $q_1,\dotsc,q_s \in C$ and their inertia group orders $r_1,\dotsc,r_s$. The reduced preimages $\Sigma_i := \pi^{-1}(q_i)_{red}$ are residual $\mu_{r_i}$-gerbes at the points $p_i$ lying over $q_i$. 
	
	\begin{rmk} In \cite{AOV2}, a more general definition of twisted curve is used where $\cC$ is allowed to have at worst nodal singularities. Our definition is exactly the ones of \emph{loc. cit.} which are smooth. 
	\end{rmk}
	
	
	\begin{defn} A family of twisted curves of genus $g$ with inertia groups $\mu_{r_1}, \ldots, \mu_{r_s}$ over $B$ is a tuple $(\Cc \to B, \Sigma_i)$ where
		\begin{enumerate}
			\item $\Cc \to B$ is a flat and proper morphism with $\Cc_b$ a stacky curve of genus $g$ for all $b \in B$, 
			\item $\Sigma_i \subset \Cc$ are closed substacks such that the composition $\Sigma_i \to B$ is a $\mu_{r_i}$-gerbe, and
			\item $\Cc \setminus \bigsqcup \Sigma_i \to B$ is representable. 
		\end{enumerate}
		A family of twisted curves is an object as above for some $g$ and tuple $(r_1, \ldots, r_s)$. 
	\end{defn}
	
	Now we can define the stack of twisted maps. 
	
	\begin{defn}  Let $\cX$ be a proper and tame Artin stack. A family of twisted maps to $\cX$ over $B$ is a tuple $(\Cc \to B, \Sigma_i, f)$ where $(\Cc, \Sigma_i) \to B$ is a family of twisted curves and $f : \Cc \to \cX$ is a representable morphism.
	\end{defn}
	
	When $\cX = \Pc(\lambda_j)$, a map $f : \Cc \to \Pc(\lambda_j)$ is the same as the data of a line bundle $\cL \in \mathrm{Pic}(\Cc)$ and sections $s_j \in \cL^{\lambda_j}$ which don't simultaneously vanish. Since $\Pc(\lambda_j) \to \mathcal{B} \bG_m$ is representable, then $f$ is representable if and only if the map $\Cc \to \mathcal{B} \bG_m$ defined by the line bundle $\cL$ is representable. In order to describe line bundles on $\Cc$, we consider the coarse map $\pi : \Cc \to C$ and denote as above $q_1, \ldots, q_s \in C$ the points where $\pi$ is not an isomorphism with $p_i$ the point of $\Cc$ lying over $q_i$ and with stabilizer $\mu_{r_i}$.
	
	\begin{prop}\label{prop:twisted:pic} The pullback map $\pi^*$ induces an injection $\pi^*:\Pic(C) \to \Pic(\Cc)$ and a short exact sequence
		\begin{equation}\label{eqn:twisted:pic}
			0 \to \Pic(C) \to \Pic(\Cc) \to \bigoplus_{i}^s \mathbb{Z}/r_i\mathbb{Z} \to 0
		\end{equation}
		Moreover, $\Pic(\cC)$ is generated by $\Pic(C)$ and $\cO_{\cC}(p_i)$ for $i = 1, \ldots, s$ and $\cO_{\cC}(p_i)$ map to generators of the cokernel of $\pi^*$. \\
	\end{prop}
	\begin{proof} Consider the Leray spectral sequence for $R\pi_* \bG_{m,\cC}$. First since $\pi_*\cO_{\cC} = \cO_C$, we conclude that $\pi_*\bG_{m, \cC} = \bG_{m,C}$. 
		
		Next we compute $R^1\pi_*\bG_{m,\cC}$. The stalk $(R^1\pi_*\bG_{m,\cC})_x = 0$ for any point $x$ where $\pi$ is an isomorphism. Thus  $R^1\pi_*\bG_{m,\cC}$ is a direct sum of local contributions over stacky points. By flat base change, it suffices to compute it in the special case that $\cC = [\Spec A/\mu_r]$ where $A$ is a 1 dimensional regular Henselian local ring. In this case, $R^1\pi_*\bG_{m, \cC}$ is simply $\Pic(\cC)$ but $\Pic(\Spec A)$ is trivial so this is the same as the character group of $\mu_r$ which we can identify with $\Z/r\Z$.
		
		Putting this together with the exact sequence coming from low degree terms of the Leray spectral sequence for $R\pi_*\bG_{m,\cC}$, we get
		$$
		0 \to \Pic(C) \to \Pic(\cC) \to \bigoplus_{i = 1}^{s} \Z/r_i \Z
		$$
		where the last map can be identified as the restriction $\Pic(\cC) \to \Pic(\mathcal{B} \mu_{r_i})$. 
		
		To prove exactness on the right, it suffices to construct a line bundle whose character at the $i^{th}$ marked gerbe $p_i$ generates $\Z/r_i\Z$ and is trivial away from $p_i$. The line bundle $\cO_{\cC}(p_i)$ does the job and this also proves the last claim. 
	\end{proof} 
	
	\begin{rmk}\label{rmk:character} Given a line bundle $\cL$ on $\cC$ and a point $p$ with stabilizer $\mu_{r_i}$, we obtain a character $\chi^{-a} : \zeta \in \mu_{r_i} \mapsto \zeta^{-a} \in \bG_m$ via the action of $\mu_{r_i}$ on the fiber $\cL|_{p_i}$. This corresponds to the image $\cL$ consisting of $a \mod r_i$ in the $i^{th}$ component of the cokernel $\pi^*$. Indeed, as in the proof of the proposition, we can identify the generator of $\Z/r_i\Z$ with the pullback of $\cO(p_i)$ to the strict Henselization $[\Spec \overline{k(p_i)}[t]^{sh}/\mu_r]$ where $t$ is a uniformizer at $p_i$. The pullback of $\cO(p_i)$ is generated by $t^{-1}$ so the fiber of $\cO(p_i)$ carries the character $\chi^{-1}$. 
	\end{rmk}
	
	\begin{cor} There is a presentation of abelian groups
		$$
		\Pic(\cC) = \left(\Pic(C) \oplus \bigoplus_{i = 1}^s \Z [\cO(p_i)] \right)/ \left( \left\{r_i[\cO(p_i)] = [\cO(q_i)] \right\}_{i = 1}^s \right)
		$$
		and of component groups
		$$
		\pi_0(\Pic(\cC)) = (\Z \oplus \Z c_1 \oplus \ldots \oplus \Z c_s)/(\{r_ic_i =  [k(q_i):k]\}_{1 = 1}^s).
		$$
		Moreover, there is a well defined degree homomorphism
		$$
		\deg : \Pic(\cC) \to \Q
		$$
		extending the usual degree map $\deg : \Pic(C) \to \Z$ such that $\deg(c_i) = \frac{1}{r_i}[k(q_i):k]$. 
	\end{cor}
	
	\begin{cor}\label{cor:push:round:down} Every line bundle $\cL$ on $\cC$ can be written uniquely as 
		$$
		\pi^*L\left( \sum_{i = 1}^s a_i p_i \right)
		$$
		where $0 \le a_i < r_i$ where $L \in \Pic(C)$. Moreover, there is an isomorphism $\pi_*\cL \cong L$. 
	\end{cor}
	\begin{proof} The first part follows immediately from the presentation of $\Pic(\cC)$. For the second part, consider the exact sequence 
		$$
		0 \to \pi^*L \to \cL \to \cL \otimes \cO_{\sum a_i p_i} \to 0. 
		$$
		Applying $\pi_*$ and using the fact that $\cC$ is tame, it suffices to show that $\pi_* (\cL \otimes \cO_{\sum a_i p_i}) = 0$. By induction on $s$, it suffices to show that $\pi_*(\cL \otimes \cO_{ap}) = 0$ where $p$ is a point with stabilizer $\mu_r$ and $0 < a < r$. After passing to the strict Henselization, we can compute $\cL \otimes \cO_{ap}$ explicitly as the $\mu_r$ representation on $k(p)[t]^{sh}/(t^a)t^{-a}$ where $\mu_r$ acts on $t$ as in the definition of twisted curve. In particular, $\pi_*(\cL \otimes \cO_{ap}) = (k(p)[t]^{sh}/(t^a)t^{-a})^{\mu_r} = 0$ as required. 
	\end{proof}
	
	\begin{cor}
		If $\cC \to B$ is a family of twisted curves over a connected base, then the quotient $\Pic_{\cC/B}/\Pic_{C/B}$ is a finite \'etale group scheme over $B$ with fibers isomorphic to $\oplus_{i = 1}^s \bZ/r_i\bZ$ for some fixed tuple $(r_1, \ldots, r_s)$.  
	\end{cor}
	
	\begin{defn}
		If $(\cX, \cL)$ is a proper cyclotomic stack with a polarizing line bundle $\cL$, we define the degree of a map $f : \cC \to \cX$ as $\deg (f^*\cL)$. 
	\end{defn}
	\begin{rmk}\label{rem:deg}
		Suppose that $\cL^{\otimes M} = \pi^*L$ where $\pi : \cX \to X$ is the coarse map and $L$ is ample. Then the degree of $f$ in our definition agrees with $\frac{1}{M}\deg(C \to X)$ where the degree of the coarse map is measured with respect to $L$. 
	\end{rmk}
	
	\begin{lem}\label{lemma:loc:const}
		Let $(\cX, \cL)$ be a proper polarized cyclotomic stack.
		\begin{enumerate}
			\item If $\cC$ is a twisted curve and $f : \cC \to \cX$ is any morphism, then $f$ is representable if and only if for each $i = 1, \ldots, s$, the projection of the class $[f^*\cL]$ to $\bZ/r_i\bZ$ under the sequence \ref{eqn:twisted:pic} is a unit. 
			\item If $f : \cC/B \to \cX$ is a family of twisted maps to $\cX$, then $\deg(f_b)$ and the class of $[f_b^*\cL]$ in $\Pic(\cC_b)/\Pic(C_b)$ are locally constant functions on $B$. 
		\end{enumerate}
	\end{lem}
	
	\begin{defn}\label{def:local-conditions} 
		We let $\Gamma = (\{r_1, a_1\}, \ldots, \{r_s, a_s\})$ denote a tuple of pairs of integers where $r_i > 1$ and $0 < a_i < r_i$ is a unit mod $r_i$. We call $\Gamma$ the tuple of local conditions for a twisted map. We say that $\Gamma$ is admissible for $(\cX, \cL)$ if $r_i \mid M$ for all $i$ where $M > 0$ is the smallest positive integer such that $\cL^{\otimes M} = \pi^*L$ where $\pi$ is the coarse moduli map. 
	\end{defn}
	
	We are now ready to construct the stack of twisted maps from a fixed curve $C$ to a polarized cyclotomic stack $(\cX,\cL)$.
	
	\begin{defn}\label{defn:twisted:maps} Fix $(\cX, \cL)$, $C$ and $\Gamma$ as above and an integer $d$. A family of twisted maps from $C$ to $\cX$ of type $\Gamma$ and degree $d$ over a scheme $B$ is a family of twisted maps $(\cC \to B, \Sigma_i, f)$ with inertia groups $\mu_{r_1}, \ldots, \mu_{r_s}$ and a map $\pi : \cC \to C \times B$ over $B$ such that 
		\begin{enumerate}
			\item $\pi$ is the coarse moduli space of $\cC$, 
			\item for all $b \in B, \deg(f_b) = d$, and
			\item for all $b \in B$ the class of $f_b^*\cL$ in $\Pic(\cC_b)/\Pic(C)$ is given by
			$$
			\sum_{i = 1}^s (a_i \mod r_i). 
			$$
		\end{enumerate}
	\end{defn}
	
	\begin{thm}\label{thm:moduli-of-twisted-maps-aov}
		Let $(\cX, \cL)$ be a proper polarized cyclotomic stack over $k$ with coarse moduli space $\pi : \cX \to X$ and fix an integer $M$ and ample line bundle $L$ on $X$ such that $\cL^{\otimes M} = \pi^*L$. For each smooth projective curve $C$ of genus $g$, local conditions $\Gamma$ and degree $d$, there exists a finite type separated algebraic stack $\cH^\Gamma_{d,C} = \cH^\Gamma_{d,C}(\cX, \cL)$ with quasi-projective coarse moduli space parametrizing twisted maps from $C$ to $\cX$ of type $\Gamma$ and degree $d$. Moreover, $\cH^\Gamma_{d,C}$ is quasi-finite over the scheme $\Hom_{Md}((C, q_1, \ldots, q_s), X)$ of $s$-pointed degree $Md$ maps $(C, q_1, \ldots, q_s) \to X$. 
	\end{thm}
	
	\begin{proof} 
		We first consider the stack of twisted curves with coarse moduli space $C$ and inertia groups $\mu_{r_1}, \ldots, \mu_{r_s}$. This is the stack of data $(\cC \to B, \Sigma_i, \pi)$ where $(\cC \to B, \Sigma_i)$ is a twisted curve with stabilizer $\mu_{r_i}$ along $\Sigma_i$ and $\pi : \cC \to C \times B$ is a map over $B$ which exhibits $C \times B$ as the coarse moduli space of $\cC$. The images of $\Sigma_i$ under $\pi$ define $s$ disjoint sections $\sigma_i : B \to C \times B$. The data of $\sigma_i$ defines a map $B \to \Conf_s(C)$, the configuration space of $s$ distinct points on $C$ and the association 
		$$
		(\cC \to B, \Sigma_i, \pi) \mapsto (B \to \Conf_s(C))
		$$
		is functorial since coarse moduli space commutes with basechange for tame stacks \cite[Cor. 3.3]{AOV}. 
		
		On the other hand, given a map $B \to \Conf_s(C)$ induced by a family of $s$ disjoint sections $\sigma_1, \ldots, \sigma_n : B \to B \times C$, the images of $\sigma_i$ are relative effective Cartier divisors. We can then take the $r_i$ root stack along the image of $\sigma_i$ for each $i = 1, \ldots, s$ to obtain coarse moduli map $\pi : \cC \to C \times B$. The reduced preimages $\Sigma_i = \pi^*(\sigma_i)_{red}$ are disjoint $\mu_{r_i}$-gerbes over $B$ and $(\cC \to B, \Sigma_i, \pi)$ is a family of twisted curves with coarse moduli space $C$ and inertia groups $\mu_{r_i}$. Thus this stack is representable by the configuration space $Z := \Conf_s(C)$. 
		
		Let $(\cC \to Z, \Sigma_i)$ be the universal family of twisted curves over the configuration space and let 
		$$
		\cH = \Hom_{Z}(\cC, \cX \times Z)
		$$
		be the relative Hom-stack over $Z$. Then $\cH$ is an algebraic stack locally of finite type with quasi-compact and separated diagonal by \cite[Thm. C.2]{AOV}. By Lemma \ref{lemma:loc:const}, the degree and the class in $\Pic(\cC)/\Pic(C)$ are locally constant on $\cH$ and so there is an open and closed (and thus algebraic) substack $\cH^\Gamma_{d,C} \subset \cH$ parametrizing those maps of degree $d$ and local twisting conditions $\Gamma$. 
		
		To see that $\cH^\Gamma_{d,C}$ is separated and has finite inertia, we can suppose that the base field is algebraically closed. Composing with the closed embedding $\cX \to \Pcv$ yields a representable monomorphism
		$$
		\Hom_{Z}(\cC, \cX \times Z) \to \Hom_{Z}(\cC, \Pcv \times Z)
		$$
		which preserves the degree $d$ and twisting condition $\Gamma$. Thus it suffices to prove that $\cH^\Gamma_{d,C}$ is separated with finite inertia for target weighted projective stack. In this case, the claim follows from the proof of Theorem \ref{thm:height:moduli} for weighted projective stacks as in \S \ref{proof:wps}, where it is shown that $\cH^\Gamma_{d,C}$ is isomorphic to a locally closed substack of a separated stack with finite inertia $\cR^\mu_n$ (c.f. Definition \ref{def:moduli-interp-rnmu}).

		Taking a tuple $(\cC \to B, \Sigma_i, f, \pi)$ to its coarse moduli space $(C \times B \to X, \sigma_i)$ produces a morphism 
		$$
		\rho :\cH^\Gamma_{d, C} \to \Hom_{Md}((C, q_1, \ldots, q_s), X).
		$$
		Note that the family of maps $C \times B \to X$ has degree $Md$ by Remark \ref{rem:deg}. The map $\rho$ is quasi-finite by \cite[Prop. 4.4]{AOV} and its proof. Let $$\psi : H^\Gamma_{d,C} \to \Hom_{Md}((C, q_1, \ldots, q_s), X)$$
		be the factorization through the coarse moduli space. This map is also quasi-finite so by Zariski's Main Theorem, we can factor it as an open immersion $H^\Gamma_{d,C} \subset Y$ followed by a finite morphism $\bar{\psi} : Y \to \Hom_{Md}((C, q_1, \ldots, q_s), X)$. Since the target is a quasi-projective scheme, it follows that $Y$ and thus $H^\Gamma_{d,C}$ is quasi-projective. This completes the proof. 
	\end{proof}
	
	Later we will need to work with twisted maps where the gerbes are not marked. 
	
	\begin{defn}\label{def:Sgamma}
		Fix a tuple of local twisting conditions $\Gamma = (\{r_1, a_1\}, \ldots, \{r_s, a_s\})$. Let $S_{\Gamma} \subset \mathrm{Aut}\{1, \ldots, s\}$ be the subset of permutations that only permute those indices which have the same pair of $\{r_j, a_j\}$.
	\end{defn}
	
	\begin{prop}\label{prop:quotient-Hbar-Gamma-dC} The group $S_{\Gamma}$ acts on $\cH^{\Gamma}_{d,C}$ and the quotient $\overline{\cH}^\Gamma_{d,C}:=\cH^\Gamma_{d,C}/S_{\Gamma}$ is the stack of diagrams $(\cC \to B, f : \cC \to \cX, \pi : \cC \to C \times B)$ such that for each geometric point $\bar{t} \in B$, $\cC_{\bar{t}}$ is a twisted curve with inertia groups $\mu_{r_1}, \ldots, \mu_{r_s}$ and satisfying $(1), (2)$ and $(3)$ of Definition \ref{defn:twisted:maps}. 
	\end{prop}
	
	
	\section{Correspondence between weighted linear series and twisted maps}
	\label{sec:min-lin-ser-univ-tuning}
	
	Let $K=k(C)$ be the function field of a smooth projective curve $C$. Then rational points $P\in\Pcv(K)$ correspond to rational maps $C\dasharrow\Pcv$. In this section, we show how the universal tuning stack of $P$ is explicitly constructed from the data of a $\lambdavec$-weighted linear series (Definition \ref{defn:wls}). Furthermore, we show that every rational map $C\dasharrow\Pcv$ is determined by a unique \emph{minimal weighted linear series} defined below. Finally, we compute the height of a point in terms of the minimal weighted linear series. 
	
	The indeterminacies of the associated map $C\dasharrow\Pcv$ are the set of points $x \in C$ such that $s_i(x) = 0$ for all $i$, that is, the set theoretic base locus of the weighted linear series. Two $\vec{\lambda}$-weighted linear series are said to be equivalent if they induce the same rational map $f \colon C \dasharrow \mathcal{P}(\vec{\lambda})$, or equivalently if they induce the same rational point $P \in \Pcv(K)$. We use the notation $\nu_x(s_i)$ to denote the order of vanishing of $s_i$ at $x$.

	\begin{defn}\label{def:rmin-for-any-lin-ser}
		Let $(L, s_0, \ldots, s_N)$ be a $\vec{\lambda}$-weighted linear series on a curve $C$. For every $x\in C$, we let
		\begin{align*}
			r_{\min}(x;L,s_0,\dots,s_N)&:=\frac{\lambda_j}{\gcd(\nu_x(s_j),\lambda_j)} \\
			a_{\min}(x;L,s_0,\dots,s_N)&:=\frac{\nu_x(s_j)}{\gcd(\nu_x(s_j),\lambda_j)}
		\end{align*}
		where $j$ is a choice of index such that
		\[
		\frac{\nu_x(s_j)}{\lambda_j}=\min_i\left\{\frac{\nu_x(s_i)}{\lambda_i}\right\}.
		\]   
	\end{defn}
	
	\begin{defn}\label{def:mimal-wted-lin-ser} A $\vec{\lambda}$-weighted linear series $(L, s_0, \ldots, s_N)$ is \emph{minimal} if for each indeterminacy point $x \in C$, there exists an $j$ such that $\nu_x(s_j) < \lambda_i$.
	\end{defn} 
	
	We now state the main result of this section.
	
	\begin{thm}\label{thm:tuning-stacks-and-minimal-linear-series}
		Let $K=k(C)$ be the function field of a smooth projective curve $C$, let $f\colon C\dasharrow\Pcv$ be a rational map, and let $P\in\Pc_{C}(\vec\lambda)(K)$ denote the corresponding rational point, where $\Pc_{C}(\vec\lambda)=\Pcv\times C\to C$ is the constant family. Let $\{x_j\}$ be the indeterminacy points of $f$. Assume the $\lambda_i$ are prime to the characteristic of the ground field.
		
		\begin{enumerate}
			\item\label{thm:tuning-stacks-and-minimal-linear-series::tuning-stack} Let $(L,s_0,\dots,s_N)$ be any $\lambdavec$-weighted linear series inducing $f$. Then the universal tuning stack $(\cC,\pi,\overline{P})$ of $P$ is the root stack of $C$ obtained by taking the $r_j$-th root at $x_j$, where $r_j=r_{\min}(x_j;L,s_0,\dots,s_N)$. Moreover, the induced morphism on stabilizers over $x_j$ is given by the character $\chi_j^{-a_j}$ where $a_j = a_{\min}(x_j, L, s_0, \ldots, s_N)$. 
			
			\item\label{thm:tuning-stacks-and-minimal-linear-series::minimal-lin-ser} There exists a unique minimal $\lambdavec$-weighted linear series inducing $f$.
			
			\item\label{thm:tuning-stacks-and-minimal-linear-series::height} The stacky height $\height_{\cO(1)}(P)$ is equal to $\deg L$ where $(L, s_0, \dots, s_N)$ is the unique minimal linear series. Moreover, the stable height is given by $\height_{\cO(1)}^{st}(P) = \deg \overline{P}^*\cO(1)$ and the local contribution at $x_j$ is given by $\delta_{x_j}(P) = \frac{a_j}{r_j}[k(x_j):k]$. 
		\end{enumerate}
	\end{thm}
	
	\begin{rmk} Since the quotient map $\bA^{N + 1} \setminus 0 \to \Pcv$ can be identified with the total space of $\cO(-1)$, the morphism of stabilizers $\mu_r \to \bG_m$ induced by the structure map $\Pcv \to \mathcal{B} \bG_m$ can be identified with the character of the stabilizer $\mu_r$ acting on the fiber of $\cO(-1)$. By the theorem, the action of $\mu_{r_j}$ on the fiber of the line bundle $\bar{P}^*\cO(1)$ on the universal tuning stack $\cC$ is then given by $\chi_j^{a_j}$. Combining this with Proposition \ref{prop:twisted:pic} and Remark \ref{rmk:character}, we conclude that the the class in $\Pic(\cC)/\Pic(C)$ of $\overline{P}^*\cO(1)$ is
		$$
		\sum _j (-a_{\min}(x_j) \mod r_{\min}(x_j)) \in \bigoplus_j \Z/r_{\min}(x_j)\Z
		$$
		where the sum runs over the indeterminacy points of the weighted linear series.    
	\end{rmk}
	
	\subsection{Characterizing the universal tuning stack}
	
	In this subsection, we prove Theorem \ref{thm:tuning-stacks-and-minimal-linear-series} (\ref{thm:tuning-stacks-and-minimal-linear-series::tuning-stack}). The following is the key local computation needed to prove the result.
	
	
	\begin{prop}\label{prop:univ-tuning-stack-local-computation}
		Let $R$ be a DVR over a field $k$ and let $\Omega$ be the fraction field of $R$. Consider the $\Omega$-point $(s_0,\dots,s_N)$ of $\bA^{N+1}\setminus0$ which yields a map $f\colon\Spec\Omega\to\Pcv$. Assume $\lambda_0,\dots,\lambda_N$ are prime to the characteristic of $k$. Let $s_j$ have valuation $\nu_j$ and suppose
		\begin{equation}\label{eqn:minnu0-extend-DVR}
			\frac{\nu_0}{\lambda_0}=\min_j\left\{\frac{\nu_j}{\lambda_j}\right\}.
		\end{equation}
		Let
		\[
		r=\frac{\lambda_0}{\gcd(\nu_0,\lambda_0)}
		\]
		and $\cX$ be the $r$-root stack of $\Spec R$ at the closed point. Then $f$ extends uniquely (up to unique $2$-isomorphism) to a map $g\colon\cX\to\Pcv$. Furthermore, $g$ is representable.
	\end{prop}
	\begin{proof}
		Upon showing that $f$ extends to $g$, since $\cC$ and $\Pcv$ are separated Deligne--Mumford stacks and $\cC$ is normal, \cite[Prop. 1.2]{FMN} shows that $g$ is unique up to unique $2$-isomorphism.
		
		We now turn to the construction of $g$. Let $u$ be a uniformizer for $R$. Recall that $\cX=[\Spec R'/\mu_r]$, where $R'=R[u']/((u')^r-u)$ and $\mu_r$ acts with weight $1$ on $u'$. Let $\Omega'$ be the fraction field of $R'$. Write $s_j=u^{\nu_j}t_j$, where $t_j$ has valuation $0$. Our $\Omega$-point $(s_0,\dots,s_N)\in\bA^{N+1}\setminus0$ can be viewed as an $\Omega'$-point; we show that after acting by $\bG_m(\Omega')$, we may extend it to an $R'$-point. Notice that $\lambda_0\mid r\nu_0$; acting by $(u')^{-r\nu_0/\lambda_0}\in\bG_m(\Omega')$, we obtain
		\[
		s':=(u')^{-\frac{r\nu_0}{\lambda_0}}*(s_0,\dots,s_N)=(t_0,(u')^{r\nu_1-\frac{r\nu_0}{\lambda_0}\lambda_1}t_1,\dots,(u')^{r\nu_N-\frac{r\nu_0}{\lambda_0}\lambda_N}t_N).
		\]
		By \eqref{eqn:minnu0-extend-DVR}, we see
		\[
		r\nu_j-\frac{r\nu_0}{\lambda_0}\lambda_j\geq0
		\]
		and so $s'$ defines an $R'$-point of $\bA^{N+1}$. Since the first coordinate of $s'$ has valuation $0$, we see $s'$ is an $R'$-point of $\bA^{N+1}\setminus0$. Furthermore, if we let
		\[
		\chi\colon\mu_r\to\bG_m,\quad\quad\chi(\zeta)=\zeta^{-\frac{r\nu_0}{\lambda_0}}
		\]
		then we have a commutative diagram
		\[
		\xymatrix{
			\mu_r\times\Spec R'\ar[r]^-{\chi\times s'}\ar[d]_-{\sigma'} & \bG_m\times(\bA^{N+1}\setminus0)\ar[d]^-{\sigma}\\
			\Spec R'\ar[r]^-{s'} & \bA^{N+1}\setminus0
		}
		\]
		where $\sigma$ and $\sigma'$ are the two action maps. As a result, $s'$ induces a map
		\[
		\xymatrix{
			\cX=[\Spec R'/\mu_r]\ar[r]^-{g} & \Pcv
		}
		\]
		which restricts to $f$ on $[\Spec\Omega'/\mu_r]=\Spec\Omega$.
		
		Lastly, we see $\frac{r\nu_0}{\lambda_0}=\frac{\nu_0}{\gcd(\nu_0,\lambda_0)}$ and $r=\frac{\lambda_0}{\gcd(\nu_0,\lambda_0)}$ are relatively prime. As a result, $\chi$ is injective, and so $g$ is representable.
	\end{proof}

	\begin{proof}[{Proof of Theorem \ref{thm:tuning-stacks-and-minimal-linear-series} (\ref{thm:tuning-stacks-and-minimal-linear-series::tuning-stack})}]
		Let $\pi\colon\cC\to C$ be the root stack obtained by taking the $r_i$-th root at $x_i$. By construction, $\pi$ is an isomorphism over $C\setminus\{x_i\}$. We show there is a representable map $g\colon\cC\to\Pcv$ which agrees with $f$ on $C\setminus\{x_i\}$. Upon showing this, we see that, by definition, $(\cC,\pi,g)$ is a tuning stack; moreover, by Remark \ref{rmk:universal-tuning-stack-rep}, $(\cC,\pi,g)$ is the universal tuning stack.
		
		It remains to prove the existence of a representable map $g\colon\cC\to\Pcv$ extending $f$. We first note that it is enough to show the existence of an \'etale cover $\{V_\ell\to C\}_\ell$ and representable maps $g_\ell\colon\cC\times_C V_\ell\to\Pcv$ that agree with $f$ over the inverse image of $C\setminus\{x_i\}$ under the map $V_\ell\to C$. Indeed, by \cite[Prop. 1.2]{FMN}, any such $g_\ell$ is uniquely determined (up to unique $2$-isomorphism) by $f$; hence, the $g_\ell$ descend to yield a map $g\colon\cC\to\Pcv$. Since representability can be checked on geometric points \cite[Cor. 2.2.7]{Conrad}, representability of the $g_\ell$ imply representability of $g$.
		
		To construct our desired \'etale cover $\{V_\ell\to C\}_\ell$ and representable maps $g_\ell\colon\cC\times_C V_\ell\to\Pcv$, it is enough to do so at the strict Henselization $R_i$ of each point of indeterminacy $x_i$. To see this, we may write $\Spec R_i$ as an inverse limit of affine \'etale neighborhoods $V_{i,\ell}\to C$ of $x_i$; then by Proposition B.1 and Proposition B.3 of \cite{rydh}, for any representable map $g_i\colon\cC\times_C\Spec R_i\to\Pcv$ which agrees with $f$, there is some index $\ell$ and an extension $g_{i,\ell}\colon\cC\times_C V_{i,\ell}\to\Pcv$ of $g_i$ where $g_{i,\ell}$ is representable and agrees with $f$.
		
		For the remainder of the proof, we fix a point of indeterminacy $x_i$. Then $R_i$ is a DVR by \cite[\href{https://stacks.math.columbia.edu/tag/0AP3}{Tag 0AP3}]{Stacks}; let $\Omega_i$ be its fraction field. The rational map $f\colon C\dasharrow\Pcv$ yields a morphism $f_i\colon\Spec\Omega_i\to\Pcv$. We continue to denote by $(L,s_0,\dots,s_N)$ the pullback of the linear series to $\Spec R_i$. Since $L$ is trivial over $\Spec\Omega_i$, we see that $f_i$ lifts to the cover $\Spec\Omega_i\to\bA^{N+1}\setminus0$, where this latter map is defined by the $\Omega_i$-point $(s_0,\dots,s_N)$. Letting $\nu_j$ denote the order of vanishing of $s_j$ at the closed point of $\Spec R_i$, we may assume without loss of generality that
		\[
		\frac{\nu_0}{\lambda_0}=\min_j\left\{ \frac{\nu_j}{\lambda_j}\right\}.
		\]
		By Proposition \ref{prop:univ-tuning-stack-local-computation}, we see $f_i$ extends to a representable map $g_i\colon\cC\times_C\Spec R_i\to\Pcv$.
	\end{proof}

	\subsection{Uniqueness of minimal linear series}
	
	In this subsection, we prove Theorem 3.3(\ref{thm:tuning-stacks-and-minimal-linear-series::minimal-lin-ser}). We begin by associating a minimal linear series on $C$ to any morphism $\cC\to\Pcv$, where $\cC$ is a tame root stack over $C$.
	
	\begin{prop}\label{prop:tuningstack->min-lin-ser}
		Let $C$ be a smooth proper curve over a field $k$, let $x_1,\dots,x_m\in C$, and let $r_1,\dots,r_m$ be positive integers prime to the characteristic of $k$. Let $\pi\colon\mathcal{C}\to C$ be the root stack obtained by taking the $r_j$-th root at $x_j$. Let $y_j$ denoted the reduced preimage of $x_j$. By definition, $\cC$ carries distinguished line bundles with section $(\cO_{\cC}(y_j),u_j)$ and isomorphisms $\iota_j\colon\cO_{\cC}(y_j)^{\otimes r_j}\xrightarrow{\simeq}\pi^*\cO_C(x_j)$ with $\iota_j(u_j^{\otimes r_j})$ vanishing to order $1$ at $x_j$.
		
		Let $g\colon\mathcal{C} \to \mathcal{P}(\vec{\lambda})$ be a morphism induced by a basepoint-free $\vec{\lambda}$-weighted linear series $(\mathcal{L}, t_0,\ldots, t_N)$. Let $\chi_j\colon\mu_{r_j}\to\bG_m$ denote the canonical embedding and suppose that the stabilizer $\mu_{r_j}$ of $y_j$ acts on $\mathcal{L}$ through the character $\chi_j^{a_j}$ with $0\leq a_j<r_j$. Then 
		$$
		\left(\pi_*\mathcal{L}(\sum_j a_j y_j), \pi_*(t_0 \prod_j u_j^{\lambda_0 a_j}), \ldots, \pi_*(t_N \prod_j u_j^{\lambda_N a_j})\right)
		$$
		is a minimal weighted linear series on $C$ inducing the rational map $f\colon C \dashrightarrow \mathcal{P}(\vec{\lambda})$ given by factoring $g$ through the coarse space. 
	\end{prop}
	\begin{proof} By construction, the stabilizers of $\mathcal{C}$ act trivially on $\mathcal{L}(\sum_j a_j y_j)$. Thus $\mathcal{L}(\sum_j a_j y_j) = \pi^*L$ for a unique line bundle $L$ on $C$, e.g., by \cite[Thm. 10.3]{Alper}. Since $\pi$ is tame, we have $\pi_*\mathcal{O}_{\mathcal{C}} = \mathcal{O}_C$ and so by the projection formula, 
		$$
		\pi_*\mathcal{L}(\sum_j a_j y_j) = \pi_*\pi^*L = L. 
		$$
		Now $\pi^*$ is multiplicative so $\mathcal{L}(\sum_j a_j y_j)^k = \pi^*L^k$ and again by the projection formula we have 
		$$
		\pi_*(\mathcal{L}(\sum_j a_j y_j)^k) = \pi_*\pi^*L^k = L^k = (\pi_*\mathcal{L}(\sum_j a_j y_j))^k.
		$$
		We conclude that 
		$$
		\pi_*(t_i \prod_j u_j^{\lambda_i a_j}) =: s_i
		$$
		is a section of $L^{\lambda_i}$ as required. Moreover, $\pi$ is an isomorphism on $C\setminus\{x_j\}$ so the rational map $f$ induced by the weighted linear series $(L, s_0, \ldots, s_N)$ agrees with $g$ on $C\setminus\{x_j\}$. 
		
		Finally, we check that $(L, s_0, \ldots, s_N)$ is minimal. Since $g$ is a morphism, for each $j$ there exists an $i$ such that $t_i$ does not vanish at $y_j$. Thus the order of vanishing of $t_i\prod_j u_j^{\lambda_i a_j}$ at $y_j$ is $\lambda_i a_j$ and the order of vanishing of $s_i$ at $x_j$ is 
		$$
		\frac{\lambda_i a_j}{r_j} < \lambda_i,$$
		proving minimality of the weighted linear series. Note that $\frac{\lambda_i a_j}{r_j}$ is an integer since $s_i$ does not vanish at $y_j$, showing that $\mathcal{L}^{\lambda_i}$ is trivial when pulled back to the residual gerbe at $y_j$; since the stabilizer $\mu_{r_j}$ acts on $\mathcal{L}^{\lambda_i}$ through the character $\chi_j^{-\lambda_i a_j}$ and since this character must be trivial, we see $r_j\mid\lambda_i a_j$.
	\end{proof}

	Next we consider the converse to Proposition \ref{prop:tuningstack->min-lin-ser} in the case where $\cC$ is the universal tuning stack, characterizing the weighted linear series of $\cC\to\Pcv$ in terms of the weighted linear series of $C\dasharrow\Pcv$.
	
	\begin{prop}\label{prop:lin-ser->univ-tuning-stack-map}
		Let $(L, s_0,\ldots, s_N)$ be a weighted linear series on $C$ inducing a rational map $f\colon C\dasharrow\Pcv$, and assume $\lambda_0,\dots,\lambda_N$ are prime to the characteristic. Let $(\cC,\pi,\overline{P})$ be as in Theorem \ref{thm:tuning-stacks-and-minimal-linear-series} (\ref{thm:tuning-stacks-and-minimal-linear-series::tuning-stack}). Let 
		\[
		a_j = \frac{\nu_{x_j}(s_{i_j})}{\gcd(\nu_{x_j}(s_{i_j}),\lambda_{i_j})}
		\]
		where
		\[
		\frac{\nu_{x_j}(s_{i_j})}{\lambda_{i_j}}=\min_i\left\{\frac{\nu_{x_j}(s_i)}{\lambda_i}\right\}.
		\]
		Then the morphism $\overline{P}\colon\cC\to\Pcv$ is defined by the weighted linear series $(\mathcal{L},t_0,\dots,t_N)$ with
		\[
		\mathcal{L}=\pi^*L(-\sum a_j y_j)\quad\textrm{and}\quad
		t_i = \frac{\pi^*s_i}{\prod_j u_j^{\lambda_i a_j}}.
		\]
		
		Furthermore, if $(L,s_0,\dots,s_N)$ is minimal and $(L',s'_0,\dots,s'_N)$ denotes the minimal linear series obtained from $(\mathcal{L},t_0,\dots,t_N)$ in Proposition \ref{prop:tuningstack->min-lin-ser}, then there is a canonical isomorphism $\beta\colon L\xrightarrow{\simeq} L'$ such that $\beta(s_i)=s'_i$.
	\end{prop}
	
	\begin{proof} A priori, $t_i$ is only a rational section of $\mathcal{L}^{\lambda_i}$. We first show it is a regular section. By construction, $\pi \colon \mathcal{C} \to C$ is the $r_j^{th}$ root stack along $x_j$ where
		$$
		r_j = \frac{\lambda_{i_j}}{\gcd(\nu_{x_j}(s_{i_j}),\lambda_{i_j})}. 
		$$
		For all $j$, we have 
		$$
		\nu_{y_j}(\pi^*s_i) = r_j\nu_{x_j}(s_i) = \frac{\lambda_{i_j} \nu_{x_j}(s_i)}{\gcd(\nu_{x_j}(s_{i_j}), \lambda_{i_j})} \geq \frac{\lambda_i \nu_{x_j}(s_{i_j})}{\gcd(\nu_{x_j}(s_{i_j}), \lambda_{i_j})} = \lambda_i a_j
		$$
		since the index $i_j$ minimizes the ratio  $\nu_{x_j}(s_i)/\lambda_i$. Therefore $t_i$ is actually a regular section of $\mathcal{L}^{\lambda_i}$ and $(\mathcal{L}, t_0, \ldots, t_N)$ is a well defined $\vec{\lambda}$-weighted linear series on $\mathcal{C}$.
		
		Next, by construction, $\nu_{y_j}(t_{i_j}) = 0$. Therefore $(\mathcal{L},t_0,\dots,t_N)$ is a basepoint-free weighted linear series and thus induces a morphism $\mathcal{C} \to \mathcal{P}(\vec{\lambda})$. This morphism agrees with $f$ away from the basepoints $x_j$ so by \cite[Prop. 1.2]{FMN}, it is uniquely isomorphic to $\overline{P}$ up to unique $2$-isomorphism.

		Lastly, suppose $(L,s_0,\dots,s_N)$ is minimal. Then, for every $j$, there exists $s_{\ell_j}$ with $\nu_{x_j}(s_{\ell_j})<\lambda_{\ell_j}$. Since $\frac{\nu_{x_j}(s_{i_j})}{\lambda_{i_j}}\leq \frac{\nu_{x_j}(s_{\ell_j})}{\lambda_{\ell_j}}$, it follows that $\nu_{x_j}(s_{i_j})<\lambda_{i_j}$, and hence, $0\leq a_j<r_j$. Let $\chi_j\colon\mu_{r_j}\to\bG_m$ denote the canonical embedding; since the stabilizers of $\cC$ act trivially on $\pi^*L$, we see the stabilizer $\mu_{r_j}$ at $y_j$ acts on $\mathcal{L}$ via the character $\chi_j^{-a_j}$. Having now verified that the hypotheses of Proposition \ref{prop:tuningstack->min-lin-ser} hold, we see 
		\[
		\xymatrix{
			L'=\pi_*\mathcal{L}(\sum_ja_jy_j)=\pi_*\pi^*L & L\ar[l]_-{\beta}^-{\simeq}
		}
		\]
		where $\beta$ is the adjuncation map; it is an isomorphism since $\cC$ is a tame Deligne--Mumford stack, so we have a canonical isomorphism $\pi_*\cO_\cC\simeq\cO_C$ and $\beta$ is the composition of the canonical isomorphisms $\pi_*\pi^*L\simeq L\otimes\pi_*\cO_\cC\simeq L$. Furthermore, by construction 
		\[
		\beta(s_i)=\pi_*\pi^*(s_i)=\pi_*(t_i\prod_ju_j^{\lambda_ia_j})=s'_i.\qedhere
		\]
	\end{proof}

	We are now ready to prove Theorem \ref{thm:tuning-stacks-and-minimal-linear-series} (\ref{thm:tuning-stacks-and-minimal-linear-series::minimal-lin-ser}), thereby finishing the proof of Theorem \ref{thm:tuning-stacks-and-minimal-linear-series}.
	
	\begin{proof}[{Proof of Theorem \ref{thm:tuning-stacks-and-minimal-linear-series} (\ref{thm:tuning-stacks-and-minimal-linear-series::minimal-lin-ser})}]
		Let $f\colon C\dasharrow\Pcv$ be a rational map. Suppose it is induced by a minimal linear series $(L_1,s_{1,0},\dots,s_{1,N})$ and it is also induced by a minimal linear series $(L_2,s_{2,0},\dots,s_{2,N})$. We prove there exists a canonical isomorphism $\gamma\colon L_1\xrightarrow{\simeq} L_2$ with $\gamma(s_{1,i})=s_{2,i}$.
		
		For $\ell\in\{1,2\}$, let $\pi_\ell\colon\cC_\ell\to C$ be the root stack obtained by taking the $r_{\ell,j}$-th root at $x_j$, where $r_{\ell,j}=r_{\min}(x_j;L_\ell,s_{\ell,0},\dots,s_{\ell,N})$. Let $g_\ell\colon\cC_\ell\to\Pcv$ be the induced representable morphism constructed in Theorem \ref{thm:tuning-stacks-and-minimal-linear-series} (\ref{thm:tuning-stacks-and-minimal-linear-series::tuning-stack}) and let $(\mathcal{L}_\ell,t_{\ell,0},\dots,t_{\ell,N})$ be the corresponding basepoint-free weighted linear series. Let $(L'_\ell,s'_{\ell,0},\dots,s'_{\ell,N})$ be the minimal linear series on $C$ obtained from $(\mathcal{L}_\ell,t_{\ell,0},\dots,t_{\ell,N})$ in Proposition \ref{prop:tuningstack->min-lin-ser}. By Proposition \ref{prop:lin-ser->univ-tuning-stack-map}, we have canonical isomorphisms $\beta_\ell\colon L_\ell\xrightarrow{\simeq} L'_\ell$ with $\beta_\ell(s_{\ell,i})=s'_{\ell,i}$.
		
		Theorem \ref{thm:tuning-stacks-and-minimal-linear-series} (\ref{thm:tuning-stacks-and-minimal-linear-series::tuning-stack}) tells us that $(\cC_\ell,\pi_\ell,g_\ell)$ is a universal tuning stack for each $\ell$, so by universality, there is a canonical isomorphism $h\colon\cC_2\to\cC_1$ and $2$-isomorphism $\alpha\colon g_2\Rightarrow g_1\circ h$ with $\pi_2=\pi_1\circ h$. This yields an isomorphism $\alpha\colon h^*\mathcal{L}_1\xrightarrow{\simeq}\mathcal{L}_2$ with $\alpha(h^*(t_{1,i}))=t_{2,i}$. Applying $(\pi_2)_*$, and using the definitions of $(L'_\ell,s'_{\ell,0},\dots,s'_{\ell,N})$, we obtain an isomorphism $(\pi_2)_*\alpha\colon L'_1\xrightarrow{\simeq} L'_2$ sending $s'_{1,i}$ to $s'_{2,i}$. Composing with the isomorphisms $\beta_\ell$ yields our desired isomorphism $\gamma$.
	\end{proof}
	
	\subsection{The height of a minimal linear series}
	
	In this section we finish the proof of Theorem \ref{thm:tuning-stacks-and-minimal-linear-series} by computing the height of a rational point in terms of the minimal linear series. 
	
	\begin{proof}[{Proof of Theorem \ref{thm:tuning-stacks-and-minimal-linear-series} (\ref{thm:tuning-stacks-and-minimal-linear-series::height})}]
		The stable height $\height^{st}_{\cO(1)}(P) = \deg(\overline{P}^*\cO(1)) = \deg \cL$ by definition. By Proposition \ref{prop:lin-ser->univ-tuning-stack-map}, we have
		$$
		\cL \cong \pi^*L(-\sum a_j y_j)
		$$
		where $(L, s_0, \ldots, s_N)$ is the unique minimal linear series and $$a_j = a_{min}(x_j;L, s_0,\ldots,s_N).$$ Moreover, by minimality, $0 \le a_j < r_j$. By Corollary \ref{cor:push:round:down}, 
		$$
		\pi_*\cL^{\vee} = \pi_*\left(\pi^*L^{\vee}\left(\sum a_j y_j\right)\right) = L^{\vee}
		$$
		and so $\height_{\cO(1)}(P) = -\deg \pi_*\cL^{\vee} = \deg L$. Finally, the local contribution at $x_j$ is given by 
		$$
		\delta_{x_j}(P) = \deg \left(\mathrm{coker}(\pi^*\pi_*\cL^{\vee} \to \cL^{\vee})_{x_j}\right) = \deg \cO_{a_jy_j} = \frac{a_j}{r_j}[k(x_j):k]. 
		$$
	\end{proof}
	
	\subsection{The normalized linear series}\label{sec:normalized}  
	
	In this section, we recast the computation of the universal tuning stack in terms of the \emph{normalized linear series} which will, first, clarify the role of the minimum in the formulas for $r_{min}$ and $a_{min}$ and, second, be easier to work with in families. 
	
	Let $\kappa := \mathrm{lcm}\{\lambda_0, \ldots, \lambda_N\}$ and let $\bar{\lambda}_j := \kappa/\lambda_j$ so that $\lambda_j \bar{\lambda}_j = \kappa$. Then there is a natural map 
	$$
	\Pcv \to \Pc(\underbrace{\kappa, \ldots, \kappa}_{N+1})
	$$
	induced on $T$-points by taking
	$$
	(L, s_0, \ldots, s_N) \mapsto (L, s_0^{\bar{\lambda}_j}, \ldots, s_N^{\bar{\lambda}_j})
	$$
	where $s_i^{\bar{\lambda}_j}$ is a section of 
	$$
	(L^{\otimes \lambda_j})^{\otimes \bar{\lambda}_j} = L^{\kappa}. 
	$$
	\begin{defn} Let $(L, s_0, \ldots, s_N)$ be a $\lambdavec$-weighted linear series. 
		\begin{enumerate}[(1)]
			\item The associated \emph{normalized linear series} is $\mathrm{Span}(s_0^{\blambda_0}, \ldots, s_N^{\blambda_N}) \subset H^0(T, L^{\otimes \kappa})$. 
			\item The \emph{normalized base locus} $\overline{Bs}$ is the scheme theoretic base locus of the normalized linear series:
			$$
			\overline{Bs}(L, s_0, \ldots, s_N) := \bigcap_{j = 1}^N \{s_j^{\blambda_j} = 0\}
			$$
		\end{enumerate}
	\end{defn} 
	
	\begin{prop}\label{prop:normalized:local:conditions}
		Let $(L, s_0, \ldots, s_N)$ be a $\vec{\lambda}$-weighted linear series on a curve $C$ and let $x$ be an indeterminacy point such that $\nu_j = \nu_x(s_j)$. Let $m$ be the multiplicity of $\overline{Bs}$
		at $x$. Then the stabilizer of the universal tuning stack at $x$ is given by $\mu_r$ with
		$$
		r = r_{min} = \frac{\kappa}{\gcd(m, \kappa)}
		$$
		and the character of the $\mu_r$-action on the line bundle $\cL$ on the universal tuning stack is given by $\chi^{-a}$ where 
		$$
		a = a_{min} = \frac{m}{\gcd(m, \kappa)}. 
		$$
	\end{prop}
	\begin{proof} In Theorem \ref{thm:tuning-stacks-and-minimal-linear-series} (\ref{thm:tuning-stacks-and-minimal-linear-series::tuning-stack}), we showed $r = r_{min}(x, L, s_0, \ldots, s_N)$ and $a = a_{min}(x, L, s_0, \ldots, s_N)$. First note that the index which minimizes the ratio $\nu_j/\lambda_j$ is the same as the index which minimizes
		$$
		\kappa \frac{\nu_j}{\lambda_j} = \blambda_j \nu_j = \nu_x(s_j^{\blambda_j}). 
		$$
		On the other hand, the minimum order of vanishing of $\{s_j^{\blambda_j}\}$ is exactly the multiplicity of the base locus of the normalized linear series at $x$. Thus, if $j$ is any index minimizing the ratio $\nu_j/\lambda_j$, then we have
		$$
		\kappa \frac{\nu_j}{\lambda_j} = \blambda_j \nu_j = m. 
		$$
		Now note that for any positive integers $a$ and $b$, $a/\gcd(a,b) = a'$ where $a/b = a'/b'$ is written in lowest terms. In particular, the ratio $a/\gcd(a,b)$ depends only on the fraction $a/b$. Thus, noting that $\kappa/m = \lambda_j/\nu_j$ for $j$ any index minimizing the ratio, we see that
		$$
		r = \frac{\kappa}{\gcd(m, \kappa)}. 
		$$
		Finally, we have that for the same index,
		$$
		a = \frac{\nu_j}{\gcd(\nu_j, \lambda_j)} = \frac{\blambda_j \nu_j}{\blambda_j \gcd(\nu_j, \lambda_j)} = \frac{m}{\gcd(m, \kappa)}. 
		$$
	\end{proof} 
	\begin{rmk} Note that minimality is equivalent to the statement that the for each closed point $x \in C$, the multiplicity $m$ of the normalized base locus satisfies $m < \kappa$. 
	\end{rmk}
	
	We are now ready to express the stacky height with respect to the universal bundle $\cO_{\Pcv}(1)$ in terms of the normalized linear series. 
	
	\begin{cor} \label{cor:height:normalized} Let $K = k(C)$ and $P \in \Pcv(K)$ a $K$-rational point. Let $(L, s_0, \ldots, s_N)$ be the unique minimal weighted linear series associated to $P$. Then 
		$$
		\height(P) = \deg L,
		$$
		$$
		\height^{st}(P) = \frac{1}{\kappa}\deg(C \to \bP^N),
		$$
		and the local contribution at a closed point $x \in C$ is given by
		$$
		\delta_x(P) = \frac{a_x}{r_x} [k(x):k] = \frac{m_x}{\kappa}\deg [k(x):k]
		$$
		where $m_x$ is the multiplicity at $x \in C$ of the normalized base locus $\overline{Bs}$ and $C \to \bP^N$ is the unique morphism extending the composition $C \dashrightarrow \Pcv \to \Pc(\kappa, \ldots, \kappa) \to \bP^N$ with degree measured by $\cO_{\bP^N}(1)$. 
	\end{cor}
	
	
	\section{Moduli of minimal weighted linear series with vanishing conditions} 
	\label{sec:mod-sp-ratpts}
	
	The goal of this section is to construct the moduli space of minimal weighted linear series of degree $n$ on a curve $C$. By Theorem \ref{thm:tuning-stacks-and-minimal-linear-series}, the $k$-points of this moduli space will be canonically in bijection with $k(C)$-points of $\Pcv$ of height $n$. Our second goal is to construct a natural stratification of this moduli space corresponding to the different universal tuning stacks. In the next section, we will identify these strata with the spaces of twisted maps $\cH^{\Gamma}_{d,C}$ for different $d$ and $\Gamma$. 
	
	\subsection{The ambient stack of weighted linear series} 
	
	Fix a smooth projective curve $C/k$ with function field $K$.  
	
	\begin{defn}
		Let $B/k$ be a scheme. A family of $\lambdavec$-weighted linear series on $C$ parameterized by $B$ is the data of $(L, s_0, \ldots, s_N)$ where $L$ is a line bundle on $C_B := C \times B$ and $s_i \in H^0(C \times B, L^{\otimes \lambda_i})$. A morphism $(L, s_0, \ldots, s_N)/B \to (L', s_0', \ldots, s_N')/B'$ is a morphism $\rho : B \to B'$ and an isomorphism $\psi : \rho^*L' \to L$ such that $\psi^{\lambda_i}(s_i') = s_i$. 
	\end{defn}
	
	The category of families of weighted linear series $\mathcal{W}_C(\lambdavec)$ defines a category fibered in groupoids over $\mathrm{Sch}/k$.
	
	\begin{defn} For $\cF$ a coherent sheaf on an algebraic stack $\cX$, the \emph{abelian cone} associated to $\cF$ is 
		$$
		\bV(\cF) := \Spec_{\cX} \Sym^\bullet \cF. 
		$$
	\end{defn}
	
	Note that the formation of $\bV(\cF)$ is compatible with flat base change and satisfies fppf descent. 
	
	\begin{prop} $\cW_C(\lambdavec)$ is a locally of finite type algebraic stack over $k$. Moreover, the natural morphism $\cW_C(\lambdavec) \to \Ppic(C)$ to the Picard stack identifies $\cW_C(\lambdavec)$ with an abelian cone over $\Ppic(C)$. In particular, $\cW_C(\lambdavec)$ is a union 
		$$
		\cW_C(\lambdavec) = \bigsqcup_{n \geq 0} \cW_{n,C}(\lambdavec)
		$$
		of connected stacks of finite type indexed by $n = \deg(\cL_b)$.
	\end{prop}
	
	\begin{proof} Let $\Ppic(C)$ be the Picard stack of $C/k$ and let $\cL$ be the universal line bundle on $\pi : C \times \Ppic(C) \to \Ppic(C)$. The following lemma is a standard corollary of cohomology and base-change (see \cite[Thm. D]{cohbase} for the generalization to algebraic stacks). 
		
		\begin{lem} Let $f : \cX \to \cY$ be a proper morphism of algebraic stacks and let $\cF$ be a quasicoherent sheaf on $\cX$ flat over $\cY$. There exists a quasicoherent sheaf $\cQ$ on $\cY$ such that
			$$
			\Hom_\cY(B, \bV(\cQ)) = H^0(\cX_B, \cF_B). 
			$$
			If $\cF$ is finitely presented then so is $\cQ$. Moreover, if $f_*\cF$ is locally free and its formation commutes with arbitary base change, then $\cQ \cong (f_*\cF)^{\vee}$
		\end{lem}
		
		Applying the lemma to $\pi$ and $\cL^{\otimes d}$ we obtain a coherent sheaf $\cQ^d$ on $\Ppic(C)$ such that for any $B$-point $L \in \Pic(C_B)$, 
		$$
		\Hom_{\Ppic(C)}(B, \bV(\cQ^d)) = H^0(C_B, L^{\otimes d}). 
		$$
		Denoting $\bV(\cQ^d) =: \bV^d$, we conclude that 
		$$
		\Hom_{\Ppic(C)}\left(B, \prod_{i = 0}^N \bV^{\lambda_i}\right) = \bigoplus_{i = 0}^N H^0(C_B, L^{\otimes \lambda_i})
		$$
		and thus
		$$
		\Hom_k\left(B, \prod_{i = 0}^N \bV^{\lambda_i}\right) = \cW_C(\lambdavec)
		$$
		as required. 
		
		Now $\Ppic(C) = \bigsqcup_{n \in \bZ} \Ppic^n(C)$ is a disjoint union of finite connected components parametrizing line bundles of degree $n$. Denoting $\bV^d_n := \bV^d|_{\Ppic^n(C)}$, then $\bV^d|_{\Ppic^n(C)} = \varnothing$ for $n < 0$ so we conclude that
		$$
		\cW_C(\lambdavec)= \bigsqcup_{n \geq 0} \cW_{n,C}(\lambdavec)
		$$
		where $\cW_{n,C}(\lambdavec) = \prod_{i = 0}^N \bV^{\lambda_i}_n$. 
	\end{proof} 
	
	\begin{rmk}\label{rem:d:power} Let $\mu_{d,n} : \Ppic^n(C) \to \Ppic^{dn}(C)$ be the $d^{th}$ power map induced by $L \mapsto L^{\otimes d}$. Letting $\cL$ be the universal line bundle on $C\times\Ppic(C)$, we have $\cL^{\otimes d}|_{C\times\Ppic^n(C)} = \mu_{d,n}^*\cL|_{C\times\Ppic^{dn}(C)}$. Thus, $\bV^d_n = \mu_{d,n}^*\bV^1_{nd}$. 
	\end{rmk}
	
	\begin{rmk}\label{rem:open:sym} The stack $\bV^d_n$ contains an open substack $U^d_n$ given by the complement of the zero section $\Ppic^n(C) \to \bV^d_n$. This is the open subfunctor parameterizing sections $s \in H^0(C_B, L^{\otimes d})$ where $s_b \not \equiv 0$ for all $b \in B$. In particular, the vanishing locus $D = V(s) \subset C_B$ is flat over $B$ for any $B$-point of $U^d_n$ by \cite[\href{https://stacks.math.columbia.edu/tag/00ME}{Tag 00ME}]{Stacks}. This implies that $D$ is a relative effective Cartier divisor and we get a morphism $U^d_n \to \Sym^{dn}(C)$ which makes the square
		\[
		\xymatrix{
			U^d_n\ar[r]\ar[d] & \Sym^{dn}(C)\ar[d]\\
			\Ppic^n(C)\ar[r]^-{\mu_{d,n}} & \Ppic^{dn}(C)
		}
		\]
		cartesian.
		
	\end{rmk}
	
	\begin{rmk}\label{rem:rrbound} When $dn > 2g(C) - 2$, $\pi_*\cL^{\otimes d}|_{C\times\Ppic^n(C)}$ is a vector bundle over $\Ppic^n(C)$ whose formation commutes with base change. In this case, 
		$$
		\bV^d_n = \Spec_{\Ppic(C)} \Sym^\bullet (\pi_*\cL^d|_{C\times\Ppic^n(C)})^\vee.
		$$
	\end{rmk}
	
	\begin{rmk}\label{rem:gerbe:quotient} Let $\pi : \cX \to X$ be a $\mathbb{G}_m$-gerbe and $\cQ$ a quasicoherent sheaf on $\cX$. Then fppf locally (or even \'etale locally since $\bG_m$ is smooth) the gerbe can be trivialized to $\mathcal{B} \bG_{m,X}$. In the case where $\cX$ is a trivial gerbe, $\cQ$ can be identified with a quasicoherent $\bG_m$-sheaf on $X$. Then we have a natural identification
		$$
		\bV_{\cX}(\cQ) = [\bV_{X}(\cQ)/\bG_{m,X}]. 
		$$
		In particular, $\bV_{\cX}(\cQ)$ contains as an open substack the weighted projectivization of the abelian cone $\bV_X(\cQ)$. In particular, this applies to the case $\pi : \Ppic(C) \to \Pic(C)$, where we can achieve such a trivialization of the gerbe after passing to a field extension $k'/k$ with $C(k') \neq \varnothing$. 
	\end{rmk}

	\begin{rmk}\label{rem:rr:gerbe}
		When $C(k) \neq \varnothing$ and $dn > 2g(C) - 2$, then $U^d_n$ is isomorphic to a $\mu_d$-gerbe of over the pullback of the projective bundle $\Sym^{dn}(C) \to \Pic^{dn}(C)$. More generally, if $n\lambda_i > 2g(C) - 2$ for all $i = 0, \ldots, N$, then
		$$
		\prod_{i = 0}^N \bV^{\lambda_i}_n \setminus 0_{\Ppic^n(C)}
		$$
		is isomorphic to a weighted projective bundle over $\Pic^n(C)$ where $\bV^d_n$ carries weight $d$. If $C(k) = \varnothing$, it is instead a possibly nontrivial form of this weighted projective bundle. 
	\end{rmk}
	
	\subsection{Stratifying by normalized base locus} 
	
	As we saw in Section \ref{sec:normalized}, the multiplicity of the normalized base locus controls the local behavior of the universal tuning stack. In this section we use this observation to stratify $\cW_{n,C}(\lambdavec)$ into strata which we will relate to moduli of twisted maps in Section \ref{sec:moduli}.
	
	As $C$ and $\Pcv$ are fixed, we will denote $\cW_{n,C}(\lambdavec)$ by $\cW_n$ for convenience. Over $\cW_n$ we have a universal bundle $\cL$ on $C \times \cW_n$ with sections $s_i \in H^0(C \times \cW_n, \cL^{\otimes \lambda_i})$. Let
	$$
	\overline{\Bs} = \overline{\Bs}(\cL, s_0, \ldots, s_N) = \Bs(\cL^{\otimes \kappa}, s_0^{\bar{\lambda}_0}, \ldots, s_N^{\bar{\lambda}_N})
	$$
	be the normalized base locus of the universal weighted linear series. Now $C \times \cW_n \to \cW_n$ is a projective morphism so we can apply Grothendieck's functorial flattening stratification to $\overline{\Bs} \to \cW_n$ to obtain a stratification of $\cW_n$.

	\begin{prop}\label{prop:flattening} The flattening stratification $\{\cW_n^m\}$ for $\overline{\Bs} \to \cW_n$ is indexed by $m \in \N \cup \infty = \{0, 1, \ldots, m, \ldots, \infty\}$ where
		\begin{enumerate}
			\item $\cW^0_n \cong \Hom_n(C, \Pcv)$ is the open substack parametrizing basepoint-free weighted linear series, i.e. morphisms $C \to \Pcv$, 
			\item $\cW^m_n$ for $0 < m < \infty$ is the substack of families of weighted linear series with normalized base locus finite flat of degree $m$, 
			\item $\cW^m_n = \varnothing$ for $n\kappa < m < \infty$, and
			\item $\cW^\infty_n \cong \Ppic^n(C)$ is the closed substack where where $s_i \equiv 0$ for all $i$. 
		\end{enumerate}
		
	\end{prop}  
	
	\begin{cor} There is a natural morphism $\cW_n^m \to \Sym^m(C)$ for $0 < m < \infty$ sending a $\vec{\lambda}$-weighted linear series to its normalized scheme theoretic base locus $\overline{\Bs}$. 
	\end{cor} 
	
	The symmetric product $\Sym^m(C)$ can be further stratified into strata $\Sym^\mu(C)$ indexed by a partition $\mu \vdash m$. Here a partition $\mu$ is a sequence $\mu_1 \geq \mu_2 \geq \ldots \geq \mu_l > 0$ where 
	$$
	\sum_{i = 1}^l \mu_i = m. 
	$$
	We denote $l := l(\mu)$ the length of the partition and $m = |\mu|$ the size of the partition. Then $\Sym^\mu(C)$ is the stratum parametrizing subschemes of $C$ with $l$ irreducible components of multiplicity $\mu_i$. We say that a finite subscheme of length $m$ parametrized by a point in $\Sym^\mu(C)$ has \emph{profile} $\mu$.
	
	\begin{rmk} While $\bigsqcup_\mu \Sym^\mu(C)$ is at first glance a set theoretic stratification of $\Sym^m(C)$, in fact these strata have a canonical scheme structure as they can be identified with the \emph{locally trivial} Hilbert schemes parametrizing locally trivial flat families of subschemes in $C$ \cite[Sec. 2]{grkar}. 
	\end{rmk}
	
	\begin{defn}
		For any $n \geq 0$ and partition $\mu$, moduli space of \emph{weighted linear series of degree $n$ with normalized base profile $\mu$} is the pullback 
		$$
		\cW^\mu_n := \cW_n^m \times_{\Sym^m(C)} \Sym^{\mu}(C)
		$$
		where $m = |\mu|$. 
	\end{defn}
	
	The $B$-points of $\cW^\mu_n$ are families of weighted linear series $(\cL, s_0, \ldots, s_N)$ on $C_B$ such that the normalized scheme theoretic base locus $\overline{\Bs}(\cL, s_0, \ldots, s_N)$ is finite, flat and locally trivial of profile $\mu$ over $B$.  
	
	\begin{defn} We say a normalized base profile $\mu$ is \emph{minimal} if $\mu_i < \kappa$ for all $i$. 
	\end{defn}
	
	\begin{prop}\label{prop:stack:min}
		There exists an open substack $\cW_{n,C}^{min}(\lambdavec) \subset \cW_{n,C}(\lambdavec)$ parametrizing minimal weighted linear series. Moreover, $\cW_{n,C}^{min}(\lambdavec)$ has a stratification into locally closed strata $\cW^{\mu}_n$ over  minimal base profiles $\mu$. 
	\end{prop}
	
	\begin{proof}
		
		A linear series is minimal if and only if the normalized base multiplicity is strictly less than $\kappa$ for all points on $C$. In particular, it is a condition on the partition $\mu$: namely that $\mu_i < \kappa$ for all $i$. Thus, the stratification $\cW^{min}_n = \sqcup_{\mu \text{ minimal}} \cW^\mu_{n,C}$ is clear, at least set theoretically. To finish the proof, we show that $\cW^{min}_{n,C}$ is open in $\cW_{n,C}$. This follows from Lemma \ref{lem:open} below applied to the universal normalized base locus over $\cW_n \setminus \cW^{\infty}_n$.

		\begin{lem}\label{lem:open} Let $B$ be a locally Noetherian algebraic stack over $k$ with finite inertia. Let $C$ be a smooth curve and let $Z \subset C \times B$ be a subscheme which is finite over $B$. Then for each non-negative integer $d$, the locus of $b \in B$ such that $Z_b$ has multiplicity $ < d$ at each point is open in $B$. 
		\end{lem}
		\begin{proof} Without loss of generality, we may suppose that $k$ is algebraically closed. By flattening stratification and further stratifying by the profile $\mu$ we see that the locus of $b \in B$ where $Z_{b}$ has no point of multiplicity $\geq d$ is constructible. Thus it suffices to show this locus is stable under generalization.   
			
			Toward this end, we may suppose $B = \Spec R$ is the spectrum of a DVR $(R,\mathfrak{m})$ with fraction field $K$ and $Z = \Spec A$ for a finite $R$-algebra $A$. Suppose moreover that $\Spec A/\mathfrak{m}A$ has no point of multiplicity at least $d$. We wish to show that $A \otimes_R K$ also satisfies this condition. Let $z_i$ be the points of the special fiber $\Spec A/\mathfrak{m}A$. By assumption,
			$$
			\dim_{k} (A/\mathfrak{m}A)_{z_i} < d
			$$
			for all $z_i$. Let $z_i'$ be a point of $\Spec(A \otimes_R K)$. Since $Z \to B$ is proper, $z_i'$ specializes to some point, which up reindexing we call $z_i$, lying $\mathfrak{m}$. It suffices to show that $\dim_K (A \otimes_R K)_{z_i'} < d$. Since the question is local, we may replace $A$ by $A_{z_i}$ and suppose without loss of generality that $A$ is local with closed point $z$ and $z'$ is a point of the generic fiber $\Spec A \otimes_R K$. By upper semi-continuity, 
			$$
			\dim_K (A\otimes_R K)_{z'} \le \dim_K(A \otimes_R K) \le \dim_{k} (A/\mathfrak{m}A) < d,
			$$
			thus the generic fiber of $\Spec R$ satisfies the required condition. 
		\end{proof} 
	\end{proof}
	
	\begin{cor} For each $K = k(C)$, $n$ and $\vec{\lambda}$, the $k$-points of the finite type Deligne--Mumford stack $\cW^{min}_{n,C}(\lambdavec)$ are in canonical bijection with $K$-points of $\Pcv$ of stacky height $n$.
		
	\end{cor}
	
	For some applications it is convenient to parameterize the normalized base locus. 
	
	\begin{defn}\label{Smu-definition}
		Let $S_\mu$ be the subgroup of the symmetric group $S_{l(\mu)}$ consisting of permutations $\sigma$ such that $\mu_{\sigma(i)}=\mu_i$ for all $i$.
	\end{defn}
	
	Let $\Conf_l(C)$ be the configuration space of $l$ ordered distinct points on $C$. Then $S_\mu$ naturally acts on $\Conf_l(C)$ and the quotient $\Conf_l(C)/S_\mu$ is isomorphic to the stratum $\Sym^\mu(C)$.
	
	\begin{defn}\label{def:moduli-interp-rnmu} The \emph{space of weighted linear series with parametrized normalized base locus of profile $\mu$} is the fiber product
		$$
		\cR^\mu_n := \Conf_{l(\mu)} \times_{\Sym^{\mu}} \cW^\mu_n. 
		$$
	\end{defn}
	
	A $B$-point of $\cR^\mu_n$ consists of a weighted linear series $(L, s_0, \ldots, s_N)$ of degree $n$ on $C_B$ as well as $l$ disjoint sections $\sigma_i : B \to C_B$ such that the $\overline{\Bs}(L, s_0, \ldots, s_N)$ is the relative effective Cartier divisor $\sum \mu_i \sigma_i$. 
	
	\subsection{Imposing vanishing conditions}
	
	For certain applications such as counting elliptic surfaces with a fixed fiber type, it is useful to further stratify the moduli space of weighted linear series by specifying vanishing conditions on the $s_i$. We will accomplish this by working over the space of embeddings of infinitesimal discs of the form $\Spec k[t]/(t^n)$ into $C$.
	
	Associated to a partition $\mu$ we have a unique isomorphism type of curvilinear Artinian local scheme with profile $\mu$ and residue field $k$: 
	$$
	A_\mu \cong \prod_{i = 1}^l k[x]/(x^{\mu_i + 1}). 
	$$
	We denote by $D_\mu := \Spec A_\mu$. In what follows we view $\mu$ as an ordered tuple $(\mu_1, \ldots, \mu_l)$ with $\mu_i\geq1$, while noting that the isomorphism type of $D_\mu$ and the stratum $\Sym^\mu(C)$ are independent of the choice of ordering. 
	
	\begin{defn}
		Let $G_\mu$ denote the automorphism group $\Aut_k(D_\mu)$. 
	\end{defn}
	
	The following lemma is straightforward from the definitions.
	
	\begin{lem}\label{l:aut-ofD-mu} $G_\mu$ is a finite type group scheme over $k$ with component group $S_\mu$ and whose identity component $G^0_\mu=\prod_{i=1}^\ell(\bG_m\rtimes U_{\mu_i-1})$, where $U_{\mu_i-1}$ is a split unipotent of dimension $\mu_i-1$.
	\end{lem}
	
	\begin{defn} Let $\Emb(D_\mu, C)$ denote the space of embeddings of $D_\mu$ into $C$. 
	\end{defn}
	
	Note that $\Emb(D_\mu,C)$ is a quasi-projective scheme. Indeed, the Hom scheme $\Hom(D_\mu,C)$ is quasi-projective and $\Emb(D_\mu,C)\subset\Hom(D_\mu,C)$ is open by \cite[Prop. 12.93]{Wedhorn}. Note moreover that when $\mu_i = 1$ for all $i = 1, \ldots, l$, then $\Emb(D_\mu, C) = \Conf_l(C)$. We define a partial order on ordered $l$-tuples by $\mu' \geq \mu$ if $\mu'_i \geq \mu_i$ for all $i = 1, \ldots, l$. Note here that $\mu'$ and $\mu$ are not necessarily partitions of the same number.
	
	\begin{lem}\label{lemma:emb:torsor}
		We have the following.
		\begin{enumerate}[(i)]
			\item\label{lemma:emb:torsor::1} $G_\mu$ acts freely on $\Emb(D_\mu, C)$ and the quotient is isomorphic to $\Sym^\mu(C)$.
			\item\label{lemma:emb:torsor::2} The quotient of $\Emb(D_\mu,C)$ by $G_\mu^0$ is isomorphic to the configuration space of $l = l(\mu)$ ordered distinct points $\Conf_{l}(C)$ such that the diagram below commutes. 
			$$
			\xymatrix{\Emb(D_\mu, C) \ar[r]^{G^0_\mu} \ar[rd]_{G_\mu} & \Conf_{l(\mu)}(C) \ar[d]^{S_\mu} \\  & \Sym^\mu (C)}
			$$
			\item\label{lemma:emb:torsor::3} Let $\mu' \geq \mu$. Then $\Emb(D_{\mu'}, C) \to \Emb(D_\mu, C)$ is a torsor for the subgroup $G^{\mu}_{\mu'} \subset G_{\mu'}$ of automorphism of $D_{\mu'}$ which are the identity on $D_{\mu}$.
		\end{enumerate} 
	\end{lem}
	
	\begin{proof} We have a map $\Emb(D_\mu,C) \to \Sym^\mu(C)$ given by taking the image of an embedding $D_\mu \hookrightarrow C$ as a subscheme of $C$. Two embeddings have the same image if and only if they differ by a reparametrization of the source, proving $\ref{lemma:emb:torsor::1}$. For $\ref{lemma:emb:torsor::2}$, note that $\Conf_{l(\mu)}(C) = \Emb(\Spec \prod_{i = 1}^l k, C)$ is the space of embeddings for the partition $1 + \ldots + 1 = l$ with all parts equal to $1$. Denoting this partition by $\mu_0$, it is clear that $G^0_{\mu} = G^{\mu_0}_\mu$ and so the horizontal map in $\ref{lemma:emb:torsor::2}$ is a special case of $\ref{lemma:emb:torsor::3}$. The commutative diagram in $\ref{lemma:emb:torsor::2}$ then follows by the identification of $S_\mu$ with the component group of $G_\mu$ coming from Lemma \ref{l:aut-ofD-mu}. Finally for $\ref{lemma:emb:torsor::3}$, note that the natural map $\Emb(D_{\mu'}, C) \to \Emb(D_\mu, C)$ is induced by composition with the closed embedding $D_\mu \hookrightarrow D_{\mu'}$. Since all embeddings of $D_{\mu'}$ fixing $D_{\mu}$ have the same image in $\Sym^{\mu'}(C)$ (here we are using that $C$ is a smooth curve) then they differ by a reparametrization by $G_{\mu'}$ but this reparametrization must fix the embedding $D_{\mu} \hookrightarrow C$ so it must be an element of $G_{\mu'}^{\mu}$.
	\end{proof} 
	
	\begin{rmk} In fact, the torsor $\Emb(D_{\mu'}, C) \to \Emb(D_\mu,C)$ is a Zariski locally trivial fibration admitting a section and with fiber isomorphic (as a scheme not a group) to 
		$$
		\prod_{\mu_i = 1} \bG_m \times \mathbb{A}^{\mu_i - 2} \times \prod_{\mu_i \neq 1} \mathbb{A}^{\mu_i' - \mu_i}.
		$$
		This is because $G_{\mu'}^\mu$ is an extension of $\bG_m$ by a split unipotent group, hence a special group scheme.
	\end{rmk}
	
	\begin{defn} For $\mu' \geq \mu$ and a fixed embedding $f : D_\mu \to C$, we say that an embedding $f' : D_{\mu'} \to C$ extends $(f,D_\mu)$ if $f'|_{D_\mu} = f$. 
	\end{defn}
	
	We are now ready to define moduli spaces of weighted linear series with vanishing conditions. We will need some notation to talk about imposing both equalities and inequalities on the orders of vanishing of the sections. To accomplish this we will use a pair $\{\nu, T\}$ where $\nu = (\nu_1, \ldots, \nu_l)$ is an ordered $l$-tuple and $T \subset \{1, \ldots, l\}$ to encode the condition 
	
	$$
	\nu_{x_i}(s) \geq \nu_i \text{ with equality if } i \in T.
	$$
	Finally, if $t$ is an integer and $\nu$ is a tuple, denote by $t\nu$ the tuple $(t\nu_1, \ldots, t\nu_l)$.
	
	In the following two definitions, $\mu_i$ encodes the multiplicity of the $i$-th component of the normalized base locus $\overline{\Bs}(L, s_0, \ldots, s_N)$ and $\mu^{(j)}_i$ encodes the vanishing order of $s_j^{\blambda_j}$ along the $i$-th component of $\overline{\Bs}(L, s_0, \ldots, s_N)$. As a result, for each $i$, we want $\mu^{(j)}_i\geq\mu_i$ with equality for some $j$.

	\begin{defn}\label{defn:gamma}
		For each $j$, fix $\nu^{(j)}$ an $l$-tuple and a subset $T_j \subset \{1, \ldots, l\}$ and let $\muj = \blambda_j \nuj$. Suppose that $\muj \geq \mu$ and that for each $i = 1, \ldots, l$, $\mu_i = \min_{j}\{\mu^{(j)}_i\}$. We call $$\gamma =: (\{\nu^{(0)}, T_0\}, \ldots, \{\nu^{(N)}, T_N\})$$ a \emph{tuple of vanishing orders realizing $\mu$} and denote by $\gamma_j$ its $j^{th}$ component $\{\nu^{(j)}, T_j\}$. 
	\end{defn}
	
	\begin{defn} A family of \emph{$\lambdavec$-weighted linear series of degree $n$ on $C$ with parametrized normalized base locus and vanishing order $\gamma$} over $B$ is the data of a weighted linear series $(L, s_0, \ldots, s_N)$ of relative degree $n$ on $C_B$ and disjoint sections $\{\sigma_i : B \to C_B\}_{i = 1}^l$ such that \begin{enumerate}[(1)]
			\item $\overline{\Bs}(L, s_0, \ldots, s_N)$ is a relative effective Cartier divisor over $B$ which is equal to $\sum \mu_i \sigma_i$, 
			\item for each $j = 1, \ldots, N$ and any family of embeddings $D_{\mu^{(j)}} \times B \to C_B$ with image containing $\sum \mu_i \sigma_i$, the restriction $(L^{\otimes \kappa}, s_j^{\blambda_j})|_{D_{\mu^{(j)}}}$ is identically $0$, and
			\item for each $j = 1, \ldots, N$, each $i \in T_j$ and any family of embeddings $D_{\mu^{(j)}_i + 1} \times B \to C_B$ with image $\mu_i \sigma_i$, the restriction $(L^{\otimes \kappa}, s_j^{\blambda_j})|_{D_{{\mu^{(j)}_i}+1}}$ is not identically $0$. 
		\end{enumerate}
	\end{defn}

	Note that $S_\mu$ acts on the set of vanishing orders $\gamma_j = \{\nu^{(j)}, T_j\}$ by permuting the $l$ parts of $\nu^{(j)}$ as well as acting on $T_j \subset \{1, \ldots, l\}$ via the natural action on subsets of $\{1, \ldots, l\}$. 
	
	We prove the Main Theorem of this section.
	
	\begin{thm}\label{thm:main-thm-of-sec4} Fix $\lambdavec$, $C$, $n$ and $\mu = (\mu_1, \ldots, \mu_l)$ as above and let $\gamma$ be a tuple of vanishing conditions realizing $\mu$ as in Definition \ref{defn:gamma}. 
		\begin{enumerate}[(1)]
			\item There exists a separated algebraic stack $\cR^\gamma_{n,C}(\lambdavec)$ with a morphism $\cR^\gamma_{n,C}(\lambdavec) \to \Conf_{l(\mu)}(C)$ whose $B$-points are $\lambdavec$-weighted linear series of degree $n$ with parametrized base locus and vanishing order $\gamma$ over $B$. Moreover, there is a locally closed stratification
			$$
			\bigsqcup_{\gamma} \cR^\gamma_n \to \cR^\mu_n
			$$
			where the union is over all $\gamma$ realizing $\mu$.
			
			\item Let $\cW^\gamma_{n,C}({\lambdavec})=\left[\bigsqcup_{\gamma'}\cR^{\gamma'}_{n,C}/S_\mu\right]$ where the union is taken over all $\gamma'$ in the $S_\mu$-orbit of $\gamma$. Then $\cW^\gamma_{n,C}({\lambdavec})$ is a
			separated algebraic stack and we have an unramified surjective monomorphism
			
			$$
			\bigsqcup_\gamma \cW^\gamma_n \to \cW^\mu_n
			$$
			where the union is over a set of representatives for the $S_\mu$-orbits of the set of $\gamma$ realizing $\mu$.
		\end{enumerate}
	\end{thm}

	
	
	
	\begin{proof}
		For any ordered tuple $\mu' \geq \mu$, we let $\cR^{\mu, \mu'}_n$ denote the pullback
		$$
		\cR^{\mu, \mu'}_n = \Emb(D_{\mu'}, C) \times_{\Conf_{l(\mu)}} \cR^{\mu}_n
		$$
		A point of $\cR^{\mu,\mu'}_n$ consists of a weighted linear series $(L, s_0, \ldots, s_N)$ of degree $n$ on $C_B$, $l$ disjoint sections $\sigma_i : T \to C_B$ such that the $\overline{\Bs}(L, s_0, \ldots, s_N)$ is the relative effective Cartier divisor $\sum \mu_i \sigma_i$ as well as an embedding $D_{\mu'} \times B \to C_B$ which contains $\sum \mu_i \sigma_i$ in its image. Note that $\cR^{\mu, \mu'}_n \to \cR^{\mu}_n$ is an $G^0_{\mu'}$-torsor by Lemma \ref{lemma:emb:torsor}. 
		
		Over $\cR^{\mu,\mu'}_n$ we have a universal linear series $(\cL_{\mu}, s_0, \ldots, s_N)$ with normalized base profile $\mu$ as well as a universal closed embedding
		$$
		D_{\mu'} \times \cR^{\mu, \mu'}_n \hookrightarrow  C \times \cR^{\mu,\mu'}_n.
		$$
		For each $j$, we can restrict $s_j$ along this closed embedding to obtain a line bundle with section $(\cL_{\mu}^{\otimes \lambda_j}, s_j)|_{D_{\mu'}}$. Viewing this data as a morphism
		$$
		D_{\mu'} \times \cR^{\mu, \mu'}_n \to [\mathbb{A}^1/\bG_m]
		$$
		we obtain for each $j$ a morphism to the Hom-stack
		$$
		\alpha_{\mu',j} : \cR^{\mu, \mu'}_n \to \Hom(D_{\mu'}, [\mathbb{A}^1/\bG_m]). 
		$$
		For each $i$, we can also restrict to $D_{\mu'_i} := \Spec k[x]/(x^{\mu'_i+1})$ and obtain a further composition
		$$
		\alpha_{\mu',j,i} : \cR^{\mu,\mu'}_n \to \Hom(D_{\mu_i'}, [\bA^1/\bG_m]) 
		$$
		given by restricting $(\cL_{\mu}^{\otimes \lambda_j}, s_j)$ to the $i^{th}$ jet $D_{\mu_i}$ of $D_{\mu}$. We can also take powers of the line bundle and section for any integer $k$ to obtain morphisms $\alpha_{\mu',j}^k$ and $\alpha_{\mu',j,i}^k$ classifying the restriction of $(\cL_\mu^{\otimes k \lambda_j}, s_j^k)$. 
		
		Let $\tilde{Z}_{\mu',j,i}^k\subset\cR^{\mu,\mu'}_n$ be the closed substack defined by the vanishing locus of $\alpha_{\mu',j,i}^k$ and let $\tilde{U}_{\mu',j,i}^k=\cR^{\mu,\mu'}_n\smallsetminus \tilde{Z}_{\mu',j,i}^k$. These substacks are invariant under the action of the various groups $G^{\mu}_{\mu'}$ and $G^0_{\mu}$ of reparametrizations of the embedding and so by Lemma \ref{lemma:emb:torsor}, they descend to closed and open substacks $Z_{\mu',j,i}^k,U_{\mu',j,i}^k \subset \cR^{\mu}_n$ respectively. The substack $Z_{\mu',j,i}^k$ (resp. $U_{\mu',j,i}^k$) parametrizes those $\lambdavec$-weighted linear series $(L, s_0, \ldots, s_N)$ such that $(L^{\otimes k\lambda_j},s_j)|_{D_{\mu_i'}}$ vanishes identically (resp. does not vanish identically) for all embeddings $D_{\mu'} \hookrightarrow C$ containing the normalized base locus $\sum \mu_i x_i$ in their image. Then $\cR_n^{\gamma}$ is the intersection of $Z_{\mu',j,i}^k$ and $Z_{\mu'',j',i'}^{k'}$ as $\mu',\mu'',j,j',i,i',k,k'$ vary over appropriate values. This implies that $\cR_n^{\gamma}$ is locally closed and that they stratify $\cR_n^\mu$. In particular, $\bigsqcup_{\gamma} \cR^\gamma_n \to \cR^\mu_n$ is an unramified surjective monomorphism. Since these properties can be checked \'etale locally, the same holds for $\bigsqcup_\gamma \cW^\gamma_n \to \cW^\mu_n$.
	\end{proof} 
	
	\begin{rmk}
		The map $\cW_n^\gamma \to \cW_n^\mu$ might fail to be a locally closed embedding. The reason is that the $S_\mu$-invariant union $\bigsqcup_{\gamma'} \cR^{\gamma'}_n$ over the $S_\mu$-orbit of $\gamma$ may not be locally closed in $\cR^\mu_n$ even though $\cR^\gamma_n \to \cR^\mu_n$ is locally closed embedding. However, $\cW_n^\gamma \to \cW_n^\mu$ is locally closed up to a locally closed stratification of the source and target and this is good enough for computing most invariants we will be interested in (e.g. number of points and the class in the Grothendieck ring of stacks). 
	\end{rmk}
	
	\subsection{Defect of minimality} 
	
	In this section we define the defect of minimality $e$ of a weighted linear series. It measures the failure of a linear series to be minimal. 
	
	Let $\mu$ be the normalized base profile. We can divide each part $\mu_i$ by $\kappa$ to obtain $\mu_i = \kappa q_i + r_i$. 
	\begin{defn}
		We define $q(\mu)$ and $r(\mu)$, the quotient and remainder of $\mu$ when divided by $\kappa$, to be the partitions with parts $q_i$ and $r_i$ respectively. 
	\end{defn}
	\begin{defn}
		The minimality defect of $\mu$ is the size of the quotient
		$$
		e = |q(\mu)|. 
		$$
		The minimality defect of a weighted linear series is the minimality defect of its normalized base profile. 
	\end{defn}
	Let $(L,s_0, \ldots, s_N)$ be a weighted linear series with normalized base profile $\mu$. Let $D = \overline{\Bs}(L, s_0, \ldots, s_N) \in \Sym^{\mu}(C)$ be the normalized base locus. Condiser the round-down
	$$
	D' = \left \lfloor \frac{D}{\kappa} \right \rfloor
	$$
	and let $D'' = D - \kappa D'$. The following lemma is clear by construction. 
	
	\begin{lem} \label{div_rem} We can write $D = \kappa D' + D''$ where
		\begin{enumerate}[(1)]
			\item $D', D'' \geq 0$ are effective, 
			\item $D'$ has profile $q(\mu)$ and $D''$ has profile $r(\mu)$, and
			\item $e = \deg(D')$. 
		\end{enumerate}
	\end{lem}
	
	The constructions of $D'$ and $D''$ also make sense in any family of weighted linear series with normalized base profile $\mu$. 
	
	\begin{rmk}
		Since the minimality defect $e$ of a weighted linear series depends only on its normalized base profile, $e$ is constant on each stratum $\cW^{\mu}_n$. In particular, 
		$$
		e : \cW_n \setminus \cW_n^\infty \to \mathbb{Z}
		$$
		is a constructible function and the level sets are unions of $\cW^{\mu}_n$ over all normalized base profiles $\mu$ with the same minimality defect. 
	\end{rmk}

	
	\section{Height moduli on cyclotomic stacks}\label{sec:moduli}

	In this section, we prove the following theorem on existence of height moduli and their relation to moduli of twisted maps $\cH_{d,C}^\Gamma$ (Section \ref{sec:mod-sp-twismaps}).
	
	\begin{thm}\label{thm:height:moduli_body} Let $(\cX, \cL)$ be a proper polarized cyclotomic stack over a perfect field $k$. Fix a smooth projective curve $C/k$ with function field $K = k(C)$ and $n,d \in \mathbb{Q}_{\ge 0}$.

		\begin{enumerate} 
			\item There exists a separated Deligne--Mumford stack $\cM_{n,C}(\cX,\cL)$ of finite type over $k$ with a quasi-projective coarse space and a canonical bijection of $k$-points
			
			$$
			\cM_{n,C}(\cX, \cL)(k) = \left\{ P \in \cX(K) \mid \height_{\cL}(P) = n \right \}.
			$$
			\item There is a finite locally closed stratification 
			$$
			\bigsqcup_{\Gamma,d} 
			\cH^\Gamma_{d,C}(\cX, \cL)/S_\Gamma\to \cM_{n,C}(\cX, \cL)
			$$
			where the union runs over all possible admissible local conditions (Definition \ref{def:local-conditions})
			$$\Gamma = \left(\{r_1, a_1\}, \ldots, \{r_s, a_s\}\right)$$ and degrees $d$ for a twisted map to $(\cX, \cL)$ satisfying
			$$
			n = d + \sum_{i = 1}^s \frac{a_i}{r_i} 
			$$
			and $S_\Gamma$ is a subgroup of the symmetric group as in Definition \ref{def:Sgamma}. 
			
			\item Under the bijection in part (1), each $k$-point of $\cH_{d,C}^\Gamma(\cX, \cL)/S_\Gamma$ corresponds to a $K$-point $P$ with the stable height and local contributions given by
			$$
			\height^{st}_{\cL}(P) = d \quad \quad \quad \left\{\delta_i = \frac{a_i}{r_i}\right\}_{i = 1}^s. 
			$$
		\end{enumerate}
	\end{thm}

	\begin{rmk} When $(\cX,\cL) = \left(\Pcv, \cO_{\Pcv}(1)\right)$, we will see in the next section that the moduli space $\cM_{n,C}(\cX,\cL)$ is the space $\cW_{n,C}^{min}(\lambdavec)$ of minimal weighted linear series from Proposition \ref{prop:stack:min}; the stratification into spaces of twisted maps $\cH^\Gamma_{d,C}$ corresponds to the stratification of $\cW^{min}_{n,C}$ into $\cW^\mu_{n,C}$ for minimal partitions $\mu$. 
	\end{rmk}
	
	\begin{rmk}\label{rmk:Mnccyclotomic}
		In fact, the proof of Theorem \ref{thm:height:moduli} will show that $\cM_{n,C}(\cX, \cL)$ is a cyclotomic stack by proving that $\cM_{n,C}(\cX, \cL)\subset\cM_{n,C}(\Pcv,\cO_{\Pcv}(1))$ is a closed substack and that $\cM_{n,C}(\Pcv,\cO_{\Pcv}(1))$ is cyclotomic.
	\end{rmk}
	
	As an immediate corollary we obtain the following Northcott property. 
	
	\begin{cor}\label{cor:northcott} Fix $(\cX, \cL)$ and $C$ as above and suppose $k$ is a finite field. Then for any $B > 0$, the set
		$$
		\{ P \in \cX(K) \mid \height_{\cL}(P) \le B \}
		$$
		is finite. 
	\end{cor}
	
	\begin{rmk} It is not hard to see from the construction of $\cM_{n,C}(\cX, \cL)$ that it is compatible with base change in the following sense: if $k'/k$ is a field extension with $k'$ perfect and we denote $(\cX', \cL') = (\cX, \cL) \times_k k'$ and $C' = C \times_k k'$, then 
		$$
		\cM_{n,C}(\cX, \cL) \times_k k' \cong \cM_{n,C'}(\cX', \cL')
		$$
		as stacks over $k'$. 
		
	\end{rmk}
	
	\subsection{The weighted projective case}\label{subsec:modwps}
	
	We begin by proving Theorem \ref{thm:height:moduli} for 
	$$
	(\cX, \cL) = \left(\Pcv, \cO_{\Pcv}(1)\right).
	$$ 
	The key input is the correspondence in Theorem \ref{thm:tuning-stacks-and-minimal-linear-series::height} and our main task is to check that this correspondence holds in families.

	\begin{proof}[{Proof of Theorem \ref{thm:height:moduli} for weighted projective stacks}]\label{proof:wps}
		
		We will check that $\cW^{min}_{n,C}(\lambdavec)$, the space of minimal linear series on $C$, satisfies the properties of the theorem. 
		
		For part $(1)$, note that by definition a $k$-point of $\cW^{min}_{n,C}$ is a minimal $\lambdavec$-weighted linear series $(L, s_0, \ldots, s_N)$ on $C$ with $\deg L = n$. This induces a rational map $C \dashrightarrow \Pcv$ which gives us a rational point $P \in \Pcv(K)$. By Theorem \ref{thm:tuning-stacks-and-minimal-linear-series} (\ref{thm:tuning-stacks-and-minimal-linear-series::height}), $\height_{\cO(1)}(P) = n$. On the other hand, a rational point $P \in \Pcv(K)$ of height $n$ spreads out to a unique minimal weighted linear series by Theorem \ref{thm:tuning-stacks-and-minimal-linear-series} (\ref{thm:tuning-stacks-and-minimal-linear-series::minimal-lin-ser}) and the degree of the minimal weighted linear series is $n$ by Theorem \ref{thm:tuning-stacks-and-minimal-linear-series} (\ref{thm:tuning-stacks-and-minimal-linear-series::height}). This gives us the required canonical bijection
		$$
		\cW_{n,C}^{min}(\lambdavec)(k) \cong \left\{P \in \Pcv(K) \mid \height_{\cO(1)}(P) = n\right\}. 
		$$
		
		For part $(2)$, we begin with the following lemma. Let $\kappa := \mathrm{lcm}\{\lambda_0, \ldots, \lambda_N\}$ and $\bar{\lambda}_j := \kappa/\lambda_j$ as usual. 
		
		\begin{lem}\label{lem:bijection} Suppose $\kappa > 1$. Then the map 
			$$
			m \mapsto \left( \frac{\kappa}{\mathrm{gcd}(m,\kappa)}, \frac{m}{\mathrm{gcd}(m,\kappa)}\right)
			$$
			induces a bijection from the set $\{1,\ldots, \kappa - 1\}$ to the set 
			$$
			\left\{(r,a) : 1 \le a < r, r|\kappa, \mathrm{gcd}(r,a) = 1\right\}
			$$
			
		\end{lem}
		\begin{proof} 
			By construction, the map lands in the required set so it suffices to construct an inverse. The inverse is simply given by
			$$
			(r,a) \mapsto \frac{\kappa a}{r}
			$$
			which is an integer in $\{1, \ldots, \kappa - 1\}$ by the conditions on $(r,a)$.
		\end{proof}
		
		We have the following immediate corollary. 
		
		\begin{cor}\label{cor:bijection} For each $l \geq 1$, there is a bijection between the set of $l$-tuples of admissible local conditions $\Gamma = \left(\{r_1, a_1\}, \ldots, \{r_l, a_l\}\right)$ for a representable twisted map to $\Pcv$ and the set of minimal ordered partitions $\mu$ with $l$ parts. 
		\end{cor}
		\begin{proof} 
			A pair $(r,a)$ is admissible if and only if $r | \kappa$ while a partition $\mu$ is minimal if and only if $\mu_i < \kappa$ for each $i$. Thus the required bijection is the $l$-fold product of the bijection in Lemma \ref{lem:bijection}.
		\end{proof} 
		
		Now by Proposition \ref{prop:stack:min}, $\cW^{min}_n$ has a finite locally closed stratification into $\cW^\mu_n$ where $\mu$ runs over all minimal partitions with $|\mu| \le \kappa n$, where the inequality holds by Proposition \ref{prop:flattening} (3). Fix one such partition $\mu$ and suppose $l(\mu) = l$. By the above corollary, there exists a unique $l$-tuple of local conditions $\Gamma = (\{r_1, a_1\}, \ldots, \{r_l, a_l\})$ corresponding to the partition $\mu$. We will show that $\cH^\Gamma_d/S_\Gamma \cong \cW^\mu_n$ where $n$ and $d$ are related as in the statement of Theorem \ref{thm:height:moduli}. More precisely, we will show the following.
		\begin{prop}\label{prop:iso-HGammad-Rmun}
			$$\cH^\Gamma_d \cong \cR^\mu_n$$
			as stacks over $\Conf_l(C)$.
		\end{prop}
		Part $(2)$ of the theorem then follows.  
		
		\begin{proof}[Proof of Proposition \ref{prop:iso-HGammad-Rmun}]In one direction, let $(\cC \to B, f : \cC \to \Pcv, \{\Sigma_i\}_{i = 1}^l)$ be a family of twisted maps from $C$ to $\Pcv$ of type $\Gamma$ and degree $d$. Let $\cL = f^*\cO_{\Pcv}(1)$ and let $t_j \in H^0(\cC, \cL^{\otimes \lambda_j})$ be the pullback of the canonical section of $\cO_{\Pcv}(\lambda_j)$. Let $u_i$ be the section of $\cO_{\cC}(\Sigma_i)$ cutting out the marked $\mu_{r_i}$-gerbe $\Sigma_i$. By construction the line bundle
			$$
			\cL\left(\sum_{i = 1}^l a_i \Sigma_i\right) 
			$$
			carries the trivial character at each point of the twisted curve $\cC/B$ and thus is the pullback of a line bundle $L$ of degree $n = d + \sum_{i = 1}^l \frac{a_i}{r_i}$ along the coarse map $\pi : \cC \to C_B$. Moreover, as in the proof of Proposition \ref{prop:tuningstack->min-lin-ser}, $s_j := \pi_*\left(t_j \prod_i u_i^{\lambda_j a_i}\right)$ is a section of $L^{\otimes \lambda_j}$. Thus, $(L, s_0, \ldots, s_N)$ is a family $\lambdavec$-weighted linear series on $C_B \to B$. Moreover, since $\cC$ is tame, the coarse moduli map $\pi_*$ is exact and its formation commutes with base-change. In particular, the construction taking the twisted map to the weighted linear series is functorial and by checking on fibers over $b\in B$ and applying Proposition \ref{prop:tuningstack->min-lin-ser}, we conclude that $(L, s_0, \ldots, s_N)$ is a family of minimal linear series. Next, applying Proposition \ref{prop:normalized:local:conditions} and Corollary \ref{cor:bijection}, we conclude that $(L, s_0, \ldots, s_N)$ has base profile $\mu$ along the marked sections $\sigma_i = \pi_*\Sigma_i$. Putting this all together, we see that $(C_B \to B, \sigma_i, L, s_0, \ldots, s_N)$ is a $B$-point of $\cR^\mu_n$ as required. 
			
			On the other hand, suppose $(C_B \to B, \sigma_i, L, s_0, \ldots, s_N)$ is a $B$-point of $\cR^\mu_n$ and let $\Gamma = \left(\{r_1, a_1\}, \ldots, \{r_l, a_l\}\right)$ be the local conditions corresponding to $\mu$ via Corollary \ref{cor:bijection}. Let $\cC$ be the iterated root stack of $C_B$ along $\sigma_i$ to order $r_i$ for each $i = 1, \ldots, l$ and let $\Sigma_i$ denote the $\mu_{r_i}$ gerbe corresponding to $\sigma_i$. Let $\cL$ on $\cC$ be the line bundle
			$$
			\cL = \pi^*L\left(-\sum_{i=1}^l a_i\Sigma_i\right). 
			$$
			Then $\cL$ is uniformizing with relative degree $d$ over $B$ and local conditions $\Gamma$ by construction. Moreover, 
			$$
			t_j = \frac{\pi^*s_j}{\prod_i u_i^{\lambda_j a_i}} \in H^0(\cL^{\otimes \lambda_j})
			$$
			and by checking on fibers over $b \in B$, we conclude that $(L, t_0, \ldots, t_N)$ is a basepoint-free weighted linear series by Proposition \ref{prop:lin-ser->univ-tuning-stack-map}. The formation of root stacks commutes with base change and so this operation is functorial. Moreover these two operations are clearly inverses and so we have an isomorphism $\cH^\Gamma_d \cong \cR^\mu_n$ as claimed. 
			
			Finally, it follows from Corollary \ref{cor:height:normalized} that $\height^{st}(P) = d$ and the local contributions to height are $\frac{a_i}{r_i}$ for $P \in \Pcv(K)$ corresponding to a $k$-point of $\cH^\Gamma_d$. 
		\end{proof}

		This concludes the proof of Theorem \ref{thm:height:moduli} for the case of weighted projective stacks.
	\end{proof} 
	
	\subsection{The general cyclotomic case}
	
	In this section we reduce Theorem \ref{thm:height:moduli} for general proper polarized cyclotomic stacks $(\cX, \cL)$ to the case of $\Pcv$. 
	
	\begin{proof}[Proof of Theorem \ref{thm:height:moduli} in general]
		By Proposition \ref{prop:polarizing:embedding}, there exists a representable closed embedding $\cX \subset \Pcv$ for some $\lambdavec$ such that $\cL = \cO_{\Pcv}(1)|_\cX$. Now $\cX(K)$ is a sub-groupoid of $\Pcv(K)$. For any $P \in \cX(K)$, there exsts a unique universal tuning stack with a representable morphism $\cC \to \cX$ and generic point $P$. The composition $\cC \to \Pcv$ is a representable morphism with generic point $P$ viewed as a point in $\Pcv(K)$. Thus $\cC$ is the universal tuning stack for $P \in \Pcv(K)$ by uniqueness and $\cL|_{\cC} = \cO_{\Pcv}(1)|_{\cC}$ by functoriality. Thus we have equalities
		\begin{align*}
			&\height_{\cL}(P) = \height_{\cO_{\Pcv}(1)}(P), \ \height^{st}_{\cL}(P) = \height^{st}_{\cO_{\Pcv}(1)}(P), \text{ and } \\ & \delta_x(P, \cL) = \delta_x(P, \cO_{\Pcv}(1)) \text{ for all }x \in C. 
		\end{align*}
		
		To finish the proof we show that the condition for a family of $\lambdavec$-weighted linear series to have generic point mapping to $\cX \subset \Pcv$ is closed inside $\cW^{min}_n$ and thus cuts out the stack $\cM_{n,C}(\cX, \cL)$ inside $\cW^{min}_n$. To do this, we first consider the twisted map strata $\cH^\Gamma_{d,C}(\Pcv, \cO(1))/S_\Gamma$. Since $\cX\subset\Pcv$ is closed, it is a closed condition in the stack of twisted maps for a $\cC\to\Pcv$ to map to $\cX$, 
		and hence the condition for a family of $\lambdavec$-weighted linear series to have generic point mapping to $\cX \subset \Pcv$ is constructible inside $\cW^{min}_n$. It therefore remains to prove that this locus is stable under specialization. To this end, let $R$ by a DVR with fraction field $K'$ and residue field $k'$, and let $(L,s_0,\dots,s_N)$ be a minimal $\lambdavec$-weighted linear series on $C_R$. Let $U\subset C_R$ be the complement of the base locus, which is a non-empty open subset. Then we have a morphism $f\colon U\to\Pcv$ and we are assuming that the generic point of $U_{K'}$ maps to $\cX$. Since $U\to\Spec R$ is flat, the generic point of $U_{K'}$ specializes to the generic point of $U_{k'}$. Since $\cX\subset\Pcv$ is closed, this implies that the generic point of $U_{k'}$ also maps to $\cX$. We have therefore simultaneously shown that $\cM_{n,C}$ is a closed substack of $\cW^{rat}_n$ and that it comes equipped with the required stratification as in part (2).
	\end{proof}
	\begin{rmk}
		Our construction of $\cM_{n,C}$ gives us a closed substack of $\cW^{min}_n$ with reduced induced structure. This is good enough for applications to point counting, computing classes in the Grothendieck ring, and homological/representation stability. However, it is an interesting question if $\cM$ itself carries a natural moduli interpretation, as opposed to just its $k$-points.
		
	\end{rmk}
	
	\begin{rmk}\label{rem:log}
		Under the isomorphism $\cH^{\Gamma}_d\cong\cR^{\mu}_n$ in Proposition \eqref{prop:iso-HGammad-Rmun}, the stratification of $\cR^{\mu}_n$ constructed in Theorem \ref{thm:main-thm-of-sec4} yields a stratification of $\cH^{\Gamma}_d$. The stratum $\cR^\gamma_n$ corresponds to stratifying by tangency conditions for the marked gerbes along the toric boundary of $\Pcv$. This suggests that if one replaces $\Pcv$ by a smooth Deligne--Mumford toric stack, then one may construct analogues of the strata $\cR^\gamma_n$ in terms of logarithmic twisted maps.
	\end{rmk}
	
	\begin{defn}
		We say that a rational point $P \in \cX(K)$ is \emph{isotrivial} if the stable height is zero $\height^{st}(P) = 0$. 
	\end{defn}
	\begin{rmk}
		This is equivalent to the usual notion of isotrivial since $\height^{st}(P) = \deg(\cC \to \cX)$ measured with respect to $\cL$. Since $\cL^{\otimes M}$ descends to an ample line bundle on the coarse space $X$, the stable height is zero if and only if the morphism on coarse moduli spaces $C \to X$ is constant. Note in particular that the condition of being isotrivial is independent of the choice of polarizing line bundle $\cL$. 
	\end{rmk}
	
	\begin{prop} There is a closed substack $\cM_{n,C}^{iso}(\cX) \subset \cM_{n,C}(\cX, \cL)$ parametrizing isotrivial $K$-points. 
	\end{prop}
	\begin{proof}
		By Remark \ref{rmk:Mnccyclotomic}, $\cM_{n,C}(\cX,\cL)$ is a closed substack of $\cM_{n,C}(\Pcv,\cO_\Pcv(1))$. It is immediate from the definitions that $P\in\cX(K)$ is isotrivial if and only if $P\in\Pcv(K)$ is isotrivial. As a result, $\cM^{iso}(\cX)$ is a closed substack of $\cM^{iso}(\Pcv)$, and hence it suffices to prove the result for $\Pcv$.
		
		Given $P\in\Pcv(K)$, let $(L, s_0, \ldots, s_N)$ be the corresponding $\lambdavec$-weighted linear series and let $\cC\to\Pcv$ be the universal stack with local conditions $$\Gamma=(\{r_1,a_1\},\dots,\{r_\ell,a_\ell\})$$ Then $P$ is isotrivial if and only if 
		$$
		\height_{\cO(1)}(P)=\sum_{i=1}^\ell\frac{a_i}{r_i}=\frac{1}{\kappa}\deg(\overline{Bs}(L, s_0, \ldots, s_N)).
		$$
		In other words, $P$ is isotrivial if and only if its normalized base locus has the largest degree possible. The locus of such points $P$ is given by the smallest stratum in the stratification from Proposition \ref{prop:flattening}, and hence closed.
	\end{proof}
	
	Finally, we give a simple application to bounding points of bad reduction. A typical situation is that $\cX$ is the compactification of a moduli space $\cU$ of smooth objects and $\cX \setminus \cU = \cD$ is a Cartier divisor parametrizing singular objects. We call $\cD$ the boundary divisor. This is the case for example for the moduli space of elliptic curves $\cM_{1,1} \subset \overline{\cM}_{1,1}$ with $\cD = \infty$. 
	
	\begin{defn}\label{def:badred}
		Let $(\cX, \cL)$ be a polarized cyclotomic stack and let $\cD \subset \cX$ be a $\Q$-Cartier divisor with open complement $\cU$ and let $P \in \cU(K)$ be a rational point. Let 
		\[
		\xymatrix{\cC \ar[d]^{\pi} \ar[r]^f & \cX \\ C & }
		\]
		be the universal tuning stack of $P$ and let $\Sigma_i$ be the marked gerbes of the twisted curve $\cC$. We say $P$ has \emph{bad reduction at} $x \in C$ if $x$ is in the image under $\pi$ of the locus
		$$
		f^{-1}(\cD) \cup \bigsqcup_{i = 1}^l \Sigma_i.
		$$
		
	\end{defn}
	
	\begin{prop}
		Fix $(\cX, \cL, \cD)$ as in Definition \ref{def:badred}. Then there is a uniform bound on the number of points of bad reduction depending only on the height of the rational point. More precisely, there is a function $N_{\cX, \cL, \cD}(n)$ such that for all $K=k(C)$ and all $P \in \cU(K) \subset \cX(K)$,
		$$
		\#\{x \in C \mid P \text{ has bad reduction at }x\} \le N_{\cX, \cL, \cD}(\height_{\cL}(P)). 
		$$
	\end{prop}
	\begin{proof}
		Fix $K=k(C)$ and $P\in\cU(K)$. Let $n=\height_{\cL}(P)$ and choose $M$ such that $\cL^{\otimes M}=\pi^*L$, where $\pi\colon\cX\to X$ is the coarse space map and $L$ is ample on $X$. Then we obtain a map $\bar{f}\colon C\to X$ and we have $d=\frac{1}{M}\deg(\bar{f})$, where degree is measured with respect to $L$. Letting $\Gamma=(\{r_1,a_1\},\dots,\{r_\ell,a_\ell\})$ be the local conditions of the universal tuning stack of $P$, we have $n = d + \sum_{i = 1}^\ell \frac{a_i}{r_i}$. Since the map $f\colon\cC\to\cX$ from the universal tuning stack is representable, we see all $r_i$ are bounded. Since $n$ is also bounded by assumption, we see $d$ and the number of stacky points $\ell$ are bounded. As a result, we have bounded the number of points $P$ of bad reduction which are in the image of $\bigsqcup_{i = 1}^l \Sigma_i$.
		
		It remains to show that the number of points in the image of $f^{-1}(\cD)$ are also bounded. Let $D$ be the coarse space of $\cD$, so that $D\subset X$ is a divisor. Since $\cD$ is Cartier, then $mD$ is Cartier for some fixed $m > 0$. Then since $C$ is not contained in $D$, the number of points in the image of $f^{-1}(\cD)$ is bounded by the intersection number $C\cdot (mD)$. To bound this quantity, we consider the moduli space of degree $Md$ maps from curves to $X$. This space is quasi-projective so has finitely many irreducible components $Z_1,\dots,Z_m$. For each $i$, there is a non-empty open set $U_i\subset Z_i$ such that for all points $\bar{f}\colon C'\to X$ of $U_i$, the quantity $C'\cdot (mD)$ is a constant $N_i$. Since $C\cdot (mD)$ is bounded above by $\max_i N_i$, we have therefore bounded the number of points of bad reduction.
	\end{proof}

	\section{Rational curves on weighted projective stacks}\label{sec:ex}\label{sec:ex:P1}

	Fix $\lambdavec$, $\kappa$ and $\bar{\lambda}_j$ as in Section \ref{sec:normalized}. In this section we explicitly describe the height moduli for $k(t)$-points on $\Pcv$. By Theorem \ref{thm:tuning-stacks-and-minimal-linear-series}, this is equivalent to describing the moduli of weighted linear series on $C=\bP^1_{k}$. 
	
	In this case, $\Ppic^n(C) \cong \mathcal{B} \bG_m$ is a trivial gerbe over $\Pic^n(C) = \Spec k$. The maps $\mu_{d,n} : \Pic^n(C) \to \Pic^{nd}(C)$ (Remark \ref{rem:d:power}) are the identity and by Remark \ref{rem:rr:gerbe}, we can view $\cW_n \to \mathcal{B} \bG_m$ as the abelian cone associated to the following $\bG_m$-representation. Let $V^d_n = H^0(\bP^1, \cO(dn))$ equipped with a $\bG_m$ action of weight $d$ and let 
	$$
	\cQ = \bigoplus_{j = 0}^N (V^{\lambda_j}_n)^{\vee}
	$$
	viewed as a $\bG_m$-equivariant locally free sheaf on $\Spec k$. For a vector space $V$, the abelian cone $\Spec \Sym^\bullet (V^{\vee})$ is simply $V$ itself viewed as an affine space. Thus
	$$
	\cW_n = [\bV_{\Spec k}(\cQ)/\bG_m] \cong  \left[\bigoplus_{j = 0}^N V^{\lambda_j}_n/\bG_m\right] \to \mathcal{B} \bG_m
	$$
	is simply the stack of tuples $(f_0, \ldots, f_N)$ of binary forms where $\deg f_i = \lambda_i \cdot n$ modulo the action of $\bG_m$ given by
	$$
	t \cdot (f_0, \ldots, f_N) = (t^{\lambda_0}f_0, \ldots, t^{\lambda_N}f_N). 
	$$
	
	The stratum $\cW_n^\infty \subset \cW_n$ is the zero section of the projection 
	$$
	\left[\bigoplus_{j = 0}^N V^{\lambda_j}_n/\bG_m\right] \to \mathcal{B}\bG_m
	$$
	and thus $\cW_n \setminus \cW_n^\infty$ is identified with the ambient weighted projective stack
	$$
	\Pc\left(\bigoplus_{j = 0}^N V^{\lambda_j}_n\right).
	$$
	Fixing coordinates, we can think of this concretely as the weighted projective stack
	$$
	\Pc(\underbrace{\lambda_0, \ldots, \lambda_0}_{n\lambda_0 + 1}, \ldots, \underbrace{\lambda_N, \ldots, \lambda_N}_{n\lambda_N + 1})
	$$
	parameterizing the coefficients of the homogeneous polynomials $(f_0, \ldots, f_N)$ (see \cite[Sec. 4.1]{JJ}). 
	
	\subsection{Stratifying by defect of minimality}\label{sec:str_min}
	
	Fix a height $n$ and $0 \le e \le n$. In this section, we use the minimality defect to stratify the complement of $\cW_n^{min}$ inside $\cW_n \setminus \cW_n^{\infty}$ into strata corresponding to minimal weighted linear series of smaller height. 
	
	First we construct a map
	$$
	\psi_{n,e} : \cW_{n-e}^{min} \times \bP(V_e^1) \to \cP\left(\bigoplus_{i = 0}^N V_n^{\lambda_j}\right)
	$$
	whose image is exactly those weighted linear series with minimality defect $e$. We have a map of affine spaces
	$$
	\left(\bigoplus_{j = 0}^N V_{n-e}^{\lambda_j}\right) \oplus V_e^1 \to \bigoplus_{j = 0}^N V_n^{\lambda_j}
	$$
	$$
	((f_0, \ldots, f_N), h) \mapsto (f_0h^{\lambda_0}, \ldots, f_N h^{\lambda_N}). 
	$$
	This map is equivariant for the homomorphism $\mathbb{G}_m^2 \to \mathbb{G}_m$, $(s,t) \mapsto st$ and thus descends to the required map $\psi_{n,e}$ after taking quotients and passing to open substacks. 
	
	\begin{prop}\label{prop:loc:closed}
		Up to taking a locally closed stratification of the source, $\psi_{n,e}$ is a locally closed embedding. 
	\end{prop}
	\begin{proof} First we stratify $\cW_{n-e}^{min}$ by normalized base profile and $\bP(V_e^1) = \Sym^e(\bP^1)$ by partition type to obtain a stratification 
		$$
		\cW_{n-e}^{min} \times \bP(V_e^1) = \bigsqcup_{\mu'',\mu'} \cW_{n-e}^{\mu''} \times \Sym^{\mu'}\bP^1
		$$
		where the union is over minimal base profiles $\mu''$ for weighted linear series of height $n-e$ and $\mu'$ partitions of $e$. On the other hand, we stratify the target by normalized base profile into strata $\cW_n^{\mu}$. 
		
		Let $\mu$ be a base profile with minimality defect $e$ and let $\mu'' = r(\mu)$ and $\mu' = q(\mu)$ be the remainder and quotient respectively. First note that 
		$$
		\psi_{n,e}(\cW_{n-e}^{\mu''} \times \Sym^{\mu'}\bP^1) \subset \cW^{\mu}_n. 
		$$
		Indeed, on the level of points, the normalized base locus of $\psi_{n,e}((f_0, \ldots,f_N),h)$ is
		$$
		\kappa D' + D''
		$$
		where $D'' = \overline{\mathrm{Bs}}|f_0, \ldots, f_N|$ and $D' = V(h)$. Then $\kappa D' + D''$ has profile $\mu$ by definition of $\mu', \mu''$.  
		
		On the other hand, suppose $(f_0, \ldots, f_N)$ is a weighted linear series with base profile $\mu$. We can write its normalized base locus as
		$$
		\overline{\mathrm{Bs}} = D = \kappa D' + D''
		$$
		as in Lemma \ref{div_rem} where $D' \in \Sym^{\mu'}\bP^1$ is the vanishing of a polynomial $h$ of degree $e$ and type $\mu'$. 
		
		By definition of $\cW^\mu_n$, this decomposition of the normalized base locus holds when $(\cL, f_0, \ldots, f_N)$ is the universal linear series over the whole stratum. Then $h^{\kappa}$ divides $f_j^{\bar{\lambda}_j}$ for all $j$ and so 
		$$
		f_j = f_j'h^{\lambda_j}
		$$
		for some $f_j' \in H^0(\cP^1 \times \cW^{min}_{n-e}, \cL(-D')^{\lambda_j})$. Then the data
		$$
		(\cL(-D'), f_0', \ldots, f_N')
		$$
		is a weighted linear series of height
		$$
		\deg_{\cW_n^\mu} \cL(-D') = n - e
		$$
		and normalized base profile $\mu''$. In particular, this is a minimal weighted linear series over $\cW_n^{\mu}$ which induces a morphism 
		$$
		\cW_n^{\mu} \to \cW_{n-e}^{\mu''}. 
		$$
		Combining this with the morphism $\cW_n^{\mu} \to \Sym^{\mu'}\bP^1$ that outputs $D' = V(h)$, we have a morphism 
		$$
		\cW_n^{\mu} \to \cW_{n-e}^{\mu''} \times \Sym^{\mu}\bP^1
		$$
		which by construction is inverse to the restriction
		$$
		\psi_{n,e} : \cW^{\mu''}_{n-e} \times \Sym^{\mu'}\bP^1 \to \cW_n^{\mu}. 
		$$
	\end{proof} 
	
	\begin{cor} \label{cor:def:strat}
		The disjoint union of $\psi_{n,e}$
		$$
		\psi_n : \bigsqcup_{e = 0}^n \cW_{n-e}^{min} \times \mathbb{P}(V_e^1) \to \cP\left(\bigoplus_{i = 0}^N V_n^{\lambda_j}\right)
		$$
		is an isomorphism after stratifying the source and target. 
	\end{cor}
	
	\begin{proof} By Proposition \ref{prop:loc:closed} and its proof, after further stratifying the source and target, $\psi_n$ restricts to an isomorphism of strata
		$$
		\cW^{\mu''}_{n-e} \times \Sym^{\mu'}\bP^1 \to \cW^{\mu}_n 
		$$
		where $\mu$ is a normalized base profile with quotient $q(\mu) = \mu'$ and $r(\mu) = \mu''$. It suffices to show that each generic point of a stratum of the source maps to a unique generic point of a stratum in the target. The map on generic points of strata is injective as $\mu'' = r(\mu)$ and $\mu' = q(\mu)$. On the other hand, let $(f_0, \ldots, f_N)$ be the weighted linear series parametrized by the generic point of $\cW_n^{\mu}$. Then as above we write 
		$$
		\overline{\mathrm{Bs}}|f_0, \ldots, f_N| = \kappa D' + D''
		$$
		where $D' \in \Sym^{q(\mu)}\bP^1 \subset \bP(V_e^1)$ is the vanishing of some polynomial $h$ of degree $e = e(\mu)$ and $f_j^{\bar{\lambda}_j} = G_j h^{\kappa}$ where $D'' = \mathrm{Bs}|G_0, \ldots, G_N|$. By unique factorization, it follows that $G_j = g_j^{\bar{\lambda_j}}$ and $f_j = g_jh^{\lambda_j}$ for some $g_j$ of degree $\lambda_j(n-e)$. Thus $(g_0, \ldots, g_N)$ is a weighted linear series of degree $n-e$ with normalized base profile $r(\mu)$ which is minimal by construction and
		$$
		(f_0, \ldots, f_N) = \psi_{n,e}((g_0, \ldots, g_N),h).
		$$
		
	\end{proof}

	\subsection{Stratifying by vanishing order}\label{sec:str_van}
	Fix a partition $\mu$ with $l$ parts and $\gamma$ a tuple of vanishing orders realizing $\mu$ as in Definition \ref{defn:gamma}. We have natural locally closed subvarieties
	$$
	W_n^{min},\ W_n^\mu \subset W_n := \bigoplus_{j = 0}^N V_n^{\lambda_j} \setminus 0. 
	$$
	Here $W_n^{min}$ is the set
	$$
	\left\{(f_0, \ldots, f_N) \mid \min_{j}\left(\bar{\lambda}_j\nu_x(f_j)\right) < \kappa \text{ for all }x \in \bP^1\right\}
	$$
	which is open in $W_n$ by Lemma \ref{lem:open}, and $W_n^\mu$ is the locus of $(f_0, \ldots, f_N)$ such that $\left(f_0^{\bar{\lambda}_0}, \ldots, f_N^{\bar{\lambda}_N}\right)$ simultaneously vanish at exactly $l$ points to order $\mu_i$ for $i = 1, \ldots, l$. 
	
	\begin{lem} The subvarieties $W_n^{min}$ and $W_n^{\mu}$ of $W_n$ are $\bG_m$-invariant and we have 
		$$
		\cW^{min}_n = \left[W_n^{min}/\bG_m\right] \text{ and } \cW_n^{\mu} = \left[W_n^\mu/\bG_m\right].
		$$
	\end{lem}
	
	Next we describe the parametrized version $\cR_n^{\mu}$. Let $\pi_{n,\mu} : \cR^\mu_n \to \Conf_{l(\mu)}$ denote the natural projection and let $\pi_{n,\gamma}$ be the restriction of $\pi_{n,\mu}$ to $\cR^\gamma_n \subset \cR^\mu_n$. Then $\Conf_{l(\mu)} = (\mathbb{P}^1)^l \setminus \Delta$ where $\Delta$ is the big diagonal. For each $(x_1, \ldots, x_l) \in \Conf_{l(\mu)}(\bP^1)$, we define
	$$
	R^{\gamma}_n(x_1, \ldots, x_n) = \left\{(f_0, \ldots, f_N) \mid \nu_{x_i}(f_j) \geq \nu_i^{(j)} \text{ with equality if } i \in T_j\right\}.
	$$
	
	\begin{lem}\label{l:helper-lemma} For each $(x_1, \ldots, x_l) \in \Conf_{l(\mu)}(\bP^1)$, the locus $R^\gamma_n(x_1, \ldots, x_n) \subset W_n^\gamma$ is locally closed and $\bG_m$-invariant. Moreover, we have an identification
		$$
		\pi_{n,\gamma}^{-1}(x_1, \ldots, x_l) = \left[R^\gamma_n(x_1, \ldots, x_l)/\bG_m\right]. 
		$$
	\end{lem}
	
	We end the section with a discussion of the special case $l = 1$ which will feature prominently later in the paper. The weighted linear series has a basepoint at exactly one point $x$ with normalized multiplicity $\mu \in \Z_{>0}$. In this case the vanishing order is just a single number $\nu^{(j)}$ and $T_j$ is either empty, in which case we allow $\nu_x(f_j) \geq \nu^{(j)}$, or $\{1\}$, in which case we require $\nu_x(f_j) = \nu^{(j)}$. For this reason we denote the tuple as a list of numbers, some of which have inequalities. 
	
	\begin{exmp} If $\gamma = ( \{1,\varnothing\}, \{1,\{1\}\} )$, we denote it by the symbol $(\geq 1, 1)$. \end{exmp}
	
	
	
	\begin{exmp} If $\gamma = ( \{2,\{1\}\}, \{3,\{1\}\} )$, we denote it by the symbol $(2, 3)$. \end{exmp}
	
	\begin{prop}\label{prop:single_addi} When $l = 1$,
		$$
		\pi_{n,\gamma} : \cR^\gamma_n \to \bP^1
		$$
		is a Zariski-locally trivial fibration which can be written as a $\bG_m$ quotient of a locally trivial fibration of $\bP^1$ with fibers $R^\gamma_n(x)$. In particular, 
		$$
		R^\gamma_n(x) \cong R^\gamma_n(y)
		$$
		for all $x, y \in \bP^1$ and we have an equality
		$$
		\{\cR^\gamma_n\}=\{\bP^1\}\cdot\frac{\{R^\gamma_n(0)\}}{\{\bG_m\}}
		$$
		in the Grothendieck ring of stacks (Section \ref{subsec:K0}).
	\end{prop}
	\begin{proof} By Lemma \ref{l:helper-lemma}, $\cR^\gamma_n$ is isomorphic over $\bP^1$ to the $\mathbb{G}_m$ quotient of the locus inside
		$$
		\bP^1 \times \bigoplus_{j = 0}^N V^{\lambda_j}_N
		$$
		of pairs $(x, f_1, \ldots, f_N)$ such that the $f_i$ do not simultaneously vanish on $\bP^1 \setminus x$, $\min_j \{\bar{\lambda}_j \nu_x(f_j)\} = \mu$, and $\nu_x(f_j) \geq \nu^{(j)}$ with equality if $T_j = \{1\}$; note that this locus agrees with $W^\gamma_n$ since $\ell=1$. As $\bG_m$ is special, it suffices to show that $W^\gamma_n \to \bP^1$ is Zariski-locally trivial. Now $W^\gamma_n \to \bP^1$ is equivariant for the natural $\PGL_2$ action 
		$$
		(x, f_1, \ldots, f_N) \mapsto (gx, f_1\circ g^{-1}, \ldots, f_N \circ g^{-1}). 
		$$
		Locally in a neighborhood of any given point, say $0 \in \bP^1$, there exists a chart $0 \in U \subset \bP^1$ and a family of automorphisms $\phi : U \to \PGL_2$ such that $U$ is invariant under $\phi(U)$ and $\phi(u)\cdot 0 = u$. Now let $W^\gamma_n(0)$ denote the fiber of $W^\gamma_n$ over $0$ and consider the map
		$$
		U \times W^\gamma_n(0) \to (W^\gamma_n)|_U
		$$
		given by
		$$
		(u, f_1, \ldots, f_N) \mapsto (u, f_1 \circ \phi(u)^{-1}, \ldots, f_N \circ \phi(u)^{-1})
		$$
		On the other hand, we can act on $(W^\gamma_n)|_U$ by $\phi(u)^{-1}$ to produce a map backwards and it is clear these maps are inverses, proving the required Zariski-local triviality. The computation of the motives then follows from \cite[Prop. 1.1 iii), 1.4]{Ekedahl}.
	\end{proof}
	
	\begin{rmk} Note that when $l = 1$, $\Conf_1 \cong \Sym^{\{\mu\}}(\bP^1)$ and $\cW^\gamma_n = \cR^\gamma_n$ so we conclude that $\cW^\gamma_n$ is a locally trivial fibration over $\bP^1$ with fiber $\left[R^\gamma_n(0)/\bG_m\right]$. 
	\end{rmk}
	
	\begin{rmk}
		Under mild conditions on the characteristic, various moduli stacks of curves are isomorphic to weighted projective stacks, and thus admit height moduli. In this regard, we recall the works of \cite{Behrens, Stojanoska, HMe, Meier} for modular curves with prescribed level structures (see also the works of \cite{KM, Conrad, Niles}). For curves of arithmetic genus one with multiple marked points, we recall the works of \cite{Smyth, Smyth2, LP}. Lastly for examples in higher genus, we recall the works of \cite{Stankova, Fedorchuk, HP2}.
	\end{rmk}

	\section{Moduli stacks of elliptic surfaces with specified Kodaira fibers}
	\label{sec:Mod_Wss_Surf}

	In this section, we formulate the moduli stacks of elliptic surfaces of stacky height $n$ with fixed singular fibers. Via the isomorphism $\Me \iso \Pc(4,6)$ over $\bZ\left[\frac{1}{6}\right]$, these are identified with the moduli spaces $\cH^\Gamma_d$ and $\cR^\gamma_n$ of twisted maps with twisting conditions $\Gamma$ and weighted linear series with vanishing conditions $\gamma$ respectively and the isomorphisms between $\cH^\Gamma_d$ and $\cR^\gamma_n$ for various $\Gamma, \gamma, d$ and $n$ are interepreted in terms of the birational geometry of elliptic surfaces.

	\medskip
	
	We begin with the basic definitions surrounding elliptic surfaces following \cite{Miranda2, SS} (see also \cite{Silverman, Liu}).
	
	\begin{defn}
		An irreducible \textbf{elliptic surface} ($f: X \to C, S$) is an irreducible algebraic surface $X$ together with a surjective flat proper morphism $f: X \to C$ to a smooth curve $C$ and a section $S: C \hookrightarrow X$ such that:
		\begin{enumerate} 
			\item the generic fiber of $f$ is a stable elliptic curve, and
			\item the section is contained in the smooth locus of $f$. 
		\end{enumerate}
		
	\end{defn}
	
	\begin{defn}\label{def:minimal_ell_surf} 
		A \textbf{minimal elliptic surface} is an elliptic surface which is relatively-minimal i.e., if none of the fibers contain any $(-1)$-curves.
	\end{defn}
	
	\begin{defn}\label{def:weier_ell_surf} 
		A \textbf{Weierstrass fibration} is an elliptic surface obtained from an elliptic surface by contracting all fiber components not meeting the section. We call the output of this process a \textbf{Weierstrass model}. If starting with a smooth relatively-minimal elliptic surface, we call the result a \textbf{minimal Weierstrass model}.  
	\end{defn}
	
	Lastly, the geometry of an elliptic surface is largely influenced by the \emph{fundamental line bundle} $\cL$ which is the $L$ in a weighted linear series.
	
	\begin{defn} The \textbf{fundamental line bundle} of an elliptic surface is $\cL \coloneqq (f_* \mathcal{N}_{S/X})^\vee$, where $\mathcal{N}_{S/X}$ denotes the normal bundle of $S$ in $X$. 
		
	\end{defn}
	
	An important observation is that the Weierstrass model over $\bZ\left[\frac{1}{6}\right]$ has a global Weierstrass equation of the form
	\begin{equation}\label{eqn:weierstrass}
		\{y^2 = x^3 + Ax + B\} \in \bP_C(\cO \oplus \cL^{\otimes -2} \oplus \cL^{\otimes -3})
	\end{equation}
	where $A \in H^0(C, \cL^{\otimes 4})$ and $B \in H^0(C, \cL^{\otimes 6})$ as in \cite[Cor. 2.5]{Miranda2} and the minimality condition for the Weierstrass model is exactly the same as the minimality condition as in \S \ref{def:mimal-wted-lin-ser} for a weighted linear series.

	The singular fibers of elliptic fibrations were classified in the classical works of \cite{Kodaira, Neron}. There are two types of elliptic surfaces depending on what kind of singular fibers (i.e. bad reductions) underlying elliptic fibrations have. When there exists only \textbf{multiplicative} bad reductions then we call such a smooth relatively-minimal elliptic surface to be an \textit{semistable elliptic surface}. When there exists at least one \textbf{additive} bad reduction then we call such a smooth relatively-minimal elliptic surface to be an \textit{unstable elliptic surface}. By the well-known Tate's algorithm, the classification of singular fibers corresponds to vanishing conditions on the coefficients of the Weierstrass equation (see \cite[Table 1]{Herfurtner}).

	The following is clear.
	
	\begin{prop}\label{Wei_Ell_Moduli}
		The height moduli space $\cW_{n,C}^{min}(4,6) = \cM_{n,C}(\cP(4,6), \cO(1))$ of weighted linear series is the moduli stack of minimal Weierstrass models $(f : X \to C, S)$ with the degree of the fundamental line bundle equal to $n$. \end{prop}
	
	As an immediate corollary we obtain the following. 
	
	\begin{cor}\label{cor:unstable_height} The stratum $\cR^\gamma_{n,C}(4,6)$ is the moduli space of Weierstrass elliptic surfaces with unstable fibers specificed by $\gamma$. Moreover, the Faltings height of an elliptic surface given by the degree of the of the fundamental line bundle agrees with the stacky height for a rational point of $\overline{\cM}_{1,1} \cong \cP(4,6)$ with $\cO(1)$. 
	\end{cor}
	
	On the other hand, the spaces of twisted maps to $\overline{\cM}_{1,1}$ also have an interpretation as moduli of elliptic surfaces. This was originally used in \cite{av, AB2} to compactify the moduli space of elliptic surfaces. We review the construction here for the benefit of the reader. 
	
	\begin{defn}
		A \textbf{stable stack-like} elliptic surface is a tuple $(h : \cY \to \cC, \cS)$ where $h$ is a surjective proper representable morphism from an orbifold surface to a twisted curve and $(\cY, \cS) \to \cC$ is a flat family of stable elliptic curves with section such that the stabilizers of $\cC$ act generically fixed point free on the fibers of $h$. A \textbf{twisted model} is the coarse moduli space $(g : Y \to C, S_0)$ of a stable stack-like surface.
	\end{defn}
	
	A stable stack-like surface induces a representable classifying morphism
	$$
	\varphi : \cC \to \overline{\cM}_{1,1}
	$$
	where $\cY \cong_{\cC} \varphi^*\overline{\cE}$ for $\overline{\cE} \to \overline{\cM}_{1,1}$ the universal family. In particular, the twisted structure on $\cC$ and the stabilizer action on $\varphi^*\cO(1)$ are encoded by a tuple of local twisting conditions $\Gamma$. 
	
	\begin{prop}\label{prop:Min_Ell_Moduli}
		The space $\cH^\Gamma_d(\Pc(4,6))$ is the moduli of stable stack-like elliptic surfaces with local twisting conditions $\Gamma$ such that $\varphi^*\cO(1) \cong \cL$ is the fundamental line bundle and $12d = \deg(j)$ where $j : C \to \overline{M}_{1,1}$ is the $j$-map. In particular, the stable height of an elliptic surface is $\frac{1}{12}\deg(j)$. 
	\end{prop}
	
	\begin{proof}
		The identification of moduli spaces is clear. The only thing we need to check is that $12d = \deg(j)$ but this follows from the observation that $\pi^*\cO_{\bP^1}(1) = \cO(12)$ where $\pi : \overline{\cM}_{1,1} \to \overline{M}_{1,1} \cong \bP^1$ is the coarse map. 
	\end{proof}
	
	\begin{rmk}
		In the special case $\Gamma$ is empty, $\cH^\Gamma_n$ is just $\Hom_n(C, \overline{\cM}_{1,1})$ studied in \cite{HP} for $C = \Pb^1$ case regarding its motive in the Grothendieck ring of stacks and \cite{BPS} for a general smooth projective curve $C$ of any genus $g$ regarding its $\ell$-adic \'etale cohomology with Frobenius weights. The identification between $\cH^\emptyset_n$ and $\cR^\emptyset_n$ is simply the fact that a semi-stable elliptic surface can equivalently be described by a Weierstrass equation or by a classifying morphism $C \to \overline{\cM}_{1,1}$ and in this case the stacky height is exactly $n = \frac{1}{12}\deg(j)$.
	\end{rmk}
	
	\subsection{Geometric interpretation of Tate's algorithm} 
	
	By Tate's algorithm, the base profile $\gamma$ of the Weierstrass equation determines the unstable singular Kodaira fibers of the minimal resolution. On the other hand, by Theorem \ref{thm:height:moduli} for weighted projective stacks, this is equivalent to local twisting conditions $\Gamma$ which in turn determines the stabilizers of the stable stack-like model. The unstable singular fibers of the twisted model are then determined by the following lemmas. 
	
	\begin{lem}\label{lem:tangent:action} Let $h: \cE \to \cD$ be a semistable family of elliptic curves where $\cD = \left[\Spec k[t]^{sh}/\mu_r\right]$ and $\varphi : \cD \to \overline{\cM}_{1,1}$ is a twisted map with local condition $(r,a)$. Let $p \in \cE_0$ be a smooth fixed point of the $\mu_r$ action on the central fiber $\cE_0$. Then the weights of the $\mu_r$ action on the tangent space $T_{\cE,p}$ are $(-1, -a)$.
	\end{lem}
	\begin{proof} Since $\cD$ is strictly henselian, there is a section of $S$ of $h:\cE \to \cD$ passing through $p$. Thus the tangent space $T_{\cE, p}$ splits as a direct sum of fiber and section direction and each of these directions are eigenspaces for the $\mu_r$ action. The section direction is canonically isomorphic to the tangent space $T_{\cD, 0}$ of the base which is isomorphic to $(\mathfrak{m}/\mathfrak{m}^2)^\vee$ where $\mathfrak{m}$ is the maximal ideal of $0$. Since $t$ is a uniformizer, $\mathfrak{m}/\mathfrak{m}^2$ is rank $1$ generated by $t$ which transforms as $t \mapsto \zeta t$, so the dual has weight $(-1)$. On the other hand, $\cL = \varphi^*\cO(1)$ is the fundamental line bundle of the elliptic fibration and thus isomorphic to $h_*(N_{S/\cE})^\vee$. Thus the fiber 
		$$
		\cL^\vee|_0 \cong N_{S/\cE}|_0 = T_{\cE_0, p}
		$$
		and so the weight of $\mu_r$ acting on $T_{\cE_0,p}$ is minus the weight on $\cL|_0$ which is $-a$. 
	\end{proof} 
	
	\begin{lem}\label{lem:mult}
		Let $(h : \cE \to \cD, \cS)$ be a stable stack-like elliptic surface over $\cD = \left[\Spec k[t]^{sh}/\mu_r\right]$ and let $(g : Y \to D, S_0)$ be the coarse moduli space. Then the central fiber $g^*(0)$ has multiplicity $r$. 
	\end{lem}
	\begin{proof}
		Let $\pi : \cE \to Y$ and $\rho : \cD \to D$ be the coarse moduli maps. Then $\rho$ is ramified to order $r$ at $0$ so $h^*\rho^*(0) = r\cE_0$. By commutativity of $\pi, \rho, h$ and $g$, we have $\pi^*g^*(0) = r\cE_0$. On the other hand, $\mu_r$ acts faithfully on $\cE_0$. In particular, this action is generically fixed point free so $\pi$ is generically \'etale along $\cE_0$ and of degree $1$. Thus $\pi_*\cE_0 = Y_0$ so by push-pull, we have $g^*(0) = rY_0$. 
	\end{proof}
	
	Particularly, we remark that this gives a new interpretation of Tate's algorithm as: the vanishing condition $\gamma = (\nu(a_4), ~ \nu(a_6))$ of the minimal weighted linear series $\implies$ the local twisting condition $\Gamma = (r,a)$ of the twisted maps $\implies$ the unstable singular fibers of the twisted model $\implies$ specified type of Kodaira fiber by resolving singularities and contracting $(-1)$-curves. This is summarized in the diagram below.
	
	
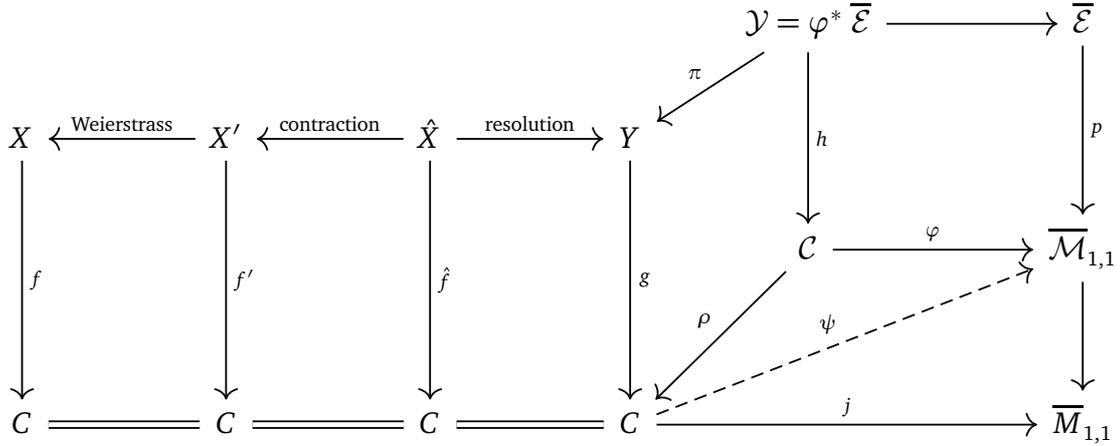
\begin{equation*}\label{diag:Master}
		\begin{tikzcd}
			&  &                                                                          &  &                                                                                                      &  &                                                  & {\mathcal{Y} = \varphi^* ~ \overline{\mathcal{E}}} \arrow[rr] \arrow[dd, "h"] \arrow[ld, "\pi"'] &  & {\overline{\mathcal{E}}} \arrow[dd, "p"] \\
			X \arrow[ddd, "f"]                &  & X^{\prime} \arrow[ddd, "f^{\prime}"] \arrow[ll, "\mathrm{Weierstrass}"'] &  & \hat{X} \arrow[rr, "\mathrm{resolution}"] \arrow[ddd, "\hat{f}"] \arrow[ll, "\mathrm{contraction}"'] &  & Y \arrow[ddd, "g"]                               &                                                                                                        &  &                                                \\
			&  &                                                                          &  &                                                                                                      &  &                                                  & \mathcal{C} \arrow[rr, "\varphi"] \arrow[ldd, "\rho"']                                                 &  & {\overline{\mathcal{M}}_{1,1}} \arrow[dd]      \\
			&  &                                                                          &  &                                                                                                      &  &                                                  &                                                                                                        &  &                                                \\
			C \arrow[rr, Rightarrow, no head] &  & C \arrow[rr, Rightarrow, no head]                                        &  & C \arrow[rr, Rightarrow, no head]                                                                    &  & C \arrow[rrr, "j"] \arrow[rrruu, "\psi", dashed] &                                                                                                        &  & {\overline{M}_{1,1}}                          
		\end{tikzcd}
	\end{equation*}
	
	Here $f$ is a Weierstrass model, $\psi$ is the associated weighted linear series viewed as a rational map to $\overline{\cM}_{1,1}$, $j$ is the $j$-invariant, $\varphi$ is the universal tuning stack which induces a stable stack-like model $h : \cY \to \cC$, $g : Y \to C$ is the twisted model, $\hat{f}$ is a resolution of $Y$, and $f'$ is the relative minimal model obtained by contracting relative $(-1)$-curves. Lemmas \ref{lem:tangent:action} and \ref{lem:mult} determine singular fibers of $g$ and the singularities of $Y$ which determine $\hat{X}$ and in turn the Kodaira fiber of the minimal model $f'$. By Proposition \ref{prop:normalized:local:conditions}, the local twisting conditions for $\varphi$ depend only on the base multiplicity of the normalized linear series. The extra data we need to determine the singular fiber is simply the order of vanishing of the $j$ map at $j = \infty$ as this determines the singularities of $\cY$. Combining these observations, we get. 
	
	\begin{thm}[Tate's algorithm via twisted maps]\label{thm:tate} The local twisting condition $(r,a)$ and the order of vanishing of $j$ at $j = \infty$ determine the Kodaira fiber type of the relative minimal model, and $(r,a)$ is in turn determined by $m = \min\{3\nu(a_4), 2\nu(a_6)\}$. 
	\end{thm}
	
	\begin{rmk}
		The fact that Tate's algorithm can be phrased so that the Kodaira fiber depends only on the normalized base multiplicity, that is, the minimum of the normalized orders of vanishing of the Weierstrass equation was first observed in \cite{Dokchitser}.
		
	\end{rmk}
	
	The output of this analysis is summarized in Theorem~\ref{thm:corr_46}. A key point is that Theorem \ref{thm:tuning-stacks-and-minimal-linear-series} and Proposition \ref{prop:normalized:local:conditions} allows us to carry out this analysis for any moduli spaces which are isomorphic to weighted projective stacks and thus gives us a generalization of Tate's algorithm for maps to weighted projective stacks.

	\begin{exmp} Suppose that normalized base multiplicity $m = 3$. This occurs if and only if $(\nu(a_4), \nu(a_6)) = (1, \geq 2)$. Then $r = 12/\gcd(3,12) = 4$ and $a = 3/\gcd(3,12) = 1$. Thus the stabilizer of the twisted curve acts on the central fiber of the twisted model via the character $\mu_4 \to \mu_4$, $\zeta_4 \mapsto \zeta^{-1}_4$. In particular, the central fiber $E$ of $\cY$ has $j = 1728$. The $\mu_4$ action on $E$ has two fixed points, and there is an orbit of size two with stabilizer $\mu_2 \subset \mu_4$. Let $E_0$ be the image of $E$ in the twisted model $Y$. By Lemma \ref{lem:mult}, $E$ appears with multiplicity $4$ By Lemma \ref{lem:tangent:action}, $Y$ has $\frac{1}{4}(-1,-1)$ quotient singularities at the images of the the fixed points and a $\frac{1}{2}(-1,-1)$ singularity at the image of the orbit of size two. Each of these singularities is resolved by a single blowup to obtain $\hat{X}$ with central fiber $4\tilde{E}_0 + E_1 + E_2 + E_3$ where $E_i$ are the exceptional divisors of the resolution for $i = 1,2,3$ and $E_1^2 = E_2^2 = -4$ with $E_3^2 = -2$. Then $\tilde{E}_0$ is a $(-1)$-curve so it needs to be contracted. After this contraction $E_2$ becomes a $(-1)$ curve and must also be contracted. Since $E_i$ for $i = 1,2,3$ are incident and pairwise transverse after blowing down $\tilde{E}_0$, then the images of $E_1$ and $E_2$ must be tangent after blowing down $E_3$. Moreover, they are now $(-2)$-curves and so we conclude that this is the relatively minimal model and that it is the type $\mathrm{III}$ configuration in Kodaira's classification. See \cite[Sec. 4]{AB} for more details on these blowup computations. 
	\end{exmp}
	
	\begin{exmp}
		Suppose $m = 6$ which occurs in the cases $(\nu(a_4), \nu(a_6)) = (2,3), (\geq 3, 3)$ or $(2, \geq 4)$. Note these three cases are distinguished by their $j$-invariant but are identical from the point of view of twisted maps. Then $(r,a) = (2,1)$ so we have a $\mu_2$ automorphism on the central fiber $E \subset \cY$. If $j \neq \infty$, $E$ is smooth and the $\mu_2$ action has $4$ fixed points, the $2$-torsion points. The image $E_0$ of $E$ appears with multiplicity $2$ in the central fiber of the twisted model $Y$ and $Y$ has four $\frac{1}{2}(-1,-1)$ singularities. Note these are simply $A_1$ singularities so they are resolved by a single blowup which extracts a $(-2)$-curve. This resolution is already relatively minimal and it is exactly a Kodaira $\mathrm{I}_0^*$ fiber. If $j = \infty$ then $E$ is nodal. There are two smooth fixed points as well as the nodal point. At the nodal point $\cY$ has an $A_{2n-1}$ singularity where $2n$ is the ramification of $\varphi$, or equivalently $n$ is the order of vanishing $j$ at $\infty$. The quotient of the $A_{2n-1}$ singularity by the $\mu_2$ action is an $A_3$ singularity if $n = 1$ and a $D_{n + 2}$ singularity if  $n \geq 2$. In either case, resolving these singularities yields a $D_{n+4}$ configuration of $(-2)$ curves, that is, an $\mathrm{I}_n^*$ fiber. See \cite[Lem. 4.2]{AB} for more details about this resolution. 
	\end{exmp}
	
	\subsection{Identifying the universal families} 
	
	In this section we promote the geometric discussion of Tate's algorithm via twisted maps in the previous section (which we now think of as a bijection between Weierstrass elliptic surfaces of height $n$ with vanishing conditions $\mu = (m_1, \ldots, m_s)$ and twisted elliptic surfaces of stable height $d$ with twisting conditions $\Gamma = \{(r_1, a_1), \ldots, (r_s, a_s)\}$) into an identification of respective universal families. 
	
	More specifically, by Proposition \ref{prop:iso-HGammad-Rmun}, we have a canonical isomorphism 
	\begin{equation}\label{eqn:outline:iso}
		\rho : \cH^\Gamma_n \cong \cR^\mu_n
	\end{equation}
	where $\mu$ and $\Gamma$ are in bijection via Lemma \ref{lem:bijection} and 
	$$
	n = d + \sum \frac{a_i}{r_i}. 
	$$
	
	By Corollary \ref{cor:unstable_height}, we have a universal family $(f : X \to C \times \cR^\mu_n, S, \sigma_1, \ldots, \sigma_s)$ of Weierstrass elliptic surfaces of height $n$ with section $S$ and unstable fibers along the marked points $\sigma_i : \cR^\mu_n \to C \times \cR^\mu_n$ with normalized base multiplicity $m_i$. On the other hand, by Proposition \ref{prop:Min_Ell_Moduli}, there is a universal family of stable stack-like elliptic surfaces $(h : \cY \to \cC \to \cH^\Gamma_d, \cS, \Sigma_1, \ldots, \Sigma_s)$ where $\Sigma_i$ are a family of marked gerbes with twisting conditions $\Gamma$. By taking coarse space we also have a universal family of twisted surfaces $(g : Y \to C \times \cH^\Gamma_d \to \cH^\Gamma_d, S_0, \sigma_i)$.
	
	\begin{thm}\label{thm:sec7main}
		The isomorphism $\rho$ is induced by a birational transformation of universal families. More precisely, $\rho^*(f : X \to C \times \cR^\mu_n, S, \sigma_i)$ is the family of Weierstrass models of $(h : \cY \to \cC \to \cH^\Gamma_d, \cS, \Sigma_i)$ and conversely $(\rho^{-1})^*(g : Y \to C \times \cH^\Gamma_d \to \cH^\Gamma_d, S_0, \sigma_i)$ is the family of twisted models of $(f : X \to C \times \cR^\mu_n \to \cR^\mu_n, S, \sigma_i)$. 
	\end{thm}
	
	\begin{proof}
		The proof follows the same strategy as the general wall-crossing theorems of \cite{AB3, Inchiostro, ABIP} for moduli of elliptic surfaces. Namely both twisted models and Weierstrass models are canonical models of the pair $(X', S' + \sum a_i F_i)$ for different coefficients $a_i$. Here $(X' \to C, S')$ is a relatively minimal elliptic surface and $F_i$ are the unstable fibers. Both $\cH$ and $\cR$ can then be identified as strata in the moduli space of log canonical models and $\rho$ is the restriction of a wall-crossing morphism which relates the moduli space for different values of $a_i$. The content of the theorem then is that the minimal model programs that terminate in the Weierstrass model (resp. the twisted model) can be run on the universal families. In order to avoid the technical details of moduli spaces of canonical models, we sketch a direct proof here. 
		
		First note that since our stacks are tame, the formation of the coarse moduli space $g : Y \to C \times \cH^\Gamma_d \to \cH^\Gamma_d$ commutes with base change and $Y \to \cH^\Gamma_d$ is flat.  Moreover, for each $i$, we can consider the family $\cP_i = \cS \cap h^{-1}(\Sigma_i)$. The stabilizer $\mu_{r_i}$ of $\Sigma_i$ acts by an automorphism of the fiber as a pointed elliptic curve and in particular fixes the section $\cS$. Thus $\cP_i \to \cH^\Gamma_d$ is an \'etale $\mu_{r_i}$-gerbe, in fact isomorphic to $\Sigma_i$, and the weight of the $\mu_{r_i}$ action on the tangent space of the fibers of $\cY \to \cH^\Gamma_d$ is fixed by Lemma \ref{lem:tangent:action} as we fix the local twisting condition $(r_i, a_i)$. Thus the family of twisted surfaces $Y \to \cH^\Gamma_d$ is equisingular with singularity $\frac{1}{r_i}(-1,-a_i)$ along the coarse space $P_i = S_0 \cap f^{-1}(\sigma_i) \subset Y$ of $\cP_i \subset \cY$. By \cite[Prop. 5.9]{Inchiostro}, we can take a partial fiberwise resolution $\mu : \hat{X} \to Y \to \cH^\Gamma_d$ which resolves the singularities of the family $Y \to \cH^\Gamma_d$ along $P_i$ for each $i$. 
		
		Denote the corresponding family of elliptic fibrations by
		$$
		(\hat{f} : \hat{X} \to C \times \cH^\Gamma_d \to \cH^\Gamma_d, \hat{S}, \hat{\sigma}_i). 
		$$
		By construction, the section $\hat{S}$ passes through the smooth locus of $\hat{f}$ and the formation of $\mu$ commutes with basechange. Now the Weierstrass model of $\hat{f}$ can be computed as the relative Proj 
		$$
		X' = \mathrm{Proj}_{C \times \cH^\Gamma_d} \bigoplus_{m \geq 0} \hat{f}_*\cO_{\hat{X}}(m\hat{S}) \to C \times \cH^\Gamma_d. 
		$$
		The cohomology $H^i\left(\hat{f}^{-1}(p), \cO_{\hat{X}}(m\hat{S})\big|_{\hat{f}^{-1}(p)}\right) = 0$ for each $p \in C \times \cH^\Gamma_d$ and $i > 0$ and so the formation of $X'$ commutes with base change. Thus we have produced a family of Weierstrass models $(f' : X' \to C \times \cH^\Gamma_d \to \cH^\Gamma_d, S', \sigma_i')$ over $\cH^\Gamma_d$. Moreover, for each $\xi \in \cH^\Gamma_d$, the fiber $(f'_\xi : X'_{\xi} \to C, S'_\xi, (\sigma_i')_{\xi})$ is the Weierstrass model of the fiber $(h : \cY_\xi \to \cC_\xi, \cS_\xi, (\Sigma_i)_\xi)$. By Tate's algorithm via twisted maps (Theorem \ref{thm:tate}), this family of Weierstrass model has marked unstable fibers with vanishing conditions $\mu$ and thus induces the map $\rho : \cH^\Gamma_d \to \cR^\mu_n$ such that
		$$
		(f' : X' \to C \times \cH^\Gamma_d \to \cH^\Gamma_d, S', \sigma_i') = \rho^*(f : X \to C \times \cR^\mu_n \to \cR^\mu_n, S, \sigma_i) 
		$$
		as required. The converse follows since $\rho$ is an isomorphism by Proposition \ref{prop:iso-HGammad-Rmun}. 
	\end{proof}
	
	The upshot is that under the identification $\rho : \cH^\Gamma_d \cong \cR^\mu_n$, the height moduli space has two interpretations as the moduli space of elliptic surfaces with specified Kodaira fibers and these two interpretations are equivalent by Tate's algorithm via twisted maps.

	\section{Motives \& Point counts of height moduli over finite fields}
	\label{sec:motive_ptcts}
	
	In this section, we compute the classes in the Grothendieck ring of stacks of height moduli spaces of $k(t)$-points of $\overline{\cM}_{1,1}$. These computations will be used in the next section to obtain weighted and unweighted counts of minimal elliptic surfaces over $\bP^1$, and also counts of minimal elliptic surfaces over $\bP^1$ with exactly one specified Kodaira fiber type.

	\subsection{Point counts of algebraic stacks over finite fields}
	
	We review the basics on arithmetic of algebraic stacks over finite fields. Due to the presence of automorphisms, point counts of an algebraic stack $\cX$ are usually weighted.

	\begin{defn}\label{def:wtcount}
		The weighted point count of an algebraic stack $\cX$ with finite inertia over $\Fb_q$ is defined as a sum
		\[
		\#_q(\cX)\coloneqq\sum_{x \in \cX(\Fb_q)/\sim}\frac{1}{|\mathrm{Aut}(x)|},
		\]
		where $\cX(\Fb_q)/\sim$ is the set of $\Fb_q$--isomorphism classes of $\Fb_q$--points of $\cX$.
	\end{defn}

	In particular, the weighted point count $\#_q(\cX)$ is not equal to the number $|\cX(\Fb_q)/\sim|$ of $\Fb_q$--isomorphism classes. The main advantage of the weighted point count is that it depends only on the cohomology of $\cX$ and is equal to the usual point count of the coarse moduli space when it exists. On the other hand, it is often interesting to also consider the unweighted count of isomorphism classes. The following result of \cite{HP2} shows that the unweighted point count is also natural and depends on the arithmetic geometry of the inertia stack of $\cX$.

	\begin{thm}[Theorem 1.1. of \cite{HP2}] \label{thm:ptcounts}
		Let $\cX$ be an algebraic stack over $\Fb_q$ of finite type with quasi-separated finite type diagonal and let $\Ic(\cX)$ be the inertia stack of $\cX$. Then, 
		$$ |\cX(\Fb_q)/\sim| = \#_q(\Ic(\cX))$$ 
	\end{thm}

	\subsection{Grothendieck ring of stacks}\label{subsec:K0}

	We review some properties of the Grothendieck ring of stacks introduced in \cite{Ekedahl}. 
	
	\begin{defn}\label{defn:GrothringStck}
		\cite[\S 1]{Ekedahl}
		The \emph{Grothendieck ring of stacks} $K_0(\mathrm{Stck}_k)$ is the abelian group generated by classes $\{\cX\}_k$ for each algebraic stack $\cX$ of finite type over $k$ with affine inertia modulo the relations
		\begin{itemize}
			\item $\{\cX\}_k=\{\Zc\}_k+\{\cX \setminus \Zc\}_k$ for $\Zc \subset \cX$ a closed substack,
			\item $\{\Ec\}_k=\{\cX \times_k \Ab^n \}_k$ for $\Ec$ a vector bundle of rank $n$ on $\cX$.
		\end{itemize}
		Multiplication on $K_0(\mathrm{Stck}_k)$ is induced by $\{\cX\}_k\{\Yc\}_k\coloneqq\{\cX \times_k \Yc\}$. 
		There is a distinguished element $\Lb \coloneqq\{\A^1\}_k \in K_0(\mathrm{Stck}_k)$, called the \emph{Lefschetz motive}. We drop the subscript if $k$ is clear. 
	\end{defn}
	
	We denote by $K_0'(\mathrm{Stck}_k)$ the ring obtained by imposing only the cut-and-paste relation but not the vector bundle relation and denote the class of a stack in this ring by $\{\cX\}'$.  
	
	\medskip
	
	The Grothendieck ring is universal with respect to additive invariants. When $k=\Fb_q$, the point counting measure $\{\cX\} \mapsto \#_q(\cX)$ is a well-defined ring homomorphism $\#_q: K_0(\mathrm{Stck}_{\Fb_q}) \rightarrow \Q$ giving the weighted point count $\#_q(\cX)$ of $\cX$ over $\Fb_q$. When $\{\cX\}$ is a polynomial in $\Lb$, the weighted point count is polynomial in $q$.
	
	\medskip
	
	Recall that an algebraic group $G$ is \textit{special} in the sense of \cite{Serre} and \cite{Grothendieck}, if every $G$-torsor is Zariski-locally trivial; for example $\bG_a,~\GL_{d},~\SL_{d}$ are special and $\PGL_2,~\PGL_3$ are non-special. If $\mathcal X \rightarrow \mathcal Y$ is a $G$-torsor and $G$ is special, then we have $\{\mathcal X\} = \{G\} \cdot \{\mathcal Y\}$ (\cite[Prop. 1.1 iii)]{Ekedahl}).

	\medskip
	
	Finally, we can use the following result to access unweighted point counts. 
	
	\begin{prop}{\cite[Prop. 5.3]{Fernex}}\label{prop:df} The association $\cX \mapsto \cI \cX$ extends to a unique ring homomorphism 
		$$
		\cI : K_0'(\mathrm{Stck}_k) \to K_0'(\mathrm{Stck}_k)
		$$
		which we call the inertia operator. 
	\end{prop}
	
	Note that $\cI$ does not descend to a well defined operator on $K_0(\mathrm{Stck})$. For example, by \cite[Prop. 1.1 iii)]{Ekedahl},
	$$
	\{\cP(2,2)\} = \left\{ \left[\mathbb{A}^2 \setminus \{(0,0)\}/\mathbb{G}_m\right] \right\} = \{\bP^1\}
	$$
	but $\{\cI \cP(2,2)\} = 2\{\cP(2,2)\} \neq \{\cI\bP^1\}$. 
	
	\subsection{The motivic height zeta function of a weighted projective stack}\label{sec:zeta}
	
	In this section we define and compute the motivic height zeta function for $\cP(\lambda_0, \ldots, \lambda_N)$ for the function field $K = k(\bP^1)$. 
	
	\begin{defn}
		The motivic height zeta function of $\cP(\lambda_0, \ldots, \lambda_N)$ is the formal power series
		$$
		Z_{\vec{\lambda}}(t) := \sum_{n \geq 0} \left\{\cW_n^{min}\right\}t^n \in K_0(\mathrm{Stck})\llbracket t \rrbracket
		$$
		where $\cW_n^{min}$ is the space of minimal weighted linear series on $\bP^1$ of height $n$. We let $Z'_{\vec{\lambda}}(t) \in K_0'(\mathrm{Stck}_k)\llbracket t \rrbracket$ denote the same generating series over $K_0'(\mathrm{Stck}_k)$. \\
		
		\noindent We also define the variant
		$$
		\cI Z_{\vec{\lambda}}(t) := \sum_{n \geq 0} \left\{\cI\cW_n^{min}\right\}t^n \in K_0(\mathrm{Stck}_k)\llbracket t \rrbracket
		$$
		and $\cI Z'_{\vec{\lambda}}(t) \in K_0'(\mathrm{Stck}_k)\llbracket t \rrbracket$ the same generating series over $K_0'(\mathrm{Stck}_k)$. 
	\end{defn}
	
	We also denote by 
	$$
	Z(t) = \sum \{\Sym^e\bP^1\}t^e = \frac{1}{(1-\Lb t)(1-t)}
	$$
	the usual motivic zeta function of $\bP^1$. 
	
	\begin{lem}\label{lem:wps}
		The following equalities of formal power series hold over $K_0'(\mathrm{Stck}_k)$. 
		\begin{align}
			\sum_{n \ge 0} \left\{\Pc \left(\bigoplus_{i = 0}^N V_{n}^{\lambda_i} \right) \right\}' t^n = Z'_{\vec{\lambda}}(t) \cdot Z(t) \label{eqn:wps}\\
			\sum_{n \ge 0} \left\{\cI\Pc \left(\bigoplus_{i = 0}^N V_{n}^{\lambda_i} \right) \right\}' t^n = \cI Z'_{\vec{\lambda}}(t) \cdot Z(t) \label{eqn:iwps}
		\end{align}
	\end{lem}
	\begin{proof} 
		We can stratify each of the terms on the left hand side by the minimality defect as in Corollary \ref{cor:def:strat} to obtain an equality
		$$
		\left\{\Pc \left(\bigoplus_{i = 0}^N V_{n}^{\lambda_i} \right) \right\}' = \sum_{e = 0}^n \{\cW_{n-e}^{min}\}' \{\Sym^e\bP^1\}'
		$$
		which implies the first equality by expanding the product. The second equality follows by applying $\cI$ to both sides and using Proposition \ref{prop:df}. 
	\end{proof} 
	
	Next we compute the image of the left hand side of Equation (\ref{eqn:wps}) in $K_0(\mathrm{Stck}_k)$. 
	Let $|\vec{\lambda}| = \sum \lambda_i$ denote the sum of the weights. 
	\begin{lem}\label{cor:tbd6}
		
		\begin{align*}
			\sum_{n \ge 0} \left\{\Pc \left(\bigoplus_{i = 0}^N V_{n}^{\lambda_i} \right) \right\} t^n = \frac{ \{\Pb^N\} + \Lb^{N+1} \{\Pb^{|\vec{\lambda}|-N-2} \} t}{(1-t)(1-\Lb^{|\vec{\lambda}|}t)}
		\end{align*}
		
	\end{lem}

	\begin{proof} 
		
		Note that $\bigoplus_{i = 0}^N V_{n}^{\lambda_i}$ has dimension $\sum\limits_{i=0}^{N} (n\lambda_i + 1) = n|\vec{\lambda}| + N + 1$ thus by \cite[Prop. 1.1 iii)]{Ekedahl}, we have $\left\{\Pc \left(\bigoplus_{i = 0}^N V_{n}^{\lambda_i} \right) \right\} = \frac{\Lb^{n|\vec{\lambda}| + N + 1}-1}{\Lb-1}$ and this implies \\
		
		\begin{align*}
			\sum_{n \ge 0} \frac{\Lb^{n|\vec{\lambda}| + N + 1}-1}{\Lb-1} t^n &= \frac{\Lb^{N+1}}{\Lb-1} \sum_{n \ge 0} (\Lb^{|\vec{\lambda}|}t)^n - \frac{1}{\Lb-1} \sum_{n \ge 0} t^n\\
			&= \frac{\Lb^{N+1}}{(\Lb-1)(1-\Lb^{|\vec{\lambda}|}t)} - \frac{1}{(\Lb-1)(1-t)} \\
			&= \frac{1}{\Lb-1} \left( \frac{ \Lb^{N+1} - 1 + (\Lb^{|\vec{\lambda}|} - \Lb^{N+1})t }{(1-t)(1-\Lb^{|\vec{\lambda}|}t)} \right) \\
			&= \frac{ \{\Pb^N\} + \Lb^{N+1} \{\Pb^{|\vec{\lambda}|-N-2} \} t}{(1-t)(1-\Lb^{|\vec{\lambda}|}t)}  \text{ since $|\vec{\lambda}| \ge N+1$.}
		\end{align*}
		
	\end{proof}
	
	To compute the left hand side of Equation (\ref{eqn:iwps}), we need some notation Following \cite[Sec. 3]{HP2}. Let $\Bmu_k$ denote the $\mathrm{Gal}(k^{sep}/k)$ orbits of roots of unity in $k^{sep}$ and for any $g \in \Bmu_k$ let $k(g)/k$ be the field extension obtained by adjoining those roots of unity. For each element $g$ we have a well defined order $\mathrm{ord}(g)$ since the Galois action preserves the order of the root of unity. For any $g \in \Bmu_k$, let $I_g \subset \{0, \ldots, N\}$ denote the largest subset of indices such that $\mathrm{ord}(g)$ divides $\gcd\{\lambda_i\}_{i \in I_g}$ and we let $\vec{\lambda}_g = (\lambda_i)_{i \in I_g}$ be the weights corresponding to $I_g$. We let $N_g = \# I_g - 1$. Finally, we let $\mu_\infty(k)$ be the roots of unity in $k$, i.e. the elements $g \in \Bmu_k$ such that $k(g) = k$. 
	
	\begin{lem}\label{cor:tbd7}
		
		\begin{align*}
			\sum_{n \ge 0} \left\{ \Ic\Pc \left(\bigoplus_{i = 0}^N V_{n}^{\lambda_i} \right) \right\} t^n = \sum_{g \in \mu_\infty(k)} \frac{ \{\Pb^{{N_g}}\} + \Lb^{{{N_g}}+1} \{\Pb^{|\vec{\lambda_{g}}|-{{N_g}}-2} \} t}{(1-t)(1-\Lb^{|\vec{\lambda_{g}}|}t)} + \sum_{g \in \Bmu_k \setminus \mu_\infty(k)} Q_g(t)
		\end{align*}
		where the $Q_g(t)$ are power series in $K_0(\mathrm{Stck}_k)$ whose coefficients have no $k$-points. 
		
	\end{lem}
	
	\begin{proof} 
		
		By \cite[Prop. 3.5]{HP2}, we have
		
		\begin{equation*}
			\sum_{n \ge 0} \left\{ \Ic\Pc \left(\bigoplus_{i = 0}^N V_{n}^{\lambda_i} \right) \right\}_k t^n = \sum_{g \in \Bmu_k}\sum_{n \geq 0} \left\{\cP_{k(g)}\left(\bigoplus_{i \in I_g} V_n^{\lambda_j}\right)\right\}_k t^n
		\end{equation*}
		Note that for any stack $\cX/k$ and finite extension $k'/k$, $\cX_{k'}$ has no $k$-point and so the terms where $k(g) \neq k$ satisfy the condition on $Q_g(t)$. For those $g$ such that $k(g) = k$, note then we have $\dim \bigoplus\limits_{i \in I_g} V_{n}^{\lambda_i} = \sum\limits_{i \in I_g} (n\lambda_i+1) = n|\vec{\lambda_{g}}| + {N_g} + 1$ and so by Lemma \ref{cor:tbd6},
		
		\begin{align*}
			\sum_{n \ge 0} \left\{ \Ic\Pc \left(\bigoplus_{i = 0}^N V_{n}^{\lambda_i} \right) \right\} t^n &= \sum_{g \in \Bmu_k} \sum_{n \ge 0} \left\{\Pc \left(\bigoplus_{i \in I_g} V_{n}^{\lambda_i} \right)\right\} t^n \\
			&= \sum_{g \in \mu_\infty(k)} \sum_{n \ge 0} \frac{ \Lb^{n|\vec{\lambda_{g}}| + {N_g} + 1} - 1}{\Lb-1}t^n + \sum_{g \in \Bmu_k \setminus \mu_\infty(k)} Q_g(t) \\
			&= \sum_{g \in \mu_\infty(k)} \frac{ \{\Pb^{{N_g}}\} + \Lb^{{{N_g}}+1} \{\Pb^{|\vec{\lambda_{g}}|-{{N_g}}-2} \} t}{(1-t)(1-\Lb^{|\vec{\lambda_{g}}|}t)} +\sum_{g \in \Bmu_k \setminus \mu_\infty(k)} Q_g(t) 
		\end{align*}

	\end{proof} 
	
	\begin{thm}\label{thm:ratl} The motivic height zeta functions $Z_{\vec{\lambda}}(t)$ and $\cI Z_{\vec{\lambda}}(t)$ are given by the following formulas in $K_0(\mathrm{Stck}_k)$:
		\begin{enumerate}[(a)]
			\item \label{thm:ratla}    
			$$
			Z_{\vec{\lambda}}(t) = 
			\frac{ 1-\Lb t }{1-\Lb^{|\vec{\lambda}|}t} \left( \{\Pb^N\} + \Lb^{N+1}  \{\Pb^{|\vec{\lambda}|-N-2} \} t \right)
			$$
			\item \label{thm:ratlb}    
			$$
			\cI Z_{\vec{\lambda}}(t) = 
			\sum_{g \in \mu_{\infty}(k)} \frac{ 1-\Lb t }{1-\Lb^{|\vec{\lambda_{g}}|}t} \left( \{\Pb^{{N_g}}\} + \Lb^{{N_g}+1}  \{\Pb^{|\vec{\lambda_{g}}| - {N_g} -2} \} t \right) + \sum_{g \in \Bmu_k \setminus \mu_\infty(k)} R_g(t)
			$$
			
			\noindent where $R_g(t)$ are power series whose coefficients have no $k$-points.
		\end{enumerate}
	\end{thm}
	
	\begin{rmk}\begin{enumerate}
			\item Note that the sum is finite as $I_g$ is empty for all but finitely many $g$. 
			\item We expect that $R_g(t)$ is still a rational function in $K_0(\mathrm{Stck}_k)$ though this is not clear. 
		\end{enumerate}
	\end{rmk}

	\begin{proof} This follows by combining Lemma \ref{lem:wps} with the computations in Lemmas \ref{cor:tbd6} and \ref{cor:tbd7} and the explicit formula $Z(t) = \frac{1}{(1-t)(1- \bL t)}$. 
		
	\end{proof} 
	
	\subsection{Motives and point counts of Poly-spaces over finite fields} \label{subsect:Poly}
	
	The Poly-space is a variety of independent interest which is used throughout the subsequent proofs. We begin with $\mathrm{Poly}_1^{(d_1,\dotsc,d_m)}$, a slight generalization of \cite[Def. 1.1]{FW}.
	
	\begin{defn}\label{def:poly}
		Fix $m \in \bZ_{+}$ and $d_1,\dotsc,d_m \ge 0$. Define $\mathrm{Poly}_1^{(d_1,\dotsc,d_m)}$ as the set of tuples $(f_1,\dotsc,f_m)$ of monic polynomials in $K[z]$ so that
		\begin{enumerate}
			\item $\deg f_i=d_i$ for each $i$, and
			\item $f_1,\dotsc,f_m$ have no common roots in $\overline{K}$.
		\end{enumerate}
	\end{defn}
	
	Note that $\mathrm{Poly}_1^{(d_1,\dotsc,d_m)}$ sits inside an affine space parameterizing tuples of monic coprime polynomials. It is the complement of the resultant hypersurface and so can be endowed with the structure of affine variety defined over $\bZ$.

	\medskip
	
	The motive of the Poly-space $\mathrm{Poly}_1^{(d_1, d_2)}$ over $k$ is given by the following.
	
	\begin{prop}
		\label{prop:poly_m}
		Fix $d_1,d_2 \ge 0$. 
		\[ \left\{\mathrm{Poly}_1^{(d_1, d_2)}\right\}=
		\begin{cases}
			\Lb^{d_1 + d_2}-\Lb^{d_1 + d_2 - 1} & \text{ if } d_1,d_2 >0\;, \\
			\Lb^{d_1 + d_2} & \text{ if } d_1=0 \text{ or } d_2=0\;.
		\end{cases}
		\]
	\end{prop}
	
	\begin{proof}
		We refer the reader to \cite[Prop. 18]{HP} for the details of the proof. The proof in (loc.cit.) is analogous to the proof of \cite[Thm. 1.2]{FW}. 
		Here we point out that the motive formula $\left\{\mathrm{Poly}_1^{(d_1, d_2)}\right\}$ not only holds in $\mathrm{char}(k)=0$ (c.f. the corrigendum \cite{FW2}) but also holds in any $\mathrm{char}(k)=p$ due to the critical correction from \cite[Prop. 3.1]{PS} utilizing the Euclidean algorithm. 
		
		For the more general result on $\left\{\mathrm{Poly}_1^{(d_1,\dotsc,d_m)}\right\}$ with more tuples of polynomials, we refer the reader to \cite[Prop. 4.4]{HP2}.
	\end{proof}
	
	We now consider Poly-spaces of polynomials with a common zero of specified vanishing orders.
	
	\begin{defn}\label{def:polygeqab}
		Fix $a,b \in \bZ_{+}$ and $d_1,d_2 \ge 0$. Define $\mathrm{Poly}_{(\ge a,b)}^{(d_1, d_2)}$ as the space parameterizing monic polynomials $(f_1,f_2)$ in $k[z]$ of degrees $(d_1, d_2)$ such that
		\begin{enumerate}
			\item $0$ is the only common root of $f_1$ and $f_2$ over $\overline{k}$,
			\item At $0$, $f_1$ vanishes to order at least $a$ and $f_2$ has exact order of vanishing $b$.
		\end{enumerate}
		We have analogously defined Poly-spaces $\mathrm{Poly}_{(a,\geq b)}^{(d_1, d_2)}$ and $\mathrm{Poly}_{(a,b)}^{(d_1, d_2)}$.
	\end{defn}

	\begin{prop}
		\label{prop:poly_m, Pc(4,6)}
		Fix $a,b \in \bZ_{+}$ and $0 \le d_1 \le d_2$. Set $\alpha = d_1 - a$ and $\beta = d_2 - b$. The motive of the Poly-space $\mathrm{Poly}_{(\ge a,b)}^{(d_1, d_2)}$ over $k$ is given by 
		\[ \left\{\mathrm{Poly}_{(\ge a,b)}^{(d_1, d_2)}\right\} = 
		\left\{\mathrm{Poly}_{(b, \ge a)}^{(d_2, d_1)}\right\} = 
		\begin{cases}
			(\Lb-1) \cdot \frac{ \Lb^{\alpha + \beta} - \Lb^{\alpha - \beta} }{ \Lb + 1 }, & \text{ if } 0 < \beta\leq\alpha,  \\\\
			(\Lb-1) \cdot \frac{ \Lb^{\alpha + \beta} + \Lb^{\beta - \alpha - 1} }{ \Lb + 1 }, & \text{ if }  0\leq\alpha<\beta,  \\\\
			\Lb^\alpha, & \text{ if } \beta=0
		\end{cases}
		\]
		
		The motive of the Poly-space $\mathrm{Poly}_{(a,b)}^{(d_1, d_2)}$ (i.e. the space of monic polynomials $(f_1,f_2)$ in $K[z]$ of degrees $(d_1, d_2)$ having exact order of vanishing $(a,b)$ at 0) over $k$ is given by 
		\[ \left\{\mathrm{Poly}_{(a,b)}^{(d_1, d_2)}\right\} = \left\{\mathrm{Poly}_{1}^{(d_1-a, d_2-b)}\right\} - \left\{\mathrm{Poly}_{(\ge (a+1),b)}^{(d_1, d_2)}\right\} - \left\{\mathrm{Poly}_{(a,\ge (b+1))}^{(d_1, d_2)} \right\} 
		\]
	\end{prop}
	
	\begin{proof}
		We prove the result by induction on $\min(\alpha,\beta)$. If $0 = \beta \le \alpha$, we have $\mathrm{Poly}_{(\ge a,d_2)}^{(d_1, d_2)}$ is the space of polynomials $(z^ag(z),z^{d_2})$ with $g(z)$ monic; this space is isomorphic to $\bA^\alpha$, so has motive $
		\Lb^\alpha$. Similarly, if $0 = \alpha < \beta$, then $\mathrm{Poly}_{(\ge d_1,b)}^{(d_1, d_2)}$ is the space of polynomials of the form $(z^{d_1},z^bg(z))$, where $g(z)$ is monic of degree $d_2-b$ and $g(0)\neq0$, i.e., the constant term of $g$ is non-zero. Thus, $\mathrm{Poly}_{(\ge d_1,b)}^{(d_1, d_2)}\simeq\bG_m\times\bA^{\beta-1}$ which has motive $(\Lb-1)\Lb^{\beta-1}$. 
		
		We may now assume $\alpha$ and $\beta$ are positive. Notice that via the map $(f_1,f_2)\mapsto (z^af_1,z^bf_2)$, the space $\mathrm{Poly}_{1}^{(d_1-a, d_2-b)}$ may be identified with the locally closed subscheme $\mathrm{Poly}_{(\ge a,b)}^{(d_1, d_2)}\cup \mathrm{Poly}_{(a,\geq b+1)}^{(d_1, d_2)}$; the union is taking place in the ambient weighted projective stack. Therefore
		\[
		\left\{\mathrm{Poly}_{(\ge a,b)}^{(d_1, d_2)}\right\} =
		\left\{\mathrm{Poly}_{1}^{(d_1-a, d_2-b)}\right\} 
		-\left\{\mathrm{Poly}_{(\geq b+1,a)}^{(d_2, d_1)}\right\}.
		\]
		If $d_2 - b = 1$ then we apply Proposition \ref{prop:poly_m} to the first term and the $\alpha = 0$ case above to the second term (note the roles of $d_2$ and $d_1$ have switched). 
		
		Otherwise, we apply the same relation to $\mathrm{Poly}^{(d_2, d_1)}_{(\geq b + 1, a)}$ to obtain 
		\[
		\left\{\mathrm{Poly}_{(\ge a,b)}^{(d_1, d_2)}\right\} =
		\left\{\mathrm{Poly}_{1}^{(d_1-a, d_2-b)}\right\} 
		- \left\{\mathrm{Poly}_1^{d_2 - b - 1, d_1 -a}\right\} + \left\{\mathrm{Poly}_{(\geq a+1,b+1)}^{(d_1, d_2)}\right\}.
		\]
		Then $\min\{\alpha, \beta\}$ has dropped by $1$ for the last term and the result follows by induction and Proposition \ref{prop:poly_m}. 
		
		The motive of the Poly-space $\mathrm{Poly}_{(a,b)}^{(d_1, d_2)}$ is immediate from the definitions.
	\end{proof}
	
	

	\subsection{Motives of height moduli with a fixed Kodaira type}  \label{subsect:Mot_Kod_One}
	
	Following Proposition~\ref{prop:single_addi}, we consider the case $l = 1$ (i.e. there is exactly one additive reduction fiber of specified Kodaira type and the rest of the bad reduction fibers are strictly multiplicative). We now compute a formula for $\cW^\gamma_{n,\Pb^1}(\Pc(4,6))$ for each Kodaira type $\gamma : (\nu(a_4), ~ \nu(a_6))$ as prescribed in Theorem~\ref{thm:corr_46}.

	\smallskip
	
	More generally, we consider $\cW^\gamma_{n,\bP^1}(\cP(\lambda_0,\lambda_1)) =\colon \cW^{\gamma}_n$ where $\gamma = (a,b), (\geq a, b)$, or $(a, \geq b)$ is a single vanishing condition. 
	
	\begin{prop} \label{prop:W0}
		The class of $\cW^\gamma_n \coloneqq \cW^\gamma_{n,\Pb^1}(\lambda_0,\lambda)$ for the vanishing condition $\gamma$ is
		\begin{equation*}\label{eqn:W_gen}
			\{\cW^\gamma_n\} = \{\Pb^1\} \cdot \{\Gb_m\} \cdot \left(\left\{\mathrm{Poly}_{\gamma}^{(\lambda_0n,\lambda_1n)}\right\}+\sum_{k=a}^{\lambda_0n-1}\left\{\mathrm{Poly}_{\gamma}^{(k,\lambda_1n)}\right\}+\sum_{l=b}^{\lambda_1n-1}\left\{\mathrm{Poly}_{\gamma}^{(\lambda_0n,l)}\right\}\right)
		\end{equation*}
	\end{prop}
	
	\begin{proof}
		Recall first that by Proposition~\ref{prop:single_addi} we have $\{\cW^\gamma_n\}=\{\Pb^1\} \cdot \{\cW_{n}^{\gamma}(0)\}$. So it suffices to compute $\{\cW_{n}^{\gamma}(0)\}$ which leads us to the stratification into Poly-spaces adopted from the proof of Theorem 1 in \cite[\S 4.1]{HP}.
		
		Consider $T \subset H^0(\Oc_{\Pb^1}(\lambda_0n))\oplus H^0(\Oc_{\Pb^1}(\lambda_1n))\setminus \{0\}$ a $\Gb_m$-equivariant open subset parameterizing pairs $(U,V)$ with exactly one common zero at $0 = [1:0] \in \bP^1$ and with specified vanishing condition $\gamma$, where $\Gb_m$ acts via $t*(U,V)=(t^{\lambda_0}U,t^{\lambda_1}V)$. Since $\gamma$ consists of a single tuple, in the notation from Definition \ref{Smu-definition}, we have $S_\mu$ is the trivial group. Then by Theorem \ref{thm:main-thm-of-sec4}, we have $\cW_{n,C}^\gamma=\cR_{n,C}^\gamma$. Using the moduli interpretation for $\cR_{n,C}^\gamma$ given after Definition \ref{def:moduli-interp-rnmu}, we see $\cW_n^\gamma(0)=[T/\bG_m]$.
		
		Now fix a chart $\A^1 \hookrightarrow \Pb^1$ with $x \mapsto [1:x]$, and call $0=[1:0]$ and $\infty=[0:1]$. It comes from a homogeneous chart of $\Pb^1$ by $[Y:X]$ with $x:=X/Y$ away from $\infty$. Then for any $(U,V) \in T$, $U$ and $V$ are homogeneous polynomials in $X$ and $Y$ with degrees $\lambda_0n$ and $\lambda_1n$ respectively. By plugging in $Y=1$, we obtain representations of $u(x)$ and $v(x)$ as polynomials in $x$ with degrees at most $\lambda_0n$ and $\lambda_1n$, respectively. For instance, $\deg u < \lambda_0n$ as a polynomial in $x$ if and only if $U(X,Y)$ is divisible by $Y^{\lambda_0n - k}$, i.e. vanishes at $\infty$. From now on, $\deg P$ means the degree of $P$ as a polynomial in $x$. We will say $(u,v) \in T$ if $u = U(x,1)$ and $v = V(x,1)$ for $(U,V) \in T$. 
		
		Denoting $\deg u:=k$ and $\deg v:=l$, then one of $k=\lambda_0n$ or $l=\lambda_1n$ if $(U,V) \in T$ (as they do not simultaneously vanish at $\infty$) and $u,v$ have a unique common root at $0$ with specified vanishing condition $\gamma$. Since there are many possible degrees for a pair $(u,v) \in T$, consider the locally closed subsets $T_{k,l}:=\{(u,v) \in T : \deg u=k, \; \deg v=l \}$. Notice that $T_{k-1,\lambda_1n} \subset \overline{T}_{k,\lambda_1n}$ as for any $(u,v) \in T_{k-1,\lambda_1n}$, $U(X,Y)$ has a description as $Y^{\lambda_0n-k+1}U'(X,Y)$ which is $U_{[1:0]}(X,Y)$ for a pencil polynomials $U_{[t_0:t_1]}(X,Y)=Y^{\lambda_0n-k}(t_1Y-t_0X)U'(X,Y)$ where $(U_{[1:t_1]},V) \in T_{k,\lambda_1n}$. Hence, we obtain the following stratification:
		
		\begin{align*}
			T&=T_{\lambda_0n,\lambda_1n} \sqcup \left(\bigsqcup_{k=a}^{\lambda_0n-1} T_{k,\lambda_1n}\right) \sqcup \left(\bigsqcup_{l=b}^{\lambda_1n-1} T_{\lambda_0n,l} \right)\\
			T&=\overline{T_{\lambda_0n,\lambda_1n}} \supsetneq \overline{T_{\lambda_0n-1,\lambda_1n}} \supsetneq \dotsb \supsetneq \overline{T_{a,\lambda_1n}}\\
			T&=\overline{T_{\lambda_0n,\lambda_1n}} \supsetneq \overline{T_{\lambda_0n,\lambda_1n-1}} \supsetneq \dotsb \supsetneq \overline{T_{\lambda_0n,b}}
		\end{align*}
		\[\overline{T_{\lambda_0n-k,\lambda_1n}} \cap \overline{T_{\lambda_0n,\lambda_1n-l}} =\varnothing ~~\; \forall k,l>0 \]
		Then,
		\begin{equation}
			\label{eqn:sum}
			\{T\} =\{T_{\lambda_0n,\lambda_1n}\}+\sum_{k=a}^{\lambda_0n-1}\{T_{k,\lambda_1n}\}+\sum_{l=b}^{\lambda_1n-1}\{T_{\lambda_0n,l}\}
		\end{equation}
		
		Lastly, note that for each term $T_{\alpha, \beta}$ in the above sum, we have a map $T_{\alpha, \beta} \to \bG_m^2$ which takes $(u,v)$ to its leading coefficients. This map is $\bG_m^2$-equivariant for the natural scaling action on $(u,v)$ and the fiber over $(1,1)$ is exactly $\mathrm{Poly}^{(\alpha,\beta)}_\gamma$. Thus taking $\bG_m$-quotients and applying \cite[Prop. 1.1 iii)]{Ekedahl} yields the result. \end{proof}
	
	\section{Enumerations of elliptic curves over rational function fields} \label{sec:enumeration}
	
	In this section we put together the results of the previous sections to compute the motives of various height moduli of elliptic curves over the function field of $\bP^1_{k}$ and in particular prove the explicit enumerations in Theorems \ref{thm:ell_curve_min_count} \& \ref{thm:ell_curve_count_2} \ref{thm:ell_curve_min_count_body}.
	
	\subsection{Motives of height moduli spaces and associated inertia stacks}
	
	We keep notation as in Section \ref{sec:zeta} By extracting the coefficients of the generating series in Theorem \ref{thm:ratl}, we obtain the following exact formulas for the classes of $\cW_n^{min}$ and $\cI \cW_n^{min}$. 
	
	\begin{thm}\label{thm:tbd9}
		The classes $\left\{ \cW_{n}^{\mathrm{min}} \right\}$ are given by the following formulas:
		\begin{align*}
			\left\{ \cW_{0}^{\mathrm{min}} \right\} &=  \{\bP^N\} \\
			\left\{ \cW_{1}^{\mathrm{min}} \right\} 
			&= \{\Pb^N\} ( \Lb^{|\vec{\lambda}|} - \Lb ) + \Lb^{N+1}\{\Pb^{|\vec{\lambda}|-N-2}\} \\
			\left\{ \cW_{n \ge 2}^{\mathrm{min}} \right\} 
			&= \Lb^{(n-2)|\vec{\lambda}| + N + 2} ( \Lb^{|\vec{\lambda}|-1} - 1 ) ( \Lb^{|\vec{\lambda}|-N-1}\{\Pb^N\} + \{\Pb^{|\vec{\lambda}|-N-2}\} )
		\end{align*}
	\end{thm}
	
	\begin{proof}  For $n = 0$ its clear. For $n = 1$, the coefficient of $t$ in the generating series is
		$$
		\Lb^{|\vec{\lambda}|}\{\Pb^N\} - \Lb\{\Pb^N\} + \Lb^{N+1}\{\Pb^{|\vec{\lambda}|-N-2}\} = \{\Pb^N\} ( \Lb^{|\vec{\lambda}|} - \Lb ) + \Lb^{N+1}\{\Pb^{|\vec{\lambda}|-N-2}\}. 
		$$
		Finally, for $n \geq 2$ the coefficient of $t^n$ is
		\begin{align*}
			\left\{ \cW_{n \ge 2}^{\mathrm{min}} \right\} &= 
			\Lb^{n|\vec{\lambda}|}\{\Pb^N\} + \Lb^{(n-1)|\vec{\lambda}|} (\Lb^{N+1}\{\Pb^{|\vec{\lambda}|-N-2}\} - \Lb \{\Pb^N\}) - \Lb^{(n-2)|\vec{\lambda}|} \Lb^{N+2}\{\Pb^{|\vec{\lambda}|-N-2}\} \\  
			&= \Lb^{(n-2)|\vec{\lambda}|} \left( \Lb^{2|\vec{\lambda}|} \{\Pb^N\} + \Lb^{|\vec{\lambda}|}  (\Lb^{N+1}\{\Pb^{|\vec{\lambda}|-N-2}\} - \Lb \{\Pb^N\}) - \Lb^{N+2} \{\Pb^{|\vec{\lambda}|-N-2}\} \right) \\ 
			&= \Lb^{(n-2)|\vec{\lambda}|} \left( (\Lb^{2|\vec{\lambda}|} - \Lb^{|\vec{\lambda}|+1}) \{\Pb^N\} + (\Lb^{|\vec{\lambda}|+N+1} - \Lb^{N+2})\{\Pb^{|\vec{\lambda}|-N-2}\} \right) \\ 
			&= \Lb^{(n-2)|\vec{\lambda}|} ( \Lb^{|\vec{\lambda}|-1} - 1 ) ( \Lb^{|\vec{\lambda}|+1}\{\Pb^N\} + \Lb^{N+2}\{\Pb^{|\vec{\lambda}|-N-2}\} ) \\
			&= \Lb^{(n-2)|\vec{\lambda}| + N + 2} ( \Lb^{|\vec{\lambda}|-1} - 1 ) ( \Lb^{|\vec{\lambda}|-N-1}\{\Pb^N\} + \{\Pb^{|\vec{\lambda}|-N-2}\} )
	\end{align*} \end{proof}

	\begin{thm}\label{thm:tbd10}
		The motive of $\left\{ \Ic\cW_{n}^{\mathrm{min}} \right\}$ is given by the following:
		
		\begin{align*}
			\left\{ \Ic\cW_{0}^{\mathrm{min}} \right\} &= \sum_{g \in \mu_\infty(k)} \Pb^{N_g} + \sum_{g \in \Bmu_k \setminus \mu_\infty(k)} R_{g,0} \\
			\left\{ \Ic\cW_{1}^{\mathrm{min}} \right\} &= 
			\sum_{g \in \mu_\infty(k)} \{\Pb^{N_g}\} ( \Lb^{|\vec{\lambda_{g}}|} - \Lb ) + \Lb^{{N_g}+1}\{\Pb^{|\vec{\lambda_{g}}|-{N_g}-2}\} + \sum_{g \in \Bmu_k \setminus \mu_{\infty}(k)} R_{g,1} \\
			\left\{ \Ic\cW_{n \ge 2}^{\mathrm{min}} \right\} &= 
			\sum_{g \in \mu_\infty(k)} \Lb^{(n-2)|\vec{\lambda_{g}}| + {N_g} + 2} ( \Lb^{|\vec{\lambda_{g}}|-1} - 1 ) ( \Lb^{|\vec{\lambda_{g}}|-{N_g}-1}\{\Pb^{N_g}\} + \{\Pb^{|\vec{\lambda_{g}}|-{N_g}-2}\} ) \\ & \quad \quad \quad \quad \quad \quad + \sum_{g \in \Bmu_k \setminus \mu_{\infty}(k)} R_{g,n}
		\end{align*}
		
		\noindent where $R_{g,n}$ is the coefficient of $t^n$ in $R_g(t)$ (see Theorem 8.9\ref{thm:ratlb}) and in particular $R_{g,n}$ is the motive of a stack with no $k$-points. 
	\end{thm}
	\begin{proof} The computations follow by applying the previous theorem to each term in the sum over $g \in \Bmu_k$ in Theorem \ref{thm:ratl}. 
	\end{proof} 
	
	\begin{rmk}
		Note that for each $g \in \mu_{\infty}(k)$, the corresponding term in the formula for $\cI \cW_n^{min}$ is simply $\{\cW_n^{min}((\lambda_i)_{i \in I_g})\}$, the motive of the moduli of minimal linear series of height $n$ on the weighted projective substack $\cP((\lambda_i)_{i \in I_g}) \subset \cP(\lambda_0, \ldots, \lambda_N)$.  
	\end{rmk}

	\subsection{Proof of Theorems~\ref{thm:motive_Rat}}
	
	We now apply Proposition \ref{prop:W0} to compute the motives for each Kodaira type. 
	
	As an example, we first compute the motive for additive reduction of type $\II$ at $j=0$ with the vanishing condition $\gamma = (\ge 1,1)$. The first equality below follows from Proposition \ref{prop:W0} and the second equality from Proposition \ref{prop:poly_m, Pc(4,6)}. 
	
	\begin{align*}
		&\{\cW_{n}^{(\ge 1,1)}(4,6)\}= \{\Pb^1\} \cdot \{\Gb_m\} \cdot \left(\left\{\mathrm{Poly}_{(\ge 1,1)}^{(4n,6n)}\right\}+\sum_{k=1}^{4n-1}\left\{\mathrm{Poly}_{(\ge 1,1)}^{(k,6n)}\right\}+\sum_{l=1}^{6n-1}\left\{\mathrm{Poly}_{(\ge 1,1)}^{(4n,l)}\right\}\right)\\
		&= (\Lb^2-1) \cdot \left[(\Lb - 1)\left( \frac{\Lb^{10n-2} + \Lb^{2n-1}}{(\Lb+1)} \right) + \sum_{k=1}^{4n-1} (\Lb - 1)  \left(\frac{\Lb^{k+6n-2} + \Lb^{-k+6n-1}}{(\Lb+1)}\right) + \Lb^{4n-1} \right]\\
		&\phantom{=}\text{ } + (\Lb^2-1) \cdot \left[ \sum_{l=2}^{4n}  (\Lb - 1)   \left(\frac{\Lb^{l+4n-2} - \Lb^{-l+4n}}{(\Lb+1)} \right) + \sum_{l=4n+1}^{6n-1} (\Lb - 1)  \left(\frac{\Lb^{l+4n-2} + \Lb^{l-4n-1}}{(\Lb+1)} \right) \right]\\
		&=  (\Lb^2 - 1)\Lb^{10n-2}
	\end{align*}

	In fact the same computation yields the following. Suppose $\lambda_1 > \lambda_0$ and let $\gamma = (a,b), (\geq a, b)$ or $(a, \geq b)$ be a vanishing condition. 
	
	\begin{prop}\label{prop:(a,b)counts} Suppose $\lambda_1n - b \geq \lambda_0n - a$, e.g. if $n \gg 0$. Then 
		\begin{align*}
			&\{\cW_n^{(\geq a, b)}(\lambda_0, \lambda_1)\} = \{\cW_n^{(a, \geq b)}(\lambda_0, \lambda_1)\} = (\bL^2 - 1)\bL^{(\lambda_0 + \lambda_1)n - a - b} \\
			&\{\cW_n^{(a,b)}(\lambda_0, \lambda_1)\} = (\bL^2 - 1)\left(\bL^{(\lambda_0 + \lambda_1)n - a - b} - \bL^{(\lambda_0 + \lambda_1)n - a - b - 1}\right)
		\end{align*}
	\end{prop}
	
	\begin{proof} The first formula follows as in the case above using Propositions \ref{prop:poly_m, Pc(4,6)} and \ref{prop:W0}. 
		
		For the second formula, by Proposition \ref{prop:W0} and the second part of Proposition \ref{prop:poly_m, Pc(4,6)}, we have
		\begin{align*}
			\{\cW^{(a,b)}_n\} &= \{\Pb^1\} \cdot \{\Gb_m\} \cdot \left(\left\{\mathrm{Poly}_{(a,b)}^{(\lambda_0n,\lambda_1n)}\right\}+\sum_{k=a}^{\lambda_0n-1}\left\{\mathrm{Poly}_{(a,b)}^{(k,\lambda_1n)}\right\}+\sum_{l=b}^{\lambda_1n-1}\left\{\mathrm{Poly}_{(a,b)}^{(\lambda_0n,l)}\right\}\right) \\ 
			&= (\bL^2 - 1)\left(\left\{\mathrm{Poly}_1^{(\lambda_0n, \lambda_1n)}\right\} + \sum_{k = a}^{\lambda_0n -1} \left\{\mathrm{Poly}_1^{(k - a, \lambda_1n - b)}\right\} + \sum_{l = b}^{\lambda_1n - 1} \left\{\mathrm{Poly}_1^{(\lambda_0n - a, l - b)}\right\}\right) \\
			&\quad - \cW_n^{(\geq a +1, b)} - \cW_n^{(a, \geq b + 1)}
		\end{align*}
		
		The sum of $\mathrm{Poly}$-spaces telescopes to $\bL^{(\lambda_0 + \lambda_1)n - a - b} + \bL^{(\lambda_0 + \lambda_1)n - a- b - 1}$. The result now follows by applying the first part to $\cW_n^{(\geq a + 1, b)}$ and $\cW_n^{(a, \geq b + 1)}$.

	\end{proof}

	The rest of the cases with different $\gamma$ now follow by applying the proposition. This completes the Proof of Theorems~\ref{thm:motive_Rat}.

	\subsection{Enumerating elliptic curves over \texorpdfstring{$\Fb_q(t)$}{PP\textasciicircum1} ordered by height of discriminant}
	
	Let $\Delta$ be the discriminant of a minimal elliptic fibration. Then the height of the discriminant over $\Fb_q$ is $ht(\Delta) \coloneqq q^{\deg \Delta} = q^{12n}$ where $n$ is Faltings heignt which by Proposition \ref{cor:unstable_height} agrees with the stacky height.

	\medskip
	
	We first determine the sharp enumerations for the number of minimal elliptic fibrations over $\Pb^1_{\Fb_q}$ with a specified additive reduction of Kodaira type $\Theta$ type and the rest of the bad reductions are at worst multiplicative.

	\begin{thm}\label{thm:ell_curve_count_2}
		Let $n \in \mathbb{Z}_{+}$ and $\text{char} (\Fb_q) \neq 2,3$. The function $\Nc(\Fb_q(t),~\Theta,~B)$, which counts the number of minimal elliptic curves over $\Pb^1_{\Fb_q}$ having a single specified additive reduction of Kodaira type $\Theta$ and at worst multiplicative reduction otherwise, ordered by the multiplicative height of the discriminant $ht(\Delta) = q^{12n} \le B$, satisfies:     
		\begingroup
		\allowdisplaybreaks
		\begin{align*}
			&\Nc(\Fb_q(t),~\II~\mathrm{with}~j=0,~B) = 2\frac{q^2 - 1}{ q^{10} - 1} \cdot q^8 \cdot \left( B^{5/6} - 1\right) \\
			&\Nc(\Fb_q(t),~\III~\mathrm{with}~j=1728,~B) = 2\frac{q^2 - 1}{ q^{10} - 1} \cdot q^7 \cdot \left( B^{5/6} - 1\right) \\
			&\Nc(\Fb_q(t),~\IV~\mathrm{with}~j=0,~B) = 2\frac{q^2 - 1}{ q^{10} - 1} \cdot q^6 \cdot \left( B^{5/6} - 1\right) \\
			&\Nc(\Fb_q(t),~\I_0^*~\mathrm{w.}~j \neq 0,1728~\mathrm{or}~ \I_{k > 0}^*~\mathrm{w.}~j=\infty,~B) = 2\frac{q^2 - 1}{ q^{10} - 1} \cdot (q^5 - q^4) \cdot \left( B^{5/6} - 1\right) \\
			&\Nc(\Fb_q(t),~\I_0^*~\mathrm{with}~j=0, 1728,~B) = 2\frac{q^2 - 1 }{q^{10}-1 } \cdot q^4 \cdot \left( B^{5/6} - 1\right) \\
			&\Nc(\Fb_q(t),~\IV^*~\mathrm{with}~j=0,~B) = 2\frac{q^2 - 1}{ q^{10} - 1} \cdot q^3 \cdot \left( B^{5/6} - 1\right) \\
			&\Nc(\Fb_q(t),~\III^*~\mathrm{with}~j=1728,~B) = 2\frac{q^2 - 1}{ q^{10} - 1} \cdot q^2 \cdot \left( B^{5/6} - 1\right) \\
			&\Nc(\Fb_q(t),~\II^*~\mathrm{with}~j=0,~B) = 2\frac{q^2 - 1}{ q^{10} - 1} \cdot q \cdot \left( B^{5/6} - 1\right) \\
		\end{align*}
		\endgroup
	\end{thm} 
	
	\begin{proof}[Proof of Theorems~\ref{thm:ell_curve_min_count}]

		We acquire the exact weighted point counts $\#_q\left(\cW_{n}^{\gamma}\right)$ over $\Fb_q$ from the motive formulas $\{\cW_{n}^{\gamma}\}$ by $\#_q: \Lb \mapsto q$. The exact number $|\cW_{n}^{\gamma}(\Fb_q)/\sim|$ of $\Fb_q$--isomorphism classes of $\Fb_q$--points of the moduli stack $\cW_{n}^{\gamma}$ over $\Fb_q$ with $\text{char} (\Fb_q) \neq 2,3$ is $|\cW_{n}^{\gamma}(\Fb_q)/\sim| = 2 \cdot \#_q\left(\cW_{n}^{\gamma}\right)$ where the factor of 2 comes from the hyperelliptic involution i.e., the generic $\mu_{\gcd(4,6)}$ stabilizer. By Theorem \ref{thm:sec7main} and Proposition \ref{prop:iso-HGammad-Rmun}, this is the count of elliptic surfaces of height $n$ over $\Pb^1_{\Fb_q}$ with the additive reduction controlled by vanishing conditions $\gamma$ corresponding to the given Kodaira fiber type $\Theta$. 
		
		\[ \Nc(\Fb_q(t),~\Theta,~B) = \sum \limits_{n=1}^{\left \lfloor \frac{log_q B}{12} \right \rfloor} |\cW_{n}^{\gamma}(\Fb_q)/\sim| = \sum \limits_{n=1}^{\left \lfloor \frac{log_q B}{12} \right \rfloor} 2 \cdot \#_q\left(\cW_{n}^{\gamma}\right) \]
		
		For $\gamma = (\geq a, b)$ or $(a, \geq b)$, we have by Proposition \ref{prop:(a,b)counts}
		$$
		\sum \limits_{n=1}^m \cdot \#_q\left(\cW_{n}^{\gamma}\right) = (q^2 - 1) \sum_{n = 1}^m q^{10n - a - b} = \frac{q^2 - 1}{q^{10}-1}q^{10 - a - b}(q^{10m}-1).
		$$
		Similarly, for $\gamma = (a,b)$, we have 
		$$
		\sum \limits_{n=1}^m \cdot \#_q\left(\cW_{n}^{\gamma}\right) = (q^2 - 1) \sum_{n = 1}^m q^{10n - a - b} - q^{10n - a - b - 1} = \frac{q^2 - 1}{q^{10}-1}(q-1)q^{10-a-b-1}(q^{10m}-1).
		$$
		Since $m = \left \lfloor \frac{log_q B}{12} \right \rfloor$, we have the that $q^{10m}-1$ grows as $B^{5/6} - 1$.
		
	\end{proof}
	
	\begin{rmk}
		Note that in each of the above cases, the elliptic curves over $\mathbb{F}_q(t)$ are non-isotrivial and so the automorphism group is $\mu_2$. Thus the weighted count $\Nc^w$ is simply $\frac{1}{2}\Nc$. Moreover, $ht(\Delta)$ is necessarily positive. 
	\end{rmk}

	Finally, we determine the sharp enumeration for the number of minimal elliptic fibrations over $\Pb^1_{\Fb_q}$. This is achieved by considering the totality of rational points on $\Me$ over $\Fb_q(t)$ via Theorems \ref{thm:tbd9} \& \ref{thm:tbd10}. In order to keep track of the primitive roots of unity contained in $\mathbb{F}_q$, we define the following auxillary function.
	
	\[
	\delta(x)\coloneqq
	\begin{cases}
		1 & \text{if } x \text{ divides } q-1,\\
		0 & \text{otherwise.}
	\end{cases}
	\]
	\begin{thm}\label{thm:ell_curve_min_count_body}
		Let $n \in \mathbb{Z}_{\ge 0}$ and $\text{char} (\Fb_q) > 3$. The counting function $\Nc^w(\Fb_q(t),~B)$ (resp. $\Nc(\Fb_q(t), ~B)$), which gives the weighted count (resp. unweighted count) of the number of minimal elliptic curves over $\Pb^1_{\Fb_q}$ ordered by the multiplicative height of the discriminant $ht(\Delta) = q^{12n} \le B$, is given by the following.   
		
		\begin{align*}
			\Nc^w(\Fb_q(t),~B) &= \left( \frac{q^{9} - 1}{q^{8} - q^{7}} \right) B^{5/6} - B^{1/6} \\\\
			\Nc(\Fb_q(t),~B) &= 2 \left( \frac{q^{9} - 1}{q^{8} - q^{7}} \right) B^{5/6} - 2 B^{1/6}   \\
			& + \delta(6) \cdot 4 \left( \frac{q^{5} - 1}{q^{5} - q^{4}} \right) B^{1/2} +  \delta(4) \cdot 2 \left( \frac{q^{3} - 1}{q^{3} - q^{2}} \right) B^{1/3}  \\
			& + \delta(6) \cdot 4 + \delta(4) \cdot 2 \\
		\end{align*}
		
	\end{thm}
	
	\begin{proof} 
		
		We subtract the rational points landing on $\Pc(2)$ since we do not want to count the generically singular $j = \infty$ isotrivial elliptic curves. Note that $\delta(2) = 1$ as $q-1$ is always even for any odd prime power $q$. We have
		$$
		\Nc^w(\Fb_q(t), B) = \sum_{n = 0}^{\left \lfloor \frac{log_q B}{12} \right \rfloor} \#_q \cW_{n, \bP^1}^{min}(4,6) - \sum_{n = 0}^{\left \lfloor \frac{log_q B}{12} \right \rfloor} \#_q\cW_{n, \bP^1}^{min}(2).
		$$
		Each of these sums can be computed via Theorem \ref{thm:tbd9}:
		\begin{align*}
			\sum_{n = 0}^{\left \lfloor \frac{log_q B}{12} \right \rfloor} \#_q \cW^{min}_{n, \bP^1}(4,6) &= \frac{q^9 - 1}{q^8 - q^7} \cdot q^{10\left \lfloor \frac{log_q B}{12} \right \rfloor} + 1 = \frac{q^9 - 1}{q^8 - q^7} \cdot B^{5/6} + 1 \\
			\sum_{n = 0}^{\left \lfloor \frac{log_q B}{12} \right \rfloor} \#_q \cW^{min}_{n, \bP^1}(2) &= q^{2\left \lfloor \frac{log_q B}{12} \right \rfloor} + 1 = B^{1/6} + 1.
		\end{align*}
		
		Similarly, for the unweighted count, we wish to compute
		$$
		\Nc(\Fb_q(t), B) = \sum_{n = 0}^{\left \lfloor \frac{log_q B}{12} \right \rfloor} \#_q \cI\cW^{min}_{n, \bP^1}(4,6) - \sum_{n = 0}^{\left \lfloor \frac{log_q B}{12} \right \rfloor} \#_q \cI\cW^{min}_{n, \bP^1}(2). 
		$$
		The latter term is simply $2$ times the weighted count since the automorphism group of any point of $\cP(2) = B\mu_2$ is $\mu_2$. Thus it sums to $2B^{1/6} + 2$. 
		
		For the first term, we use Theorems \ref{thm:tbd9} and \ref{thm:tbd10} to see that 
		\begin{align*}
			\#_q\cI \cW_{n,\bP^1}^{min} &= \sum_{g \in \mu_{\infty}(k)} \#_q\cW_{n,\bP^1}^{min}(I_g) \\ &= 2\#_q\cW_{n,\bP^1}(4,6) + \delta(6)\cdot 4 \#_q\cW_{n,\bP^1}(6) + \delta(4) \cdot 2 \#_q \cW_{n,\bP^1}(4).
		\end{align*}
		
		\noindent Summing over all $n$ reduces to the following computations using Theorem \ref{thm:tbd9}:
		\begin{align*}
			\sum_{n = 0}^{\left \lfloor \frac{log_q B}{12} \right \rfloor} \#_q\cW_{n,\bP^1}(4,6) &= \frac{q^9 - 1}{q^8 - q^7} \cdot B^{5/6} + 1 \\
			\sum_{n = 0}^{\left \lfloor \frac{log_q B}{12} \right \rfloor} \#_q\cW_{n,\bP^1}(6) &= \frac{q^5 - 1}{q^5 - q^4} \cdot B^{1/2} + 1 \\ 
			\sum_{n = 0}^{\left \lfloor \frac{log_q B}{12} \right \rfloor} \#_q\cW_{n,\bP^1}(4) &= \frac{q^3 - 1}{q^3 - q^2} \cdot B^{1/3} + 1
		\end{align*}
		from which the result follows.

	\end{proof}
	
	\begin{rmk}
		The lower order main term of order $B^{1/6}$ present in both the weighted and unweighted counts comes from subtracting the $\mu_2$ twist families of generically singular $j = \infty$ isotrivial elliptic curves. And the lower order main terms of order $B^{1/2}$ and $B^{1/3}$ in the unweighted count $\Nc(\Fb_q(t),~B)$ come from counting the $\mu_6$ and $\mu_4$ twist families of isotrivial elliptic curves having strictly additive bad reductions with extra automorphisms concentrated at the special $j$-invariants $j = 0, 1728$.
	\end{rmk}

	\section*{Acknowledgements}
	The authors are indebted to Yuri Manin for inspiring discussions, especially for his encouragement on understanding the link between curve counting and number theory. Warm thanks to Changho Han, Christian Haesemeyer, Jack Hall, Giovanni Inchiostro, Aaron Landesman, Johannes Schmitt, Jason Starr, Zijian Yao and Gjergji Zaimi as well for helpful discussions. We thank Tristan Phillips for many helpful comments on an earlier draft of the paper. Jun-Yong Park was supported by the ARC grant DP210103397 and the Max Planck Institute for Mathematics. MS was partially supported by a Discovery Grant from the National Science and Engineering Research Council of Canada and a Mathematics Faculty Research Chair.

	\bibliographystyle{alpha}
	\bibliography{main.bib}

\begin{thebibliography}{dFLNU07}

\bibitem[AB17]{AB}
Kenneth Ascher and Dori Bejleri.
\newblock Log canonical models of elliptic surfaces.
\newblock {\em Adv. Math.}, 320:210--243, 2017.

\bibitem[AB19]{AB2}
Kenneth Ascher and Dori Bejleri.
\newblock Moduli of fibered surface pairs from twisted stable maps.
\newblock {\em Math. Ann.}, 374(1-2):1007--1032, 2019.

\bibitem[AB21]{AB3}
Kenneth Ascher and Dori Bejleri.
\newblock Moduli of weighted stable elliptic surfaces and invariance of log
  plurigenera.
\newblock {\em Proc. Lond. Math. Soc. (3)}, 122(5):617--677, 2021.
\newblock With an appendix by Giovanni Inchiostro.

\bibitem[ABIP23]{ABIP}
Kenneth Ascher, Dori Bejleri, Giovanni Inchiostro, and Zsolt Patakfalvi.
\newblock Wall crossing for moduli of stable log pairs.
\newblock {\em Ann. of Math. (2)}, 198(2):825--866, 2023.

\bibitem[AH11]{AH}
Dan Abramovich and Brendan Hassett.
\newblock Stable varieties with a twist.
\newblock In {\em Classification of algebraic varieties}, EMS Ser. Congr. Rep.,
  pages 1--38. Eur. Math. Soc., Z\"{u}rich, 2011.

\bibitem[Alp13]{Alper}
Jarod Alper.
\newblock Good moduli spaces for {A}rtin stacks.
\newblock {\em Ann. Inst. Fourier (Grenoble)}, 63(6):2349--2402, 2013.

\bibitem[AOV08]{AOV}
Dan Abramovich, Martin Olsson, and Angelo Vistoli.
\newblock Tame stacks in positive characteristic.
\newblock {\em Ann. Inst. Fourier (Grenoble)}, 58(4):1057--1091, 2008.

\bibitem[AOV11]{AOV2}
Dan Abramovich, Martin Olsson, and Angelo Vistoli.
\newblock Twisted stable maps to tame {A}rtin stacks.
\newblock {\em J. Algebraic Geom.}, 20(3):399--477, 2011.

\bibitem[AV02]{av}
Dan Abramovich and Angelo Vistoli.
\newblock Compactifying the space of stable maps.
\newblock {\em J. Amer. Math. Soc.}, 15(1):27--75, 2002.

\bibitem[Beh06]{Behrens}
Mark Behrens.
\newblock A modular description of the {$K(2)$}-local sphere at the prime 3.
\newblock {\em Topology}, 45(2):343--402, 2006.

\bibitem[BGS20]{BGS}
L.~Beshaj, J.~Gutierrez, and T.~Shaska.
\newblock Weighted greatest common divisors and weighted heights.
\newblock {\em J. Number Theory}, 213:319--346, 2020.

\bibitem[BM90]{BM}
V.~V. Batyrev and Yu.~I. Manin.
\newblock Sur le nombre des points rationnels de hauteur born\'{e} des
  vari\'{e}t\'{e}s alg\'{e}briques.
\newblock {\em Math. Ann.}, 286(1-3):27--43, 1990.

\bibitem[BPS22]{BPS}
Oishee Banerjee, Jun-Yong Park, and Johannes Schmitt.
\newblock Étale cohomological stability of the moduli space of stable elliptic
  surfaces, 2022.
\newblock arXiv:2207.02496.

\bibitem[BST13]{BST}
Manjul Bhargava, Arul Shankar, and Jacob Tsimerman.
\newblock On the {D}avenport-{H}eilbronn theorems and second order terms.
\newblock {\em Invent. Math.}, 193(2):439--499, 2013.

\bibitem[CEF14]{CEF}
Thomas Church, Jordan~S. Ellenberg, and Benson Farb.
\newblock Representation stability in cohomology and asymptotics for families
  of varieties over finite fields.
\newblock In {\em Algebraic topology: applications and new directions}, volume
  620 of {\em Contemp. Math.}, pages 1--54. Amer. Math. Soc., Providence, RI,
  2014.

\bibitem[CJ23]{CJ}
Peter~J. Cho and Keunyoung Jeong.
\newblock On the distribution of analytic ranks of elliptic curves.
\newblock {\em Math. Z.}, 305(3):Paper No. 42, 20, 2023.

\bibitem[Con07]{Conrad}
Brian Conrad.
\newblock Arithmetic moduli of generalized elliptic curves.
\newblock {\em J. Inst. Math. Jussieu}, 6(2):209--278, 2007.

\bibitem[CS23]{CS}
John~E. Cremona and Mohammad Sadek.
\newblock Local and global densities for {W}eierstrass models of elliptic
  curves.
\newblock {\em Math. Res. Lett.}, 30(2):413--461, 2023.

\bibitem[Dar21]{Darda}
R.~Darda.
\newblock Rational points of bounded height on weighted projective stacks,
  2021.

\bibitem[DD13]{Dokchitser}
Tim Dokchitser and Vladimir Dokchitser.
\newblock A remark on {T}ate's algorithm and {K}odaira types.
\newblock {\em Acta Arith.}, 160(1):95--100, 2013.

\bibitem[dFLNU07]{Fernex}
Tommaso de~Fernex, Ernesto Lupercio, Thomas Nevins, and Bernardo Uribe.
\newblock Stringy {C}hern classes of singular varieties.
\newblock {\em Adv. Math.}, 208(2):597--621, 2007.

\bibitem[DH69]{DH}
H.~Davenport and H.~Heilbronn.
\newblock On the density of discriminants of cubic fields.
\newblock {\em Bull. London Math. Soc.}, 1:345--348, 1969.

\bibitem[DH71]{DH2}
H.~Davenport and H.~Heilbronn.
\newblock On the density of discriminants of cubic fields. {II}.
\newblock {\em Proc. Roy. Soc. London Ser. A}, 322(1551):405--420, 1971.

\bibitem[dJ02]{dJ2}
A.~J. de~Jong.
\newblock Counting elliptic surfaces over finite fields.
\newblock {\em Mosc. Math. J.}, 2(2):281--311, 2002.
\newblock Dedicated to Yuri I. Manin on the occasion of his 65th birthday.

\bibitem[DW88]{DW}
Boris Datskovsky and David~J. Wright.
\newblock Density of discriminants of cubic extensions.
\newblock {\em J. Reine Angew. Math.}, 386:116--138, 1988.

\bibitem[DY22]{DY}
R.~Darda and T.~Yasuda.
\newblock The {B}atyrev-{M}anin conjecture for {DM} stacks, 2022.
\newblock arXiv:2207.03645.

\bibitem[Eke09]{Ekedahl}
Torsten Ekedahl.
\newblock The {G}rothendieck group of algebraic stacks, 2009.
\newblock arXiv:0903.3143.

\bibitem[ESZB23]{ESZB}
Jordan~S. Ellenberg, Matthew Satriano, and David Zureick-Brown.
\newblock Heights on stacks and a generalized {B}atyrev-{M}anin-{M}alle
  conjecture.
\newblock {\em Forum Math. Sigma}, 11:Paper No. e14, 54, 2023.

\bibitem[Fed14]{Fedorchuk}
Maksym Fedorchuk.
\newblock Moduli spaces of hyperelliptic curves with {A} and {D} singularities.
\newblock {\em Math. Z.}, 276(1-2):299--328, 2014.

\bibitem[FMN10]{FMN}
Barbara Fantechi, Etienne Mann, and Fabio Nironi.
\newblock Smooth toric {D}eligne-{M}umford stacks.
\newblock {\em J. Reine Angew. Math.}, 648:201--244, 2010.

\bibitem[FMT89]{FMT}
Jens Franke, Yuri~I. Manin, and Yuri Tschinkel.
\newblock Rational points of bounded height on {F}ano varieties.
\newblock {\em Invent. Math.}, 95(2):421--435, 1989.

\bibitem[FW16]{FW}
Benson Farb and Jesse Wolfson.
\newblock Topology and arithmetic of resultants, {I}.
\newblock {\em New York J. Math.}, 22:801--821, 2016.

\bibitem[FW19]{FW2}
Benson Farb and Jesse Wolfson.
\newblock Corrigendum to ``{T}opology and arithmetic of resultants, {I}'',
  {N}ew {Y}ork {J}. {M}ath. 22 (2016), 801--821 [ {MR}3548124].
\newblock {\em New York J. Math.}, 25:195--197, 2019.

\bibitem[GGW21]{GGW}
Claudio G\'{o}mez-Gonz\'{a}les and Jesse Wolfson.
\newblock Problems in arithmetic topology.
\newblock {\em Res. Math. Sci.}, 8(2):Paper No. 23, 14, 2021.

\bibitem[GK89]{grkar}
Gert-Martin Greuel and Ulrich Karras.
\newblock Families of varieties with prescribed singularities.
\newblock {\em Compositio Math.}, 69(1):83--110, 1989.

\bibitem[Gro58]{Grothendieck}
Alexander Grothendieck.
\newblock Torsion homologique et sections rationnelles.
\newblock {\em S\'eminaire Claude Chevalley}, 3, 1958.
\newblock talk:5.

\bibitem[GW20]{Wedhorn}
Ulrich G\"{o}rtz and Torsten Wedhorn.
\newblock {\em Algebraic geometry {I}. {S}chemes---with examples and
  exercises}.
\newblock Springer Studium Mathematik---Master. Springer Spektrum, Wiesbaden,
  [2020] \copyright 2020.
\newblock Second edition [of 2675155].

\bibitem[Hal14]{cohbase}
Jack Hall.
\newblock Cohomology and base change for algebraic stacks.
\newblock {\em Math. Z.}, 278(1-2):401--429, 2014.

\bibitem[Her91]{Herfurtner}
Stephan Herfurtner.
\newblock Elliptic surfaces with four singular fibres.
\newblock {\em Math. Ann.}, 291(2):319--342, 1991.

\bibitem[HM17]{HMe}
Michael~A. Hill and Lennart Meier.
\newblock The {$C_2$}-spectrum {${\rm Tmf}_1(3)$} and its invertible modules.
\newblock {\em Algebr. Geom. Topol.}, 17(4):1953--2011, 2017.

\bibitem[HP19]{HP}
Changho Han and Jun-Yong Park.
\newblock Arithmetic of the moduli of semistable elliptic surfaces.
\newblock {\em Math. Ann.}, 375(3-4):1745--1760, 2019.

\bibitem[HP23]{HP2}
Changho Han and Jun-Yong Park.
\newblock Enumerating odd-degree hyperelliptic curves and abelian surfaces over
  {${\Bbb {P}}^1$}.
\newblock {\em Math. Z.}, 304(1):Paper No. 5, 32, 2023.

\bibitem[Inc20]{Inchiostro}
Giovanni Inchiostro.
\newblock Moduli of {W}eierstrass fibrations with marked section.
\newblock {\em Adv. Math.}, 375:107374, 57, 2020.

\bibitem[KM85]{KM}
Nicholas~M. Katz and Barry Mazur.
\newblock {\em Arithmetic moduli of elliptic curves}, volume 108 of {\em Annals
  of Mathematics Studies}.
\newblock Princeton University Press, Princeton, NJ, 1985.

\bibitem[Kod63]{Kodaira}
K.~Kodaira.
\newblock On compact analytic surfaces. {II}, {III}.
\newblock {\em Ann. of Math. (2) 77 (1963), 563--626; ibid.}, 78:1--40, 1963.

\bibitem[Liu02]{Liu}
Qing Liu.
\newblock {\em Algebraic geometry and arithmetic curves}, volume~6 of {\em
  Oxford Graduate Texts in Mathematics}.
\newblock Oxford University Press, Oxford, 2002.
\newblock Translated from the French by Reinie Ern\'{e}, Oxford Science
  Publications.

\bibitem[LP19]{LP}
Yank\i$~$ Lekili and Alexander Polishchuk.
\newblock A modular compactification of {$\mathcal M_{1,n}$} from
  {$A_\infty$}-structures.
\newblock {\em J. Reine Angew. Math.}, 755:151--189, 2019.

\bibitem[Mei22]{Meier}
Lennart Meier.
\newblock Additive decompositions for rings of modular forms.
\newblock {\em Doc. Math.}, 27:427--488, 2022.

\bibitem[Mir89]{Miranda2}
Rick Miranda.
\newblock {\em The basic theory of elliptic surfaces}.
\newblock Dottorato di Ricerca in Matematica. [Doctorate in Mathematical
  Research]. ETS Editrice, Pisa, 1989.

\bibitem[N\'64]{Neron}
Andr\'{e} N\'{e}ron.
\newblock Mod\`eles minimaux des vari\'{e}t\'{e}s ab\'{e}liennes sur les corps
  locaux et globaux.
\newblock {\em Inst. Hautes \'{E}tudes Sci. Publ. Math.}, 21:128, 1964.

\bibitem[Nil13]{Niles}
Andrew Niles.
\newblock Moduli of elliptic curves via twisted stable maps.
\newblock {\em Algebra Number Theory}, 7(9):2141--2202, 2013.

\bibitem[Ols16]{Olsson2}
Martin Olsson.
\newblock {\em Algebraic spaces and stacks}, volume~62 of {\em American
  Mathematical Society Colloquium Publications}.
\newblock American Mathematical Society, Providence, RI, 2016.

\bibitem[Phi22a]{Tristan2}
Tristan Phillips.
\newblock Average analytic ranks of elliptic curves over number fields, 2022.
\newblock arXiv:2205.09527.

\bibitem[Phi22b]{Tristan3}
Tristan Phillips.
\newblock Points of bounded height on weighted projective spaces over global
  function fields, 2022.
\newblock arXiv:2205.15877.

\bibitem[Phi22c]{Tristan}
Tristan Phillips.
\newblock Rational {P}oints of {B}ounded {H}eight on {S}ome {G}enus {Z}ero
  {M}odular {C}urves over number fields, 2022.
\newblock arXiv:2201.10624.

\bibitem[PS21a]{JJ}
Jun-Yong Park and Johannes Schmitt.
\newblock Arithmetic geometry of the moduli stack of {W}eierstrass fibrations
  over $\mathbb{P}^1$, 2021.
\newblock arXiv:2107.12231.

\bibitem[PS21b]{PS}
Jun-Yong Park and Hunter Spink.
\newblock Motive of the moduli stack of rational curves on a weighted
  projective stack.
\newblock {\em Res. Math. Sci.}, 8(1):Paper No. 1, 8, 2021.

\bibitem[Rob01]{Roberts}
David~P. Roberts.
\newblock Density of cubic field discriminants.
\newblock {\em Math. Comp.}, 70(236):1699--1705, 2001.

\bibitem[Ryd15]{rydh}
David Rydh.
\newblock Noetherian approximation of algebraic spaces and stacks.
\newblock {\em J. Algebra}, 422:105--147, 2015.

\bibitem[Ser58]{Serre}
Jean-Pierre Serre.
\newblock Espaces fibr\'es alg\'ebriques.
\newblock {\em S\'eminaire Claude Chevalley}, 3, 1958.
\newblock talk:1.

\bibitem[SF00]{Stankova}
Zvezdelina~E. Stankova-Frenkel.
\newblock Moduli of trigonal curves.
\newblock {\em J. Algebraic Geom.}, 9(4):607--662, 2000.

\bibitem[Sil09]{Silverman}
Joseph~H. Silverman.
\newblock {\em The arithmetic of elliptic curves}, volume 106 of {\em Graduate
  Texts in Mathematics}.
\newblock Springer, Dordrecht, second edition, 2009.

\bibitem[Smy11a]{Smyth}
David~Ishii Smyth.
\newblock Modular compactifications of the space of pointed elliptic curves
  {I}.
\newblock {\em Compos. Math.}, 147(3):877--913, 2011.

\bibitem[Smy11b]{Smyth2}
David~Ishii Smyth.
\newblock Modular compactifications of the space of pointed elliptic curves
  {II}.
\newblock {\em Compos. Math.}, 147(6):1843--1884, 2011.

\bibitem[SS10]{SS}
Matthias Sch\"{u}tt and Tetsuji Shioda.
\newblock Elliptic surfaces.
\newblock In {\em Algebraic geometry in {E}ast {A}sia---{S}eoul 2008},
  volume~60 of {\em Adv. Stud. Pure Math.}, pages 51--160. Math. Soc. Japan,
  2010.

\bibitem[SS22]{SST}
S.~Salami and T.~Shaska.
\newblock Local and global heights on weighted projective varieties and
  {V}ojta's conjecture, 2022.
\newblock arXiv:2204.01624.

\bibitem[{Sta}18]{Stacks}
The {Stacks Project Authors}.
\newblock ~.
\newblock \url{https://stacks.math.columbia.edu}, 2018.

\bibitem[Sto12]{Stojanoska}
Vesna Stojanoska.
\newblock Duality for topological modular forms.
\newblock {\em Doc. Math.}, 17:271--311, 2012.

\bibitem[TT13]{TT}
Takashi Taniguchi and Frank Thorne.
\newblock Secondary terms in counting functions for cubic fields.
\newblock {\em Duke Math. J.}, 162(13):2451--2508, 2013.

\bibitem[VE10]{VE}
Akshay Venkatesh and Jordan~S. Ellenberg.
\newblock Statistics of number fields and function fields.
\newblock In {\em Proceedings of the {I}nternational {C}ongress of
  {M}athematicians. {V}olume {II}}, pages 383--402. Hindustan Book Agency, New
  Delhi, 2010.

\bibitem[VW15]{VW}
Ravi Vakil and Melanie~Matchett Wood.
\newblock Discriminants in the {G}rothendieck ring.
\newblock {\em Duke Math. J.}, 164(6):1139--1185, 2015.

\end{thebibliography}
	
	\vspace{+9pt}

	\noindent Dori Bejleri \enspace -- \enspace \texttt{bejleri@math.harvard.edu} \\
	\textsc{Department of Mathematics, Harvard University, USA} \\

	\noindent Jun--Yong Park \enspace -- \enspace \texttt{june.park@unimelb.edu.au} \\
	\textsc{School of Mathematics and Statistics, University of Melbourne, Australia} \\

	\noindent Matthew Satriano \enspace -- \enspace \texttt{msatriano@uwaterloo.ca} \\
	\textsc{Department of Pure Mathematics, University of Waterloo, Canada} \\

\end{document}